\documentclass{amsart}[11pt]

\usepackage[utf8]{inputenc}
\usepackage[russian, english]{babel}
 \usepackage[T2A]{fontenc}

\usepackage{amsmath,amsthm,amsfonts,amssymb,amscd,verbatim,delarray,url}%epigraph}
\usepackage[arrow,matrix]{xy}

\usepackage{tikz}

\usetikzlibrary{decorations.markings}
\usetikzlibrary{decorations.pathreplacing}
\usetikzlibrary{cd}
\tikzset{arrow/.style={
        decoration={markings,
            mark= at position #1 with {\arrow{stealth}},
        },
        postaction={decorate}
    }
}

\tikzset{bigarrow/.style={
        decoration={markings,
            mark= at position #1 with {\arrow[scale=1.5]{stealth}},
        },
        postaction={decorate}
    }
}

\tikzset{reversearrow/.style={
        decoration={markings,
            mark= at position #1 with {\arrow{stealth reversed}},
        },
        postaction={decorate}
    }
}

\tikzset{reversebigarrow/.style={
        decoration={markings,
            mark= at position #1 with {\arrow[scale=1.5]{stealth reversed}},
        },
        postaction={decorate}
    }
}

\DeclareMathOperator{\PSL}{PSL}
\DeclareMathOperator{\SL}{SL}

\DeclareMathOperator{\GL}{GL}
\DeclareMathOperator{\tr}{tr}
\DeclareMathOperator{\rk}{rk}

\DeclareMathOperator{\id}{id}

\DeclareMathOperator{\SR}{SR}
\DeclareMathOperator{\ASR}{ASR}
\DeclareMathOperator{\SO}{SO}
\DeclareMathOperator{\sr}{sr}
\DeclareMathOperator{\asr}{asr}

\def\ad{\operatorname{ad}}
\def\sic{\operatorname{sc}}

\begin{document}

\theoremstyle{plain}
\newtheorem{theorem}{Theorem}[section]
\newtheorem{conj}[theorem]{Conjecture}
\newtheorem{corollary}[theorem]{Corollary}
\newtheorem{prop}[theorem]{Proposition}
\newtheorem{Lemma}[theorem]{Lemma}
\newtheorem{lemma}[theorem]{Lemma}
\newtheorem{maintheorem}{Theorem}
\newtheorem{smaintheorem}{Theorem}
\renewcommand{\themaintheorem}{\Alph{maintheorem}}
\renewcommand{\thesmaintheorem}{\Alph{maintheorem}}

\theoremstyle{definition}
\newtheorem{definition}[theorem]{Definition}
\newtheorem{quest}[theorem]{Question}
\newtheorem{remark}[theorem]{Remark}
\newtheorem{example}[theorem]{Example}
\newtheorem*{answer}{Answer}
\newtheorem{notation}[theorem]{Notation}
\newtheorem*{prin}{Panda Principle}
\newtheorem{question}[theorem]{Question}

\newcommand{\thmref}[1]{Theorem~\ref{#1}}
\newcommand{\secref}[1]{Section~\ref{#1}}
\newcommand{\subsecref}[1]{Subsection~\ref{#1}}
\newcommand{\lemref}[1]{Lemma~\ref{#1}}
\newcommand{\corolref}[1]{Corollary~\ref{#1}}
\newcommand{\exampref}[1]{Example~\ref{#1}}
\newcommand{\remarkref}[1]{Remark~\ref{#1}}
\newcommand{\defnref}[1]{Definition~\ref{#1}}
\newcommand{\propref}[1]{Proposition~\ref{#1}}
\newtheorem{Oldi}{Theorem}

\newcommand{\BZ}{\mathbb{Z}}
\newcommand{\Int}{\mathbb{Z}}
\newcommand{\BC}{\mathbb{C}}
\newcommand{\BN}{\mathbb{N}}
\newcommand{\BP}{\mathbb{P}}
\newcommand{\BF}{\mathbb{F}}
\newcommand{\BA}{\mathbb{A}}
\newcommand{\BQ}{\mathbb{Q}}
\newcommand{\BM}{\mathbb{M}}
\newcommand{\BX}{\mathbb{X}}
\newcommand{\BR}{\mathbb{R}}
\newcommand{\Rat}{\mathbb{Q}}

\newcommand{\ep}{\epsilon}
\newcommand{\al}{\alpha}
\newcommand{\be}{\beta}
\newcommand{\ga}{\gamma}
\newcommand{\de}{\delta}
\newcommand{\la}{\lambda}
\newcommand{\om}{\omega}
\newcommand{\vp}{\varphi}
\newcommand{\st}{\sigma}
\newcommand{\eps}{\varepsilon}

\newcommand{\vareps}{\varepsilon}

\newcommand{\G}{\Gamma}

\newcommand{\ov}{\overline}

\newcommand{\fc}{\frac}

\newcommand\cR{\mathcal R}
\newcommand\cP{\mathcal P}
\newcommand\cF{\mathcal F}
\newcommand\cA{\mathcal A}
\newcommand\cG{\mathcal G}
\newcommand{\Fg}{\mathfrak g}
\newcommand{\Fsl}{\mathfrak sl}

\newcommand\rA{\textsc A}
\newcommand\rB{\textsc B}
\newcommand\rC{\textsc C}
\newcommand\rD{\textsc D}
\newcommand\rE{\textsc E}
\newcommand\rF{\textsc F}
\newcommand\rG{\textsc G}

\newtheorem{observation}[theorem]{Observation}
\newcommand{\Mod}[1]{\ (\mathrm{mod}\ #1)}

%\begin{document}

\newcommand{\lra}{\longrightarrow}
\newcommand{\idd}{\mathop{\rm{id}}\nolimits}
\newcommand{\R}{{\mathbb R}}
\newcommand{\Co}{{\mathbb C}}
\newcommand{\komp}{{\mathbb C}}
\newcommand{\Z}{{\mathbb Z}}
\newcommand{\N}{{\mathbb N}}
\newcommand{\Q}{{\mathbb Q}}

\newcommand{\e}{{\mathfrak{e}}}

\newcommand{\Om}{{\overline{\Omega}_{a_1a_2\cdots a_s}}}
\newcommand{\Oms}{{\Omega}_{a_1a_2\cdots a_s}}
\newcommand{\Fe}{{F}_{a_1a_2\cdots a_s}}
\newcommand{\Imm}{\operatorname{Im}}
\newcommand{\Omsr}{{\mathfrak{T}_r\Omega}_{a_1a_2\cdots a_s}}
\newcommand{\Omr}{{\mathfrak{T}_r\overline{\Omega}_{a_1a_2\cdots
a_s}}}

\newcommand{\Symp}{Sp}
%\title{Maps}

%\author{}

\def\pt{{\mathfrak{p}_{\infty}}}

\newcommand{\din}{{\rm d}_{\rm in}}
\newcommand{\E}{\mathcal{E}}
\newcommand{\Ma}{\mathrm{M}}
\newcommand{\Sp}{\mathrm{Sp}}
\newcommand{\Spin}{\mathrm{Spin}}
\newcommand{\Res}{\mathrm{Res}}
\newcommand{\PSp}{\mathrm{PSp}}
\renewcommand{\G}{\mathrm{G}}
\renewcommand{\P}{\mathbb{P}}
\newcommand{\len}{\mathrm{len}}

\def\<{\langle}
\def\>{\rangle}
\def\tilde{\widetilde}
\def\o{\tilde{\omega}}
\def\phi{\varphi}
\def \g{\mathcal G}
\def \I{\mathcal I}
\def \W{\mathcal W}
\def \T{\mathcal T}
\def \H{\mathcal H}
\def \X{\mathcal X}
\def \G{\mathcal G}
\def \U{\mathcal U}
\def \K{\mathcal K}
\def \P{\mathcal P}
\def \E{\mathcal E}
\def \V{\mathcal V}
\def \B{\mathcal B}
\def \EE{\tilde{\mathcal{E}}_\Gamma }
\def\e{\epsilon}
\def\m{\mathfrak{m}}
\def\n{\mathfrak{n}}
\def\a{\mathfrak{a}}
\def\h{\mathfrak{h}}
\def\d{\mathfrak{d}}
\def\b{\mathfrak{b}}
\def\Or{\mathcal O}
\def\ot{\mathfrak{ o}}
\def\id{\operatorname{id}}
\def\tr{\operatorname{tr}}
\def\Tr{\operatorname{Tr}}
\def\Aut{\operatorname{Aut}}
\def\Max{\operatorname{Max}}
\def\w{\tilde{w}}
\def\ga{\tilde{\Gamma}_w^i}
\def\GF#1{{\mathbb F}_{\!#1}}

%\newcounter{smaintheorem}
%\setcounter{mycounter}{42}
%\setcounter{smaintheorem}{D}
%%%%%%%%%%%%%%%%%%%%%%%%%%%%%
%\smaintheoremno

\title[Bounded generation of Chevalley groups] %%% and Kac---Moody groups ]
{Bounded generation and commutator width of Chevalley
%%% and Kac--Moody
groups: function case}

\author[%%Capdeboscq,
Kunyavski\u\i, Plotkin, Vavilov] {%Inna Capdeboscq,
Boris Kunyavski\u\i , Eugene Plotkin, Nikolai Vavilov}

%\dedicatory{To Boris Isaakovich Plotkin on the occasion of his 90th birthday}

%% \address{Capdeboscq: Department of Mathematics,
%% University of Warwick, Coventry, CV4 7AL, UK}
%% \email{I.Capdeboscq@warwick.ac.uk }

\address{Kunyavski\u\i : Department of
Mathematics, Bar-Ilan University, 5290002 Ramat Gan, ISRAEL}
\email{kunyav@macs.biu.ac.il}

\address{Plotkin: Department of Mathematics, Bar-Ilan University, 5290002 Ramat Gan, ISRAEL}
 \email{plotkin@macs.biu.ac.il}

 \address {Vavilov: Department of Mathematics and Computer Science,
St.Petersburg State University}\email{nikolai-vavilov@yandex.ru}

\begin{abstract}
We prove that Chevalley groups over polynomial rings $\mathbb F_q[t]$
and over Laurent polynomial $\mathbb F_q[t,t^{-1}]$ rings,
where $\mathbb F_q$ is a finite field,  are  boundedly elementarily
generated. Using this we produce explicit bounds of the commutator width
of these groups. Under some additional assumptions, we prove similar
results for other classes of Chevalley groups over Dedekind rings of
arithmetic rings in positive characteristic.
As a corollary, we produce explicit estimates for the commutator width of  affine Kac--Moody
groups defined over finite fields. The paper contains also a broader discussion of the bounded generation problem for groups of Lie type, some
applications and a list of unsolved problems in the field.

\thanks{Research of Boris Kunyavski\u\i \ and Eugene Plotkin was supported by the ISF grants 1623/16 and 1994/20.
Nikolai Vavilov thanks the ``Basis'' Foundation grant
N.~20-7-1-27-1 ``Higher symbols in algebraic K-theory''.}
\end{abstract}

\maketitle

%\begin{keyword} word map \sep polynomial map \sep matrix group \sep
%matrix algebra \MSC 20G15 \sep 20G20
%\end{keyword}

\section*{Introduction}

In the present paper, we consider Chevalley groups $G=G(\Phi,R)$ and
their elementary subgroups $E(\Phi,R)$ over various classes of rings,
primarily over Dedekind rings of arithmetic type. In some special cases these
groups are closely related to various Kac--Moody type groups, and we
can derive some non-trivial corollaries in this situation.
\par
Primarily, we are interested in the classical problems of estimating the
width of $G(\Phi,R)$ and $E(\Phi,R)$ with respect to the two following
paradigmatic generating sets.
\par\smallskip
$\bullet$ The elementary generators $x_{\alpha}(\xi)$, $\alpha\in\Phi$, $\xi\in R$.
We say that a group $G$ is {\bf boundedly elementarily generated}
if it has finite width $w_E(G)$ with respect to elementary generators.
\par\smallskip
$\bullet$ Commutators $[x,y]=xyx^{-1}y^{-1}$, where $x,y\in G$. In
this case we say that $G$ has {\bf finite commutator width} $w_C(G)$.
\par\smallskip
However, in the proofs we work also with other related generating sets,
such as elements in the unipotent radicals of various parabolic subgroups,
which are closely related but better behaved with respect to stability maps.
\par
For Chevalley groups of rank $\ge 2$ bounded generation in terms of
elementary generators, and bounded generation in terms of commutators
are essentially equivalent. Indeed, in this case the Chevalley commutator
formula readily implies that every elementary generator can be presented
as a product of a bounded number of commutators. Conversely, a very deep
result by Alexei Stepanov and others (see in particular \cite{SiSt}, \cite{SV2}, and
in final form \cite{Step}) implies that given any commutative ring $R$, every commutator in $E(\Phi,R)$ is a
product of not more than $L$ elementary generators, with the bound
$L=L(\Phi)$ depending on $\Phi$ alone. But of course the actual
estimates of $w_E(G)$ and $w_C(G)$ can be very different.
\par
Both problems have attracted considerable attention over the last 40 years
or so. Roughly, the situation is as follows. Bounded elementary generation
always holds with obvious bounds for 0-dimensional rings and usually fails
for rings of dimension $\ge 2$. But for 1-dimensional rings it is problematic.
\par
Thus, from the existence of arbitrary long division chains in Euclidean
algorithm it follows that $\SL(2,\Int)$ and $\SL(2,\mathbb F_q[t])$ are not
boundedly elementarily generated. But this could be attributed to the
exceptional behaviour of rank 1 groups.
Much more surprisingly, Wilberd van
der Kallen \cite{vdK}  established that bounded generation fails even for
$\SL(3,\Co[t])$, a group of Lie rank 2 over a Euclidean ring!
\par
An emblematic example of 1-dimensional rings are Dedekind
rings of arithmetic type $R={\mathcal O}_S$, for which bounded elementary
generation of $G(\Phi,R)$ is intrinsically related to the positive solution
of the congruence subgroup problem in that group. This connection was first
noted by Vladimir Platonov and Andrei Rapinchuk, see \cite{PR}, \cite{Ra1},
\cite{Ra2}.

For the {\bf number case\/} the situation is well understood, even for rank 1 groups.
After the initial breakthrough by Douglas Carter and Gordon Keller
\cite{CaKe1}, \cite{CaKe2}, later expanded by Oleg Tavgen \cite{Ta} and many others,
we now know bounded generation with excellent bounds depending on
the type of $\Phi$ and the class number of $R$ for {\it all Chevalley
groups\/} of rank $\ge 2$.
%% and/or other arithmetic invariants of $R$.
Apart from the rings $R={\mathcal O}_S$, $|S|=1$, with finite
multiplicative group, similar results are even available for $\SL(2,R)$,
see a detailed survey in Section~2.
\par
However, the {\bf function case} turned out to be much more
recalcitrant, and is up to now not solved, apart from some important
but isolated results, such as the works by Clifford Queen \cite{Qu}
and Bogdan Nica \cite{Nic}, which treat the group $\SL(2,R)$ over {\it some\/}
arithmetic function rings with infinite multiplicative groups, and the
groups $\SL(n,\mathbb F_q[t])$, $n\ge 3$, respectively\footnote{The difference
between the number and function cases is subtle enough and may be overlooked when
approaching from outside. We quote from page 2 of the memoir \cite{EJZK}: `\dots $G$ is known
to be boundedly generated by $X$ only in a few cases, namely, when $R$ is a finite
extension of $\mathbb Z$ or $F[t]$, with $F$ a finite field.'
In a sense, the present paper, along with \cite{Nic}, can be viewed
as a first step along the long and painful road to justification of this brave claim.}.

\par
Here we expand these results to all Chevalley groups, obtaining
explicit bounds.
The first major new result of the present paper establishes bounded elementary
generation for all Chevalley groups of rank at least 2 over the
most classical, and in a sense the most difficult example, polynomial
rings $\mathbb F_{q}[t]$ with coefficients in finite fields\footnote
{After the preliminary version of the present paper has been finished, there appeared a preprint of Alexander Trost
\cite{Tr} where the statement of our Theorem A was established
for the ring of integers $R$ of an arbitrary global function field
$K$, with a bound of the form $L(d,q)\cdot|\Phi|$, where the factor
$L$ depends on $q$ and of the degree $d$ of $K$. His method
is similar to Morris' approach in \cite{Mor}.}.

\begin{maintheorem}\label{mtheorem2a}
Let $G(\Phi,R)$ be a simply connected
Chevalley group of type $\Phi$, $\rk(\Phi)\ge 2$ over $R=\mathbb F_{q}[t]$.
Then the width of $G(\Phi,R)$ with respect
to elementary generators is bounded by a constant not depending on $q$.
\end{maintheorem}

The proof of this result constitutes about half of the paper.
{\it Some\/} bound in the bounded generation for all Chevalley
groups can be easily derived from the case of rank two systems
by a version of the usual Tavgen's trick \cite[Theorem~1]{Ta},
described in \cite{VSS} and \cite{SSV}.
\par\smallskip
$\bullet$ For $\rA_2$ bounded generation
of $\SL(3,\mathbb F_q[t])$ is precisely the main result of Nica
\cite{Nic}.
\par\smallskip
$\bullet$  A large part of the present paper is the analysis of the
most difficult case of $\Sp(4,\mathbb F_{q}[t])$, which is the Chevalley
group of type $\rC_2$. Again, we take the proof in Tavgen's
paper \cite[Section~4]{Ta}, as a prototype. But there is a substantial difference,
since now we have to verify some arithmetic properties that are
well known in the number case, but for which we could not
find any reference in the function case.
\par\smallskip
$\bullet$ Luckily, we do not have to imitate Tavgen's proof
\cite[Section~5]{Ta},
for the remaining case of the Chevalley group of type $\rG_2$.
Instead of a difficult direct calculation, we show that this case
can be derived from the case of $\rA_2$ by the usual stability
arguments.
\par\smallskip
For $\SL(n,R)$ there is a realistic bound of the width in
elementary generators, in terms of stability conditions, taking
into account the fact that for Dedekind rings $\sr(R)=1.5$.
The aforementioned proof of Theorem A gives us occasion to return to
the stability arguments for all Chevalley groups, and obtain
bounds which are substantially better than the ones that could be
obtained via Tavgen's trick.
\par\smallskip
Alternatively, Theorem \ref{mtheorem2a} can be restated in the following
equivalent form. The difference is that in this case the
computations of many authors, subsumed and expanded by
Andrei Smolensky \cite{Sm}, allow one to produce very
reasonable bounds, usually at most 6, 7 or 8 commutators.

\begin{maintheorem}\label{mtheorem2}
Let $G(\Phi,R)$ be a simply connected
Chevalley group of type $\Phi$, $\rk(\Phi)\ge 2$ over $R=\mathbb F_{q}[t]$,
Then $G(\Phi,R)$ is of finite commutator width.
\end{maintheorem}

\begin{remark}\label{prestavl} The commutator width of a Chevalley group depends on a representation. For example,
$w_C(\PSL(2,\mathbb F_q))=1$ for all $q$ while $w_C(\SL(2,\mathbb F_2))=w_C(\SL(2,\mathbb F_3))=2$, see \cite{Th}. So, if the representation is not stated explicitly,  under $w_C(G(\Phi,R))$ we always mean maximum, i.e., the commutator width of the simply connected group.
\end{remark}

In fact, for applications to Kac--Moody groups we do not need
the full force of Theorem A. We only need a similar result for
the equally classical but {\it much easier\/} example of
{\it Laurent\/} polynomial rings $\mathbb F_{q}[t,t^{-1}]$ with coefficients
in finite fields.
\par
For Chevalley groups over such rings bounded generation can
be derived from Theorem A. Yet, the bounds thus obtained
will not be the best possible ones. However, the
multiplicative group of the ring $R=\mathbb F_{q}[t,t^{-1}]$ is
{\it infinite\/}. This means that alternatively bounded generation
can be derived --- with much better bounds! --- from the result by
Clifford Queen \cite{Qu}. Let us state the most spectacular
finiteness result in terms of unitriangular factors obtained along
this route.

\begin{maintheorem} \label{mtheorem3a}
Let $R=\mathcal O_S$ be the ring of $S$-integers of $K$, a
function field of one variable over $\mathbb F_q$
with $S$ containing at least two places. Assume that at least
one of the following holds:
\par\smallskip
$\bullet$ either at least one of these places has degree one,
\par\smallskip
$\bullet$ or the class number of $R$, as a Dedekind domain,
is prime to $q-1$.
\par\smallskip\noindent
Then any simply connected Chevalley group $G=G(\Phi ,R)$
admits the following decompositions
$$ G=UU^-UU^-U=U^-UU^-UU^-. $$
\end{maintheorem}

Such a sharp bound was quite unexpected for us.
%% since the bound on the number of factors is the same
%% as for rings with $\sr(R)=1$.
In particular, Chevalley groups over such {\it arithmetic\/} rings
have the same commutator width as Chevalley groups over rings
of stable rank 1, see \cite{Sm}.
%%% This lead us to the observation on the hierarchy
%%% of higher versions of stable rank

In particular, we can now give the same bounds for
affine Kac--Moody groups. %--- кстати, так ли это, не увеличивает
%ли центральное расширение коммутаторную длину на 1?

%\begin{Oldi}\label{Oldi}
%\end{Oldi}

\begin{maintheorem}\label{mtheorem}
The commutator width of an affine elementary untwisted Kac--Moody
group $\widetilde E_{sc}(A,\mathbb F_q)$ over a finite field $\mathbb F_q$
%($q\ne 2,3 (?))$.
is $\le L'$,
where
\par\smallskip
$\bullet$ $L'=5$ for $\Phi=\rF_4$ and $\Phi=\rA_l$, $l=2k+1$, $k=0,1,\dots ${\rm;}
\par\smallskip
%$\bullet$ $L=6$ for $\Phi=\rA_l$, $l\geq 2${\rm;}
%\par\smallskip
$\bullet$ $L'=6$ for $\Phi=\rA_l$, $l=2k$, $k=1,2,\dots $,  $\Phi=\rB_l, \rC_l, \rD_l$, for $l\ge 3$ or
$\Phi=\rE_7, \rE_8$, or, finally, $\Phi=\rC_2, \rG_2$ under the
additional assumption that $1$ is the sum of two units in $R$
{\rm(}which is automatically the case provided $q\neq2${\rm);}
\par\smallskip
$\bullet$ $L'=7$ for $\Phi=\rE_6$.

%Let $\widetilde E(A,\mathbb F_q)$ be an affine elementary untwisted  Kac--Moody
%group over a finite field $\mathbb F_q$.
%($q\ne 2,3 (?))$.
%Then the commutator width of $\widetilde E(A,\mathbb F_q)$ is $\le L$,
%where
%\par\smallskip
%$\bullet$ $L=5$ for $\Phi=\rF_4$ and $\Phi=A_l$, $l=2k+1$, $k=0,1,...${\rm;}
%\par\smallskip
%$\bullet$ $L=6$ for $\Phi=\rA_l$, $l\geq 2${\rm;}
%\par\smallskip
%$\bullet$ $L=6$ for $\Phi=A_l$, $l=2k$, $k=1,2...$,  $\Phi=\rB_l, \rC_l, \rD_l$, for $l\ge 3$ or
%$\Phi=\rE_7, \rE_8$, or, finally, $\Phi=\rC_2, \rG_2$ under the
%additional assumption that $1$ is the sum of two units in $R$
%{\rm(}which is automatically the case, provided $q\neq2${\rm);}
%\par\smallskip
%$\bullet$ $L=7$ for $\Phi=\rE_6$.

\end{maintheorem}

%%%%%%%%%%%Pozdnee ispol`zovat` pri dokazatel`stve
%\begin{theorem}\label{mtheorem1} Let $R=\mathbb F_q[t]$, let $G_{\Phi}$ be a Chevalley group over $R$, and let
%$G_A$ be the corresponding affine Kac--Moody group over $\mathbb F_q$. If $G(\Phi,R)$
%is of finite commutator width, then so is $G_A(\mathbb F_q)$.
%\end{theorem}

The paper is organised as follows. In Section \ref{sec:intro} we recall the
necessary notation and preliminaries and in Section \ref{sec:art}  provide
background and historical survey. The next four sections
constitute the technical core of the paper. Namely, in Section \ref{sec:scheme}
we sketch the scheme of the proof of Theorem A,
of which Theorem B is an immediate corollary, and
reduce its proof to the rank 2 groups. This reduction is
a variation of Tavgen's rank reduction trick,
a further slight improvement of the rank reduction results
in \cite{VSS}, \cite{SSV}.  In Section \ref{secG2}  we revisit surjective stability for
$K_1$ modeled on Chevalley groups, with explicit bounds,
and, in particular, reduce the case of the group ${\rm G}_2(R)$
to the known case of $\SL(3,R)$. In Section \ref{secC2} we prove Theorem A for the
group $\Sp(4,R)$, which is the most exciting case of all,
and requires rather difficult algebraic and arithmetic considerations. Section \ref{rank3} contains
an alternative argument based on reducing to rank 3 groups and separate consideration
of the types $\rB_3$ and $\rC_3$. Incidentally, this gives estimates with better constants.
After that, in Section \ref{Queen} we
develop an alternative approach to bounded elementary generation, based on Queen's result,
that gives sharper bounds for some classes of rings $R$
with infinite multiplicative groups, including Laurent
polynomial rings, thus proving Theorem C. The next
section is devoted to applications. In Subsection \ref{sec:KM} we discuss
applications to Kac--Moody groups over finite fields and
prove Theorem D, and in Subsection \ref{la} we obtain some applications
of bounded generation in model theory. Finally, in Section \ref{final}
we present some relevant concluding remarks and open problems.

%%%%%%%%%%%%%%%%%%%%%%%%%%%%%%

\section{Notation and preliminaries}\label{sec:intro}

In this section we briefly recall the notation that will be used
throughout the paper. For more details on Chevalley groups
over rings see \cite{NV91} or \cite{VP}, where one can
find many further references.
\subsection{Chevalley groups}
\def\wP{\mathcal P}
\def\wQ{\mathcal Q}
Let $\Phi$ be a reduced irreducible root system of rank $\ge 2$,
and $W=W(\Phi)$ be its Weyl group. Choose an order on $\Phi$
and let $\Phi^+$, $\Phi^{-}$ and
$\Pi=\big\{\alpha_1,\ldots,\alpha_l\big\}$ be the corresponding
sets of positive, negative and fundamental roots, respectively.
Further, we consider a lattice $\wP$ intermediate
between the root lattice $\wQ(\Phi)$ and the weight
lattice $\wP(\Phi)$. Finally, let $R$ be a commutative ring with 1,
with the multiplicative group $R^*$.
\par
These data determine the Chevalley group $G=G_{\wP}(\Phi,R)$,
of type $(\Phi,\wP)$ over $R$. It is usually constructed as the
group of $R$-points of the Chevalley--Demazure
group scheme $G_{\wP}(\Phi,\text{$-$})$ of type $(\Phi,\wP)$.
In the case
$\wP=\wP(\Phi)$ the group $G$ is called simply connected and
is denoted by $G_{\sic}(\Phi,R)$. In another extreme case
$\wP=\wQ(\Phi)$ the group $G$ is called adjoint and
is denoted by $G_{\ad}(\Phi,R)$. Many results do not depend
on the lattice $\wP$ and hold for all groups of a given
type $\Phi$. In all such cases, or when $\wP$ is determined by
the context, we omit any reference to $\wP$ in the notation
and denote by $G(\Phi,R)$ {\it any} Chevalley group of type
$\Phi$ over $R$. Usually, we assume that $G(\Phi,R)$
is simply connected.
\par
In what follows, we also fix a split maximal torus $T=T(\Phi,R)$ in $G=G(\Phi,R)$ and identify $\Phi$ with $\Phi(G,T)$. This choice
uniquely determines the unipotent root subgroups, $X_{\alpha}$,
$\alpha\in\Phi$, in $G$, elementary with respect to $T$. As usual,
we fix maps $x_{\alpha}\colon R\mapsto X_{\alpha}$, so that
$X_{\alpha}=\{x_{\alpha}(\xi)\mid\xi\in R\}$, and require that these parametrisations are interrelated by the Chevalley commutator formula with integer coefficients, see \cite{Carter},
\cite{Steinberg}. The above unipotent elements
$x_{\alpha}(\xi)$, where $\alpha\in\Phi$, $\xi\in R$,
elementary with respect to $T(\Phi,R)$, are also called
[elementary] unipotent root elements or, for short, simply
root unipotents.
\par
Further,
$$ E(\Phi,R)=\big\langle x_\alpha(\xi),\ \alpha\in\Phi,\ \xi\in R\big\rangle $$
\noindent
denotes the {\it absolute\/} elementary subgroup of $G(\Phi,R)$,
spanned by all elementary root unipotents, or, what is the
same, by all [elementary] root subgroups $X_{\alpha}$,
$\alpha\in\Phi$.
\par
Since we are interested in the bounded generation, we also
consider the {\it subset\/} $E^L(\Phi,R)$, consisting of
products of $\le L$ root unipotents. Since $E^L(\Phi,R)$
contains all generators of $E(\Phi,R)$, it is not a subgroup
of $E(\Phi,R)$, unless $E^L(\Phi,R)=E(\Phi,R)$.

%%%%%%%%%%%%%%%%%%%%%%%%%%%%%
\subsection{Root elements}
Further, let $\alpha\in\Phi$ and  $\eps\in R^{*}$.
As usual, we set
$$ w_{\alpha}(\eps)=
x_{\alpha}(\eps)x_{-\alpha}(-\eps^{-1})x_{\alpha}(\eps),\qquad
h_{\alpha}(\eps)=w_{\alpha}(\eps)w_{\alpha}(1)^{-1}. $$
\noindent
The elements $h_{\alpha}(\eps)$ are called semisimple root elements.
\par
By definition, $h_{\alpha}(\eps)$ is a product of {\it six\/}
elementary unipotents --- well, actually if you look inside, {\it five\/}
of them. However, it is classically known that
$h_{\alpha}(\eps)$ is a product of {\it four\/} elementary
unipotents\footnote{On the other hand, since $B\cap U^-=e$,
it is never a product of {\it three\/} such unipotents,
unless $\eps=1$.}. To somewhat improve some of the ulterior
bounds we need a still more precise form of this classical observation, asserting that the first/last of these four factors can be chosen either
lower, or upper, with an {\it arbitrary\/} invertible parameter. After
that the remaining three factors are uniquely determined.
\par
The following fact is obvious, but we could not find an explicit
reference.

\begin{lemma} \label{diagonal}
Let $R$ be any commutative ring. Then for any $\eps,\eta\in R^*$
the matrix $h_{\alpha}(\eps)$ can be represented
as the product of the form
\begin{multline*}
h_{\alpha}(\eps)=
x_{-\alpha}(\eta)x_{\alpha}\big(-\eta^{-1}(1 -\eps^{-1})\big)
x_{-\alpha}(-\eps\eta)x_{\alpha}\big(\eps^{-1}\eta^{-1}(1 - \eps^{-1})\big)=\\
x_{-\alpha}\big(\eps^{-1}\eta^{-1}(1 - \eps^{-1})\big)
x_{\alpha}(-\eps\eta)
x_{-\alpha}\big(-\eta^{-1}(1 -\eps^{-1})\big)x_{\alpha}(\eta)=\\
x_{\alpha}(\eps\eta(1-\eps))x_{-\alpha}\big(-\eps^{-1}\eta^{-1}\big)
x_{\alpha}\big(-\eta(1-\eps)\big)x_{-\alpha}\big(\eta^{-1}\big)=\\
x_{\alpha}(\eta^{-1})x_{-\alpha}\big(-\eta(1-\eps)\big)
x_{\alpha}\big(-\eps^{-1}\eta^{-1}\big)x_{-\alpha}(\eps\eta(1-\eps)\big)
\end{multline*}
\end{lemma}

\begin{proof}
Verify one of these formulae by a direct calculation in $\SL(2,R)$, then
transpose, invert and transpose-invert it.
\end{proof}

\begin{corollary} \label{cor:diagonal}
Let $R$ be any commutative ring. Then for any $\eps,\lambda\in R^*$
the matrix $h_{\alpha}(\eps)$ can be transformed to  $h_{\alpha}(\lambda)$
by $4$ elementary moves.
\end{corollary}

\begin{proof}
By Lemma \ref{diagonal}, $h_\alpha(\eps\lambda^{-1})=h_\alpha(\eps) (h_\alpha(\lambda))^{-1}$ can be transformed to 1
by 4 elementary moves, whence the statement.
\end{proof}

Next, let $N=N(\Phi,R)$ be the algebraic normaliser of the
torus $T=T(\Phi,R)$, i.~e.\ the subgroup, generated by $T=T(\Phi,R)$
and all elements $w_{\alpha}(1)$, $\alpha\in\Phi$. The factor-group
$N/T$ is canonically isomorphic to the Weyl group $W$, and for each
$w\in W$ we fix its preimage $n_{w}\in N$. Clearly, such a
preimage can be taken in $E(\Phi,R)$. Indeed, for a root
reflection $w_{\alpha}$ one can take $w_{\alpha}(1)\in E(\Phi,R)$
as its preimage, any element $w$ of the Weyl group can
be expressed as a product of root reflections.
\par
In particular, we get the following classical result, which is
crucial in reduction to smaller ranks.

\begin{lemma} \label{lem:root}
The elementary Chevalley group\/ $E(\Phi,R)$ is generated by
unipotent root elements\/ $x_{\alpha}(\xi)$, $\alpha\in\pm\Pi$, $\xi\in R$,
corresponding to the fundamental and negative fundamental roots.
\end{lemma}

Further, let $B=B(\Phi,R)$ and $B^-=B^-(\Phi,R)$ be a pair of
opposite Borel subgroups containing $T=T(\Phi,R)$, standard
with respect to the given order. Recall that $B$ and $B^-$
are semidirect products $B=T\rightthreetimes U$ and
$B^-=T\rightthreetimes U^-$, of the torus $T$ and their
unipotent radicals
$$
%\begin{equation}
\begin{aligned}
U&=U(\Phi,R)=
\big\langle x_\alpha(\xi),\ \alpha\in\Phi^+,\ \xi\in R\big\rangle, \\
\noalign{\vskip 4pt}
U^-&=U^-(\Phi,R)=
\big\langle x_\alpha(\xi),\ \alpha\in\Phi^-,\ \xi\in R\big\rangle. \\
\end{aligned}
%\end{equation}
$$
\noindent
Here, as usual, for a subset $X$ of a group $G$ one denotes by
$\langle X\rangle$ the subgroup in $G$ generated by $X$.
Semidirect product decomposition of $B$ amounts to saying that
$B=TU=UT$, and at that $U\trianglelefteq B$ and $T\cap U=1$.
Similar facts hold with $B$ and $U$ replaced by $B^-$ and $U^-$.
Sometimes, to speak of both subgroups $U$ and $U^-$ simultaneously,
we denote $U=U(\Phi,R)$ by $U^+=U^+(\Phi,R)$.
\subsection{Levi decomposition}
The main role in the reduction to smaller ranks is played by Levi
decomposition for elementary parabolic subgroups.
In general, one can associate a subgroup $E(S)=E(S,R)$
to any closed set $S\subseteq\Phi$. Recall that a subset
$S$ of $\Phi$ is called {\it closed\/}, if for any two roots
$\alpha,\beta\in S$ the fact that $\alpha+\beta\in\Phi$, implies that already
$\alpha+\beta\in S$. Now, we define $E(S)=E(S,R)$ as the subgroup generated by
all elementary root unipotent subgroups $X_{\alpha}$, $\alpha\in S$:
$$ E(S,R)=\big\langle x_{\alpha}(\xi),\ \alpha\in S,
\ \xi\in R\big\rangle. $$
\noindent
In this notation, $U$ and $U^{-}$ coincide with
$E(\Phi^{+},R)$ and $E(\Phi^{-},R)$, respectively. The groups
$E(S,R)$ are particularly important in the case where $S$ is a {\it
special\/} (= {\it unipotent\/}) set of roots; in other words, where
$S\cap(-S)=\varnothing$. In this case $E(S,R)$ coincides with the
{\it product\/} of root subgroups $X_{\alpha}$, $\alpha\in S$, in
some/any fixed order.
\par
Let again $S\subseteq\Phi$ be a closed set of roots. Then $S$ can be
decomposed into a disjoint union of its {\it reductive\/} (= {\it
symmetric\/}) part $S^{r}$, consisting of those $\alpha\in S$, for which
$-\alpha\in S$, and its {\it unipotent\/} part $S^{u}$, consisting of
those $\alpha\in S$, for which $-\alpha\not\in S$. The set $S^{r}$ is a
closed root subsystem, whereas the set $S^{u}$ is special. Moreover,
$S^{u}$ is an {\it ideal\/} of $S$, in other words, if $\alpha\in S$,
$\beta\in S^{u}$ and $\alpha+\beta\in\Phi$, then $\alpha+\beta\in S^{u}$. {\it Levi
decomposition\/} asserts that the group $E(S,R)$ decomposes into
semidirect product $E(S,R)=E(S^r,R)\rightthreetimes E(S^u,R)$ of its
{\it Levi subgroup\/} $E(S^{r},R)$ and its {\it unipotent
radical\/}~$E(S^{u},R)$.
%%%% \subsection{Elementary parabolic subgroups}
\par
Especially important is the case of elementary subgroups
corresponding to the maximal parabolic subschemes.
 Denote by
$m_k(\alpha)$ the coefficient of $\alpha_k$ in the expansion of $\alpha$
with respect to the fundamental roots:
$$ \alpha=\sum_{k=1}^l m_k(\alpha)\alpha_k. $$ %,\quad 1\le k\le l. $$
\par
Now, fix an $r=1,\ldots,l$ --- in fact, in the reduction to smaller
rank it suffices to employ only terminal parabolic subgroups,
even only the ones corresponding to the first and the last
fundamental roots, $r=1,l$. Denote by
$$ S=S_r=\big\{\alpha\in\Phi \colon  m_r(\alpha)\geq 0\big\} $$
\noindent
the $r$-th standard parabolic subset in $\Phi$. As usual,
the reductive part $\Delta=\Delta_r$ and the special part
$\Sigma=\Sigma_r$ of the set $S=S_r$ are defined as
$$ \Delta=\big\{\alpha\in\Phi \colon  m_r(\alpha) = 0\big\},\quad
\Sigma=\big\{\alpha\in\Phi \colon  m_r(\alpha) > 0\big\}. $$
\noindent
The opposite parabolic subset and its special part are defined
similarly
$$ S^-=S^-_r=\big\{\alpha\in\Phi \colon  m_r(\alpha)\leq 0\big\},\quad
\Sigma^-=\big\{\alpha\in\Phi \colon  m_r(\alpha)<0\big\}. $$
\noindent
Obviously, the reductive part $S^-_r$ equals $\Delta$.
\par
Denote by $P_r$ the {\it elementary\/} [maximal] parabolic
subgroup of the elementary group $E(\Phi,R)$. By definition,
$$ P_r=E(S_r,R)=\big\langle x_\alpha(\xi),\ \alpha\in S_r,
\ \xi\in R \big\rangle. $$
\noindent
Now Levi decomposition asserts that the group $P_r$ can be represented
as the semidirect product
$$ P_r=L_r\rightthreetimes U_r=E(\Delta,R)\rightthreetimes E(\Sigma,R) $$
\noindent
of the elementary Levi subgroup $L_r=E(\Delta,R)$ and the unipotent
radical $U_r=E(\Sigma,R)$. Recall that
$$ L_r=E(\Delta,R)=\big\langle x_\alpha(\xi),\quad \alpha\in\Delta,
\quad \xi\in R \big\rangle, $$
\noindent
whereas
$$ U_r=E(\Sigma,R)=
\big\langle x_\alpha(\xi),\ \alpha\in\Sigma,\ \xi\in R\big\rangle. $$
\noindent
A similar decomposition holds for the opposite parabolic subgroup
$P_r^-$, whereby the Levi subgroup is the same as for $P_r$,
but the unipotent radical $U_r$ is replaced by the opposite unipotent
radical $U_r^-=E(-\Sigma,R)$.
\par
As a matter of fact, we use Levi decomposition in the following
form. It will be convenient to slightly change the notation
and write $U(\Sigma,R)=E(\Sigma,R)$ and $U^-(\Sigma,R)=E(-\Sigma,R)$.

\begin{lemma} \label{lem:semidirect}
The group\/ $\big\langle U^{\sigma}(\Delta,R),U^\rho(\Sigma,R)\big\rangle$,
where\/ $\sigma,\rho=\pm 1$, is the semidirect product of its
normal subgroup\/ $U^\rho(\Sigma,R)$ and the complementary subgroup\/
$U^{\sigma}(\Delta,R)$.
\end{lemma}

In other words, it is asserted here that the subgroups
$U^{\pm}(\Delta,R)$ normalise each of the groups
$U^{\pm}(\Sigma,R)$, so that, in particular, one has the following
four equalities for products
$$ U^{\pm}(\Delta,R)U^{\pm}(\Sigma,R)=U^{\pm}(\Sigma,R)U^{\pm}(\Delta,R), $$
\noindent
and, furthermore, the following four obvious equalities for
intersections hold:
$$ U^{\pm}(\Delta,R)\cap U^{\pm}(\Sigma,R)=1. $$
\par
In particular, one has the following decompositions:
$$ U(\Phi,R)=U(\Delta,R)\rightthreetimes U(\Sigma,R),
\quad
U^-(\Phi,R)=U^-(\Delta,R)\rightthreetimes U^-(\Sigma,R). $$

%\subsection{Theorem \ref{mtheorem2} }

%Our aim is Theorem \ref{mtheorem2}

%\begin{theorem}[C]\label{mtheorem2}
 %Let $R=\mathbb F_q[t]$, and let $G(\Phi,R)$, $rk (\Phi)>1$ be a Chevalley group over $R$.
%Then $G(\Phi,R)$ is of finite commutator width.
%\end{theorem}

%%%%%%%%%%%%%%%%%%%%%%%%%%%%%%%

\section{Bounded generation. State of art} \label{sec:art}

To put the results of the present paper in context, here we briefly
recall what is known concerning the finite elementary width and
the finite commutator width of Chevalley groups over rings. This
will give us occasion to explain some basic ideas behind our proof.

\subsection{Length and width}

Let $G$ be a group and $X$ be a set of its generators. Usually
one considers symmetric sets, for which $X^{-1}=X$.
\begin{itemize}\itemsep=0pt
\item
The {\bf length\/} $l_X(g)$ of an element $g\in G$ with
respect to $X$ is the minimal $k$ such that $g$ can be expressed
as the product $g=x_1\ldots x_k$, $x_i\in X$.
\item
The {\bf width} $w_X(G)$ of $G$ with respect to $X$ is
the supremum of $l_X(g)$ over all $g\in G$.
\end{itemize}
We say that a group $G$ has {\bf bounded generation} with
respect to $X$ if the width $w_X(G)$ is finite.\footnote{In the
literature, expressions {\bf bounded generation} and {\bf finite
width} are used in several related but significantly different
contexts.
Oftentimes one calls a group $G$ boundedly generated if it has
bounded generation with respect to the powers of some
{\it finite\/} generating set. This amounts to the group being
a finite product of several {\it cyclic\/} subgroups. In many
situations it is
equally meaningful to consider groups which are finite products
of {\it abelian\/} subgroups. Finally, one calls families $G_i$
of finitely presented groups boundedly generated if they can be presented in such a way that the sums of the number of
generators and relations of $G_i$ are uniformly bounded.}
In the case when
$w_X(G)=\infty$, one says that $G$ does not have bounded word
length with respect to $X$.
\par

The problem of calculating or estimating $w_X(G)$ has
attracted a lot of attention, especially when $G$ is one of
the classical-like groups over skew-fields.
\par
There are {\it hundreds\/} of papers which address this
problem in the case when $G$ is a classical group such as
$\SL(n,R)$ or $\Sp(2l,R)$ or its large subgroup, whereas $X$
is a natural set of its generators.
\par\smallskip
$\bullet$ Classically, over fields and
other small-dimensional rings one would think of
elementary transvections, all transvections, or Eichler--Siegel--Dickson (ESD)-transvections,
reflections, pseudo-reflections, or other small-dimensional
transformations.
\par\smallskip
$\bullet$ Other common choice would be a class of
matrices determined by their eigenvalues such as the set of
all involutions, a non-central conjugacy class, or the set of
all commutators.
\par\smallskip
$\bullet$ More exotic choices include matrices distinct
from the identity matrix in one column, symmetric matrices, etc.
\par\smallskip
In many classical cases exact values or at least sharp estimates of
$w_X(G)$ are available. Sometimes there are even more precise
results, explicitly calculating the length of {\it individual\/} elements
in terms of certain geometric invariants such as, e.g., the dimension
of its residual space, or the like.
\par
More generally, oftentimes one considers any subset $X\subseteq G$ and looks at the width $w_X\big(\langle X\rangle\big)$. For instance,
one calls  the width of the commutator subgroup $[G,G]$  with respect to the set of all commutators the {\bf commutator width} of $G$ itself,
regardless of whether the group $G$ is perfect. This is a prototypical
example of what is called the {\bf word length} problems, when one
tries to calculate or estimate the width of the verbal subgroup
of $G$ with respect to a word $w$ with respect to the set of values
of $w$ in $G$.

\subsection{Elementary width and commutator width}

In the present paper we focus on the much less studied case,
where $G=G(\Phi,R)$ is a Chevalley group or its elementary subgroup
$E(\Phi,R)$ over a commutative ring $R$, and on the closely
related case of Kac--Moody groups.  In this setting exact
calculations of $w_X(G)$ with respect to most of the generating
sets are usually beyond reach.
\par
In the present paper we are primarily interested in the two
following candidates for the generating set $X$ for $E(\Phi,R)$:
\par\smallskip
$\bullet$ The set of {\it elementary\/} root unipotents
$$ \Omega=\bigl\{x_{\al}(\xi)\mid \al\in\Phi,\xi\in R\bigr\} $$
\noindent relative to the choice of a split maximal torus $T$;
\par\smallskip
$\bullet$ The set of commutators
$$ C=\bigl\{[x,y]=xyx^{-1}y^{-1}\mid x\in G(\Phi,R),\ y\in E(\Phi,R)\bigr\}. $$
\noindent
It is a classical theorem due to Suslin, Kopeiko and Taddei that
for $\rk(R)\ge 2$ one indeed has $C\subseteq E(\Phi,R)$.
\par\smallskip
The width $w_{\Omega}(E(\Phi,R))$ is usually denoted
$w_E(G(\Phi,R))$ and is called the {\bf elementary width}
of $G(\Phi,R)$. Clearly, $w_E(G(\Phi,R))$ is the smallest
$L$ such that $E(\Phi,R)=E^L(\Phi,R)$.
\par
Similarly, the width $w_C(E(\Phi,R))$ is
oftentimes called the {\bf commutator width} of $G(\Phi,R)$
itself.
\begin{remark}
Notice the subtleties related to the necessity to distinguish
the Chevalley group $G(\Phi,R)$ itself, its commutator, the
elementary subgroup $E(\Phi,R)$, etc.
%% \par
In the arithmetic situation they usually all coincide in ranks
$\ge 2$, even in the relative case, this is precisely the [almost] positive solution of the congruence subgroup problem. But for
the group $\SL(2,R)$ (and occasionally for some groups of rank 2)
one will have to impose additional restrictions.
\par
Anyway, in the arithmetic case for simply connected groups
one has
$G_{\sic}(\Phi,R)=E_{\sic}(\Phi,R)$. This means that
the above set $C$ equals the set of all commutators in
$G_{\sic}(\Phi,R)$.
\end{remark}
However, inside the proofs we have to consider some other related
generating sets, such as, for instance:
\par\smallskip
$\bullet$ the set of all unitriangular elements
$$ U(\Phi,R)\cup U^-(\Phi,R); $$
\par%%\smallskip
$\bullet$ or the set of all root unipotents
$$ \Omega^G=\bigl\{x_{\al}(\xi)^g
\mid \al\in\Phi,\xi\in R, g\in G(\Phi,R)\bigr\}, $$
\noindent
which are better behaved with respect to reduction to smaller
ranks.

\subsection{The case of $0$-dimensional rings}
Finiteness of the elementary width is a very rare and
extremely significant phenomenon which has repercussions
everywhere in the structure theory of the group. It is
obvious, and classically known that Chevalley groups over
fields and semi-local rings have finite elementary width.
In fact, the groups over 0-dimensional rings rejoice short
factorisations such as Bruhat decomposition or Gau{\ss}
decomposition. Such factorisations are essentially tantamount
to bounded elementary generation with very sharp bounds.
\par
In fact, Bruhat decomposition immediately implies that over
a field the elementary width of $G(\Phi,K)$ does not
exceed $2N+4l$ (here and below $N=|\Phi^+|$, $l=\rk(\Phi)$).
It immediately follows that the commutator width of $G(\Phi,K)$
is also finite.
\par
But determining the precise value of the commutator width turned
out to be a very challenging problem --- for finite fields it was
the famous Ore conjecture. Without trying to follow the whole
tortuous path, we just mention the two definitive contributions.
For fields containing $\ge 8$ elements Erich Ellers and Nikolai
Gordeev [EG] using Gau{\ss} decomposition with prescribed
semi-simple part
have proven that $w_C(G_{\ad}(\Phi,R))=1$, while
$w_C(G_{\sic}(\Phi,R))\le 2$. This was then extended to the groups
over small fields $\mathbb F_{q}$, $q=2, 3, 4, 5, 7$, by Martin Liebeck,
Eamonn O'Brien, Aner Shalev, and Pham Huu Tiep \cite{LOST1}, \cite{LOST2}, using
explicit information about their maximal subgroups and very
delicate character estimates.
\par
Similarly, Gau{\ss} decomposition which holds over arbitrary
{\it semi-local\/} rings --- or even over rings of stable rank 1, see
\cite{SSV}\footnote{In the literature, three of four completely
different commodities are merchandised under the common
name of Gau{\ss} decomposition: 1) the big cell decomposition LU,
as in the affine group schemes textbooks,
2) the Birkhoff LPU-decomposition as in Ellers---Gordeev \cite{EG},
3) the LUP-decomposition, as in the computational linear algebra textbooks. Here we speak of 4) the DULU-decomposition, consult \cite{SSV} for the historical background.} %%${}^,$\footnote{
in particular --- implies that the
elementary width of $G(\Phi,R)$ does not exceed $3N+4l$.
Actually, \cite{VSS} gives another estimate in terms of unitriangular decomposition, $4N$, which is usually better for groups of very small
ranks, say, up to 4 or 5. What seems to not have been noted in the literature, is that the LUP-decomposition of Chevalley groups over
{\it local\/} rings provides the same upper bound on their width as for fields, $2N+4l$.
\par
As above, bounded elementary generation implies finite
commutator width. However, providing sharp bounds for this
width turned out to be a difficult problem. Skipping a detailed description of the early work by Keith Dennis, Leonid Vaserstein,
You Hong, and others, pertaining to the classical groups
\cite{DV1, DV2, AVY, VW}, we just mention a recent paper by
Andrei Smolensky \cite{Sm}, where such an estimate is
obtained for all Chevalley groups. The commutator width
$w_C(E(\Phi,R))$ does not exceed 3 for $\Phi=\rA_l$ and $\rF_4$,
does not exceed 4 for all other types, apart maybe for $\rE_6$,
and does not exceed 5 for $G(\rE_6,R)$. [We strongly
believe that the commutator width does not exceed 4 also
for $\rE_6$, but we were discouraged by the extent of calculations necessary to improve the bound in this remaining case.]

Note that so far there are no examples of matrices from $\SL(n,\mathbb Z)$, $\SL(n,\mathbb F_q[t])$, $\SL(n,\mathbb F_q[t,t^{-1}])$
($n\ge 3$), not representable as a single commutator.

\subsection{Counter-examples} The groups of rank 1 only
occasionally can have finite elementary width, or finite commutator width, for that matter. Over a Euclidean ring $R$ elementary expressions in $\SL(2,R)$ correspond to continued fractions.
\par
In fact, the existence of arbitrarily long division chains in
$\Int$ and in $K[t]$ implies that the groups $\SL(n,\Int)$ and
$\SL(2,\mathbb F_{q}[t])$ cannot be boundedly generated. The most
classical example are the Fibonacci matrices
$$ \begin{pmatrix} F_{m+1}&F_m\\F_m&F_{m-1} \end{pmatrix} $$
\noindent
which for even $m$ require precisely $m$ elementary factors.
\begin{remark}
For an odd $m$ a similar matrix looks as
$$ \begin{pmatrix} F_{m}&F_{m+1}\\
F_{m-1}&F_{m} \end{pmatrix}, $$
\noindent
which strongly suggests that while considering the width problems
in $\GL(2,R)$ it might be more expedient to switch to Cohn's
generators
$$ \begin{pmatrix} x&1\\1&0 \end{pmatrix}. $$
\noindent
\end{remark}
Of course, the same holds for $\SL(2,\GF{q}[t])$, where instead
of two consecutive Fibonacci numbers one should take two
sufficiently generic polynomials of two consecutive degrees
$m$ and $m-1$, placing the one of the higher degree into the
NW or NE corner, depending on the parity of $m$. Many such
similar examples were constructed by Paul Cohn \cite{Cohn} and
others starting with the mid-1960s.
\par
What came as a shock, though, was that the elementary width
of rank $\ge 2$ groups over a Euclidean ring can be infinite too.
Indeed, using methods of higher algebraic $K$-theory
Wilberd van der Kallen \cite{vdK} has proven that $\SL(3,\Co[t])$
has infinite elementary width. Later Igor Erovenko came up with
a somewhat more elementary proof \cite{Erovenko}. On the other
hand, soon thereafter Dennis and Vaserstein \cite{DV2} noticed
that $\SL(3,\Co[t])$ does not even have finite commutator width.
\par

%%%%%%%%%%%%%%%%%%%%%%%%%%%%%

\subsection{Dedekind rings of arithmetic type, groups of
rank $\ge2$} For rings
of dimension $\ge 2$ one cannot in general expect bounded
generation. An extremely interesting borderline case are
1-dimensional rings, especially the classical example of
the Dedekind rings of arithmetic type. Below, $K$ is a global field,
i.e. a finite extension of $\Rat$ in charactersistic 0, or a finite
extension of $\GF{q}(t)$, $q=p^m$, in positive characteristic $p$.
Further,
$S$ is a finite set of valuations of $K$, containing all Archimedean
ones in the number case, and $R={\mathcal O}_S$.
\par
The {\it number\/} case is well understood. The initial
breakthrough was due to David Carter and Gordon Keller who
have proven that $\SL(n,R)$, $n\ge 3$, is boundedly
elementary generated and gave explicit bounds on in terms
of $n$ and the discriminant\footnote{Or, actually, the number
of its prime divisors. Later, Loukanidis and Murty
\cite{LoMu}, \cite{Mu} obtained bounds that depended on $n$ and the degree
$|K:\Rat|$ of $K$, not the discriminant.} of $K$, see \cite{CaKe1}. The
proof in this paper is essentially an {\it effectivisation\/}
of the usual verification of the familiar properties of Mennicke
symbols.
\par
Actually, their published proof is based on explicit rank reduction
in terms of the stable rank, see below. It remains to verify
bounded generation of $\SL(3,R)$. One of the key calculations
in that paper, Lemma 1,
can be described as follows. Let $A\in\SL(2,R)$ be a matrix
with the first row $(a,b)$. Then $A^m$ can be transformed
to a matrix in $\SL(2,R)$ with the first row $(a^m,b)$ by a
sequence of not more than 16 elementary transformations in
$\SL(3,R)$ --- sic!
\par
However, Carter and Keller mention that their original approach
was based on model theory. To elucidate the connection,
recall that van der Kallen \cite{vdK} observed that
the obstruction to bounded elementary generation of the group
$E(\Phi,R)$ is the quotient $E(\Phi,R)^{\infty}/E(\Phi,R^{\infty})$
(countably many copies). This establishes connection with
ultraproducts and non-standard models. Namely, it can be
interpreted as the equivalence of the bounded generation
of $E(\Phi,R)$ and the [almost] positive solution of the
congruence subgroup problem for $G(\Phi,{}^*R)$ for
non-standard models ${}^*R$ of $R$.
\par
Carter and Keller came up with such a proof for the group
$\SL(n,R)$, initially for $n\ge 3$, see \cite{CaKe3}. Dave
Witte Morris \cite{Mor} gave an exposition of this proof
in a somewhat
more traditional logical language (first-order properties,
compactness theorem, etc.). Unfortunately, %as all such similar logical proofs,
this proof is not much easier than
a direct algebraic proof\footnote{Well, explicit use of
infinitesimals does make life easier, sometimes. Say, in
$R$ itself there are no non-obvious ideals such that
$I^2=I$, whereas in ${}^*R$ there is such an ideal
$\mathcal I$ consisting of all infinitesimal elements, which
can be very handy. But these simplifications are mostly
relevant in the [difficult] case $n=2$, see the next
subsection.} and it gives no bound whatsoever on the
elementary width.
\par
In \cite{CaKe2} Carter and Keller have given a separate
elementary proof
specifically for the [easier] case of $\SL(n,\Int)$, $n\ge 3$, in
terms of direct matrix manipulations, mimicking the verification
of the properties of Mennicke symbols (but without explicitly
mentioning the work of Mennicke and/or of Bass--Milnor--Serre
\cite{BMS}). In particular, they have proven that the elementary
width of $\SL(n,\Int)$ does not exceed 48\footnote{The proof from \cite{CaKe2} with several successful deteriorations, without
reference to \cite{CaKe2}, and with a worse bound 73 was
subsequently reproduced in \cite{AdMe}.}, later this bound
was reduced by Nica \cite{Nic} to 37.
\par
Soon thereafter, Oleg Tavgen invented a different, purely
elementary approach to rank reduction, which allowed him
to reduce the proof of
bounded generation for all Chevalley groups to groups of
rank 2. After that he succeeded in settling the cases of rank 2
groups, $\Sp(4,R)$ and the Chevalley group of type $\rG_2$
(and, in fact, also twisted Chevalley groups of rank 2) by
direct matrix calculations. These important advances sum
up to his main result, the bounded elementary generation
of Chevalley groups of rank $\ge 2$ over arithmetic Dedekind rings
in the {\it number\/} case.

\subsection{Dedekind rings of arithmetic type, groups of
rank $1$} There is a critical difference in behaviour of $\SL(2,R)$,
depending on whether $|S|=1$, in which case $R^*$ is finite,
and $|S|\ge 2$, when $R^*$ is infinite.
As we know, for the case $|S|=1$ the answer
to the question on bounded elementary generation
is negative, so
in the rest of the subsection we assume that $R^*$ is infinite.
\par
Again in the {\it number\/} case the situation is well understood.
Elementary generation of $\SL(2,R)$ is closely related to
generalisations of Euclidean algorithm. Important early inroads
in this direction were suggested [apparently independently!]
by Timothy O'Meara \cite{OMeara}, who simultaneously
considered the number case and the {\it function\/} case,
and by Paul Cohn \cite{Cohn}, who proposed vast
[non-commutative] generalisations.
\par
About a decade later, George Cooke and Peter Weinberger
\cite{CW} systematically studied the length of division chains \cite{Cooke} in the number case. For the case, where $R^*$
is infinite, their main results implied that modulo some form
of the Generalised Riemann Hypothesis (GRH), any matrix in $\SL(2,R)$
is a product of $\le 9$ elementary transvections.
\par
The results of Hendrik Lenstra on the Generalised Artin
Conjecture \cite{Le} --- again conditional on GRH ---
imply that whenever $S$ contains at least one real valuation,
the bound here can be reduced to $\le 7$. Observe that the best possible bound here is\footnote{See \cite{VSS}, where it
is [essentially] proven for $R=\Int\left[1/p\right]$, again
modulo GRH.} $\le 5$. However, Cooke and Weinberger
proposed an example of a matrix over a {\it totally imaginary\/}
ring $R$ of degree 4 which cannot be expressed as a product
of less than 6 elementary matrices.
\par
It has taken quite some time to get rid of the dependence on
the GRH and to improve bounds here. Modulo the GRH, Bruce
Jordan and Yevgeny Zaytman \cite{JZ1} have slightly
remodelled the Cooke--Weinberger argument and improved
the bound to
5 elementary transvections if $K$ is {\it not\/} totally imaginary,
to 6 elementary transvections when $S$ contains at least one
non-Archimedean place, and to 7 elementary transvections
for the integers of a totally imaginary field.
\par
One of the first unconditional
results was obtained by Bernhard Liehl \cite{Li}, but he imposed some additional restrictions on the number field $K$, and his proof
does not give good bounds. Almost simultaneously Carter and Keller,
jointly with Eugene Paige, came up with the first general {\it logical\/}
proof \cite{CKP}, somewhat refashioned in \cite{Mor}. But, as
we already mentioned, this proof gives no bounds whatsoever.
About a decade later Loukanidis and Murty \cite{LoMu}, \cite{Mu} proposed
an unconditional {\it analytic\/} argument, but it only works
provided $S$ is sufficiently large, say $|S|\ge\max(5,2|K:\Rat|-3)$.
\par
Some 10 years ago Maxim Vsemirnov and Sury \cite{Vs}
considered the key example of
$\SL\left(2,\Int\left[1/p\right]\right)$, obtaining the bound
$\le 5$ {\it unconditionally\/}. This was a key inroad to the first
complete unconditional solution of the general case with a good
bound, in the work of Alexander Morgan, Andrei Rapinchuk
and Sury \cite{MRS}. The bound they gave is $\le 9$, but for the
case when $S$ contains at least one real or non-Archimedean
valuation\footnote{Recall that our standing assumption $|S|\ge 2$
excludes the problematic case $R=\Int$.} it was
almost immediately improved [with the same ideas] to $\le 8$
by Jordan and Zaytman \cite{JZ1}.
%%%%%%%%%%%%%%%%%%%%%%%%%%%%%%
\subsection{Reduction to smaller ranks} Let us explain, what do
the width bounds obtained for ranks 1 or 2 imply for higher ranks.
\par
There are two basic ways to reduce the problem of bounded
generation for a Chevalley groups to similar problems for groups
of smaller ranks. We will start with Tavgen's reduction
theorem, which came later historically, but is both more
elementary and more general, than the reduction based on
stability conditions. On the other hand, explicit factorisations
resulting from stability conditions are not always available, but
when they are, they give sharper bounds.
\par
To present Tavgen's idea in its simplest form, let us consider
not the width in elementary generators, but a coarser problem
of determining the width of $G(\Phi,R)$ in terms of the elements
belonging to the unipotent subgroups $U$ and $U^-$. As far
as we know, this problem was first systematically considered
by Dennis and Vaserstein in the context of the closely
related problem of estimating the commutator width for $\SL(n,R)$,
see \cite{DV1, DV2}. In other words, we are interested in finding the smallest
$m$ such that
$$ G(\Phi,R)=UU^-UU^-\ldots U^{\pm},
\qquad\text{$m$ factors}, $$
\noindent
where the last factor equals $U$ or $U^-$ depending on whether
$m$ is odd or even.
\par
Essentially, Tavgen observed that if there are root subsystems
$\Psi_1,\ldots,\Psi_t\subseteq\Phi$ which contain all fundamental
roots, and such that each of the Chevalley groups $G(\Psi_1,R),
\ldots, G(\Psi_t,R)$ admits a similar decomposition with $m$
factors, then $G(\Phi,R)$ itself admits such a decomposition with
$m$ factors. In this [and in fact slightly more general] form this
reduction is described in \cite{VSS, SSV}. Modulo the Levi
decomposition of parabolic subgroups and the Chevalley
commutator formula it is undergraduate group theory, see the next section for  precise statements, somewhat broader discussion, and
a proof.
\par
Since every element of $U$ is a product of not more than
$N=|\Phi^+|$ elementary generators, Tavgen's theorem suffices
to give plausible bounds for the elementary width of large rank
groups in terms of the elementary widths of their rank 1 or
rank 2 subgroups. However, these bounds tend to be somewhat exaggerated.
\par
Actually, for small dimensional rings there is a more precise form
of reduction in terms of the stability conditions. For $\GL(n,R)$
such a reduction in terms of the usual stable rank $\sr(R)$
was first proposed by Hyman Bass in 1964, and then improved
by Vaserstein, Dennis, Kolster, and others. Namely, for
$n\ge\sr(R)$ the usual proof of the surjective stability for
$\operatorname{SK}_1$ grants the following decomposition:
$$ \SL(n+1,R)=\SL(n,R)U_nU_n^-U_nU_n^-. $$
\noindent
It follows that if $\SL(n,R)$ has the elementary width $\le s$, then
$\SL(n+1,R)$ has the elementary width $\le s+4n$ --- and in
fact $\le s+3n+\sr(R)$, if you look inside the proof.
\par
For {\it Dedekind rings\/} this bound was slightly improved
by Carter and Keller \cite{CaKe1}, who noticed that one can do
slightly better
by observing that $\sr(R)\le 1.5$. This means that for $n\ge 2$
one needs just one addition instead of two, to get a shorter
unimodular row. This gives for the elementary width of
$\SL(n,R)$ the estimate $s+\frac{3}{2}n^2-{\frac{1}{2}}n-5$,
where $s$ is the elementary width of $\SL(2,R)$.
\par
Surjective stability of $\operatorname{K}_1$ in terms of various
stability conditions --- the usual stable rank $\sr(R)$, the
absolute stable rank $\asr(R)$, or the like --- is known for all
relevant embeddings of other Chevalley groups. For the usual stability
embeddings of classical groups of the same type, it is indeed
classical, starting with the work of Anthony Bak and Leonid
Vaserstein. For cross-type and exceptional emdeddings such
similar results were established by Michael Stein and Eugene
Plotkin, see in particular \cite{Stein1,Stein2,Pl1, Pl2}.
However, at least in the exceptional cases it was not stated
in the form of such precise decompositions as above.
\par
As a result, the explicit bounds for other groups --- let alone their
improvements for Dedekind rings --- were never mentioned in the
available literature. Even in the number case Tavgen only
states finiteness, without providing any specific bound. In
Section \ref{secG2} below, as part of the proof of Theorem A, we return to
this problem, and procure such explicit bounds.
\par
Let us mention yet another extremely pregnant generalisation,
{\bf bounded reduction}.  In fact, even below the usual stability conditions and even in the absense of the bounded generation
for $G(\Psi,R)$, it makes sense to speak of the number of
elementary generators necessary to reduce an element $g$
of $G(\Phi,R)$ to an element of $G(\Psi,R)$, for a subsystem
$\Psi\subseteq\Phi$.
\par
One such prominent example are polynomial rings $R[t_1,\ldots,R_m]$, where bounded reduction holds starting with a rank depending
on $R$ alone, not on the number of indeterminates. For the case
of $\SL(n,R[t_1,\ldots,R_m])$ this is essentially an effectivisation
of Suslin's solution of the $\operatorname{K}_1$-analogue of
Serre's problem, explicit bounds were obtained in the remarkable
paper by Leonid Vaserstein \cite{Vaser1}, which unfortunately
remained unpublished. For other split classical groups such
bounds were recently obtained by Pavel Gvozdevsky \cite{Gv2}.

\subsection{The function case} In the function case, until
now much less was known concerning the bounded generation
of Chevalley groups. On the one hand, an analogue of Riemann's
Hypothesis was known in this case for quite some time.
On the other hand, in the positive characteristic additional
arithmetic difficulties occur, that have no obvious counterparts
in characteristic 0. They reflect in particular in the structure of
arithmetic subgroups in the function case. For instance, it is well known that the  group $\SL(2,K[t])$ is not even finitely
generated, whereas the groups $\SL(2,K[t,t^{-1}]$ and $\SL(3,K[t])$ are finitely generated but not finitely presented.
\par
Until very recently the only published result was that by
Clifford Queen \cite{Qu}. We discuss this and related work
in much more detail in Section \ref{Queen}. Queen's main result implies
that under some additional assumptions on $R$ --- which hold,
for instance, for Laurent polynomial rings $\mathbb F_{q}[t,t^{-1}]$
with coefficients in a finite field --- the elementary width of
the group $\SL(2,R)$ is $\le 5$.
As we know, this implies, in particular, bounded elementary
generation of all Chevalley groups $G(\Phi,R)$.
\par
The case of the the groups over the usual polynomial ring
$\GF{q}[t]$ long remained open. Only in 2018 has Bogdan
Nica established the bounded elementary generation
of $\SL(n,\GF{q}[t])$, $n\ge 3$. Part of the problem is that
in characteristic $p>0$ bounded elementary generation is not
the same as bounded generation in terms of cyclic subgroups.
For instance, the groups $\SL(n,\GF{q}[t])$ do not have
bounded generation in this abstract sense, see \cite{ALP}.
\par
This is exactly where we jump in. As already stated in the
introduction, in the present paper we prove bounded
{\it elementary\/}
generation for all Chevalley groups of rank $\ge 2$ over
the usual polynomial rings
$\GF{q}[t]$ and --- with better bounds --- for Chevalley
groups of rank $\ge 1$ over a class of function rings with
infinite multiplicative group, including the Laurent polynomial
rings $\GF{q}[t,t^{-1}]$.

\subsection{Further prospects.} \label{prospects}
The historical description is
already rather long, we cannot mention many further aspects.
A systematic survey should include at least:
\par\smallskip
$\bullet$ Partial positive results, such as bounded expressions
of elementary conjugates and commutators in terms of
elementary generators --- decomposition of unipotents,
Stepanov's universal localisation, and the like.
\par\smallskip
$\bullet$ Connection with the prestability kernel, bounded generation of $\SL_2$ in terms of Vaserstein prestability generators, \cite{Vaser2}, etc.
\par\smallskip
$\bullet$ Connection of the bounded generation with the congruence subgroup problem, Kazhdan's property (T), finite presentation, super-rigidity, etc.
\par\smallskip
$\bullet$ Implications for the bounded generation
by cyclic/abelian subgroups, including actions, etc.
\par\smallskip
$\bullet$ Extension of known bounds for word width (such as in \cite{AvMe}) to the function case. 

\medskip
  
We intend to return to [some of] these subjects in an expected sequel to the present paper.

%%%%%%%%%%%%%%%%%%%%%%%%%%%%%%%

\section{Outline of the proof of Theorem \ref{mtheorem2a} and
reduction to rank 2} \label{sec:scheme}

In this section we sketch the main ideas of the proof and
implement the rank reduction. Together with the result by
Nica \cite{Nic}, this already suffices to establish
Theorem A for simply laced types and type
${\rF}_4$.

\subsection{Outline of the proof}

The proofs of bounded generation for the rings of integers of an algebraic number field, see \cite{CaKe1}, \cite{CaKe2}, \cite{AdMe}, \cite{Ta}, deploy similar ideas. Let
$$
A=\left(\begin{array}{cc} a&b \\ c&d  \\ \end{array}\right)
$$
be a matrix from $\SL(2,R)$ nested either  in $G=\SL(3,R)$ or in $G=\Sp(4,R)$. Observe that in
the second case there are two natural embeddings of
$\SL(2,R)$, on short roots and on long roots, and that is
a major aspect of the quest. We also provide an approach based on the reduction to Chevalley groups of rank 3.
This approach has some advantages and makes use of embeddings of the Chevalley group of type $G(\rA_2,R)$ into either $G(\rC_3,R)$ or $G(\rB_3,R)$. The Chevalley groups
of type $\rG_2$ are to be treated separately anyway,
but they do not occur in the analysis of higher rank cases.
\par
The goal  is to reduce $A$ to the identity matrix by elementary transformations in $G$ in such a way that the number of
elementary factors does not depend on $A$. The
guideline of the proof can be summarised as follows.
\par\smallskip
$\bullet$ Eventually, one has to transform $A$ to a matrix with
an invertible entry by a bounded number of elementary
transformations.
\par\smallskip
$\bullet$ One way to do that is to use a version of Little Fermat's Theorem. So we need some entry of $A$ in an appropriate power.
\par\smallskip
$\bullet$ Hence, we need to produce an elementary descendant $B$ of $A$ with some entry, say the first one, to be $a^k$, where $k$ is  an appropriate power. This is achieved by Lemmas 1 in \cite{CaKe1}, \cite{CaKe2},
\cite[Lemma~2]{AdMe}, \cite[Proposition~3]{Ta}.
\par\smallskip
$\bullet$ The proof follows from the miraculous fact that the matrix
  $$
A^k=\left(\begin{array}{cc} a&b \\ c&d  \\ \end{array}\right)^k
$$
coincides modulo elementary matrices with the matrix
$$
B=\left(\begin{array}{cc} a^k&b\\ *&*  \\ \end{array}\right)
$$
\par\smallskip
$\bullet$ This miracle is none other than the multiplicative property of Mennicke symbols,
so this is not a surprise at all  modulo a tricky proof of this property (see \cite{Mi}, \cite{Mag}, etc).
\par\smallskip
$\bullet$
It remains to use a combination of analytic tools such as Dirichlet's theorem on primes in arithmetic progressions %Chinese remainder theorem
and, if needed, reciprocity laws to obtain by elementary transformations a matrix of the form
$$
B=\left(\begin{array}{cc} a^k&p\\ q&*  \\ \end{array}\right)
$$
where the pair $(a^k,p)$ satisfies $a^k-1=ps$ for some $s$.
\par\smallskip
Note that Nica \cite{Nic}  modified the proof using the so-called "swindling lemma".
%Thus he avoids Mennicke type calculations and obtain a bit better bounds for the number of elementary generators needed to represent any matrix. His approach successfully works in
%$\SL_n$ case and apparently can be generalized to symplectic case as well.
We shall discuss this trick in more detail in Section \ref{secC2}.
Actually, ``swindling'' is merely a weaker version of the
multiplicativity of Mennicke symbols. However, the advantage
is that this weaker form is cheaper in terms of the number of
elementary moves, and here we generalise this approach to the symplectic case as well.

\begin{remark}\label{rm2}
One of the points of the present work is that, unlike the proofs
based on model theory, here we get {\it efficient\/} realistic
estimates for the number of elementary factors, with bounds
that depend on $\Phi$ alone.
In some cases, like for the bounded reduction to smaller rank,
our bounds are [very close to] the best possible ones. For
small ranks,
there might be still some gap between the counter-examples
and the estimates we obtain, but our upper bounds are still
reasonably close to the theoretically best possible ones. The
lower bounds in such similar problems are usually quite difficult
to obtain, anyway.
\end{remark}
\subsection{Tavgen's reduction theorem} Here we reproduce
with minor variations the {\it elementary\/} reduction procedure
due to Tavgen, in the form mentioned in \cite{VSS}, \cite{SSV}.
This procedure suffices to reduce Theorem A for groups of rank
$\ge 3$ to the groups $\SL(3,R)$ and $\Sp(4,R)$. It of course works also for reduction to $\Sp(6,R)$ and $\SO(7,R)$ used in Section \ref{rank3}. Moreover, the bounds it gives are quite reasonable, though clearly not the best possible ones. In Section \ref{secG2} we work out the {\it stable\/} reduction, based on the fact the stable rank of Dedekind rings equals 1.5. This approach gives much better bounds for reduction, sometimes the sharp ones, and for exceptional groups it is new even in the number case.
\par
Tavgen's approach works more smoothly for {\it unitriangular
factorisations\/}, in other words, for expressions of elementary
subgroup $E(\Phi,R)$ as a product of subgroups $U(\Phi,R)$ and $U^-(\Phi,R)$,
$$ E(\Phi,R)=U(\Phi,R)U^-(\Phi,R)\ldots U^{\pm}(\Phi,R).  $$
\noindent
Later on in \cite{SSV} it was applied to {\it triangular factorisations\/},
where also the toral factor is admitted\footnote{As we know from
Section \ref{sec:scheme}, this does not influence boundedness or lack thereof,
but may affect the actual bounds.}.

The leading idea of Tavgen's proof is very general and beautiful, and works in many other related situations. It relies on the fact that for systems of rank $\ge 2$ every fundamental root falls into the
subsystem of smaller rank obtained by dropping either the first
or the last fundamental root. However, as was pointed out by the referee of \cite{SSV}, the argument applies without any modification
in a much more general setting. Namely, it suffices to assume that
the required decomposition holds for {\it some\/} subsystems
$\Delta=\Delta_1,\ldots,\Delta_t$, whose union contains all
fundamental roots of $\Phi$. These subsystems do not have to be terminal, or even irreducible, for that matter!

\begin{theorem} \label{th:reduction}
Let\/ $\Phi$ be a reduced irreducible root system of rank $l\ge 2$,
and\/ $R$ be a commutative ring. Further, let
$\Delta=\Delta_1,\ldots,\Delta_t$ be some subsystems of $\Phi$,
whose union contains all fundamental roots of $\Phi$.
Suppose that for all $\Delta=\Delta_1,\ldots,\Delta_l$, the elementary Chevalley group\/ $E(\Delta,R)$ admits a unitriangular factorisation
$$ E(\Delta,R)=U(\Delta,R)U^-(\Delta,R)\ldots
U^{\pm}(\Delta,R) $$
\noindent
of length $L$. Then the elementary Chevalley group\/ $E(\Phi,R)$
itself admits unitriangular factorisation
$$ E(\Phi,R)=U(\Phi,R)U^-(\Phi,R)\ldots
U^{\pm}(\Phi,R) $$
\noindent
of the same length\/ $L$.
\end{theorem}

\par
Let us reproduce the details of the argument. By definition,
$$ Y=U(\Phi,R)U^-(\Phi,R)\ldots U^{\pm}(\Phi,R) $$
\noindent
is a {\it subset\/} in $E(\Phi,R)$. Usually, the easiest way
to prove that a subset $Y\subseteq G$ coincides with the whole
group $G$ consists in the following

\begin{lemma} \label{lem:generate}
Assume that\/ $Y\subseteq G$, $Y\neq \varnothing$,
and\/ let $X\subseteq G$ be a symmetric generating set. If\/
$XY\subseteq Y$, then\/ $Y=G$.
 \end{lemma}

 Now, we are all set to finish the proof of Theorem 3.2

\begin{proof}
By Lemma \ref{lem:root}, the group $E(\Phi,R)$ is generated by the
fundamental root elements
$$ X=\big\{x_{\alpha}(\xi)\mid \alpha\in\pm\Pi,\ \xi\in R\big\}. $$
\noindent
Thus, by Lemma \ref{lem:generate} is suffices to prove that $XY\subseteq Y$.
\par
Fix a fundamental root unipotent $x_{\alpha}(\xi)$.
Since $\rk(\Phi)\ge 2$, the root $\alpha$ belongs to at least one
of the subsystems $\Delta=\Delta_r$, where $r=1,\ldots,t$.
Set $\Sigma=\Sigma_r$ and express
$U^{\pm}(\Phi,R)$ in the form
$$ U(\Phi,R)=U(\Delta,R)U(\Sigma,R),\quad
U^-(\Phi,R)=U^-(\Delta,R)U^-(\Sigma,R). $$
\par
Using Lemma \ref{lem:semidirect}, we see that
$$ Y=U(\Delta,R)U^-(\Delta,R)\ldots U^{\pm}(\Delta,R)\cdot
U(\Sigma,R)U^-(\Sigma,R)\ldots U^{\pm}(\Sigma,R). $$
\noindent
Since $\alpha\in\Delta$, one has $x_{\alpha}(\xi)\in E(\Delta,R)$,
so that the inclusion $x_{\alpha}(\xi)Y\subseteq Y$ immediately
follows from the assumption.
\end{proof}

\subsection{Proof of Theorem A for simply laced systems and
in the case of $\rF_4$} In \cite{SSV} the authors commented
that they do not see immediate applications of the more general
form of Tavgen's reduction theorem, as stated above. Here,
we notice that it is in fact {\it surprisingly\/} strong, since it
allows one to pass from {\it some\/} smaller rank subsystems to the
whole system, without looking at any other subsystems,
including those of intermediate ranks! Indeed, it may happen
that for those other subsystems bounded generation
holds with some larger bound, or bluntly fails.
\par
Of course, the easiest case is when the group $\SL(2,R)$
itself has bounded elementary generation.

\begin{corollary} \label{tt}
Let any element of $\SL(2,R)$ be a product of $\le L$
elementaries. Then any simply connected Chevalley
group $G=G(\Phi ,R)$ admits unitriangular factorisation
$$ G=UU^-U\ldots U^{\pm} $$
\noindent
of length $L$.
\end{corollary}

However, this is very seldom the case, so one should start
looking at larger rank subsystems. Recall that in
$\rA_2$ case Theorem A was proven by Nica \cite{Nic}.
His main new result can be stated as follows.

\begin{prop}\label{nica}
Any element of $\SL(3,\GF{q}[t])$ is a
product of\/ $\le 41$ elementary transvections.
\end{prop}

We are not contending for the best possible bounds in terms of unitriangular matrices at this stage, since later we improve the
bounds anyway.
Interestingly, the main arithmetic ingredient of his proof is the Kornblum---Artin functional version of Dirichlet's theorem
on primes in arithmetic progressions. In Section \ref{secC2} below we shall
see how it works in the parallel example of $\Sp(4,\GF{q}[t])$.
\par
Now, together with Theorem \ref{th:reduction} this result by Nica implies
Theorem A for the two following cases.
\par\smallskip
$\bullet$ Chevalley groups of {\it simply laced\/} type $\Phi$
of rank $\ge 2$. Indeed, in this case $\Pi$ is covered by the
fundamental copies of $\rA_2$ spanned by all pairs of
adjacent fundamental roots.
\par\smallskip
$\bullet$ Chevalley group of type $\rF_4$. Indeed, in this case
$\Pi$ is covered by {\it two\/} fundamental copies of $\rA_2$ ---
the long one $\rA_2$, spanned by the two fundamental {\it long\/} roots, and the short one $\widetilde{\rA}_2$, spanned
by the two fundamental {\it short\/} roots.
\par\smallskip
Observe that in the second case it is neither assumed, nor does
it follow that the group $\Sp(4,R)$ is boundedly generated!
Even more amazingly, the same applies to the subgroups
of types $\rB_3$ and $\rC_3$.
\par
However, for root systems of types $\rB_l$ and $\rC_l$ there
are short/long roots that cannot be embedded into any irreducible
rank 2 subsystem other than $\rC_2$. Thus, to be able to apply
Theorem \ref{th:reduction} we have to explicitly dismantle elements of
$\Sp(4,\GF{q}[X])$ into elementary factors. This is exactly what
is achieved in Section \ref{secC2}.
\par
However, since we are interested in actual bounds, before
treating this case, we have to recall an alternative approach
to rank reduction, based on stability conditions. In the next
section we recall the stability conditions themselves and
illustrate how they work for Chevalley groups of type $\rG_2$.
Later, in Sections~5 and 6 we produce similar arguments for
groups of types $\rC_2$, $\rC_3$, and $\rB_3$.

%%%%%%%%%%%%%%%%%%%%%%%%%%%%%G2case

\section{Proof of Theorem A in the case of $\rG_2$}\label{secG2}

\def\map{\longrightarrow}
The purpose of this section is two-fold. As a first objective, here
we provide the proof of Theorem A for %the only remaining case,
the Chevalley group of type $\rG_2$. This is done by virtue
of surjective stability for the embedding $\rA_2\longrightarrow\rG_2$.
Using this opportunity, we revisit stability for Dedekind rings
also for other embeddings, and obtain accurate bounds for
reduction in this case. For exceptional groups such explicit
bounds are new even in the number case.
\subsection{Stability conditions}
Traditionally, stability results are stated in terms of stability
conditions. The first such condition, stable rank, was
introduced by Hyman Bass back in 1964. However, {\it surjective\/}
stability results for $\mathrm K_1$ for embeddings other than the simplest
stability embeddings
$$ \SL(n,R)\map\SL(n+1,R)\quad\text{and}\quad
\Sp(2l,R)\map\Sp(2(l+1),R) $$
\par\noindent
usually require {\it stronger\/} stability conditions, such as the
{\it absolute\/} stable rank, etc.
\par
Modulo some small additive constants,
all these ranks are bounded by the Krull dimension $\dim(R)$
or even the Jacobson dimension $\dim\big(\Max(R)\big)$ of
the ring $R$.
On the other hand, arithmetic rings, such as Dedekind rings and
their kin, usually satisfy even stronger stability conditions than
the ones that would follow from their dimension. Here we very briefly recall
some of these conditions, limiting ourselves only to those that are actually used in the sequel.
\par
A row $(a_1,\ldots,a_n)\in {}^{n}\!R$ is called {\it unimodular\/}
if its components $a_1,\ldots,a_n$ generate $R$ as a right ideal,
$$ a_1R+\ldots+a_nR=R, $$
\noindent
or, what is the same, if there exist
$b_1,\ldots,b_n\in R$ such that
$$ a_1b_1+\ldots+a_nb_n=1. $$
\par
A row $(a_1,\ldots,a_{n+1})\in {}^{n+1}\!R$ of length $n+1$
is called {\it stable\/} if there exist
$b_1,\ldots ,b_n\in R$ such that the ideal generated by
$$ a_1+a_{n+1}b_1,a_2+a_{n+1}b_2,\ldots,a_n+a_{n+1}b_n $$
\noindent
coincides with the ideal generated by $a_1,\ldots,a_{n+1}$.
\par
The {\it stable rank\/} $\sr(R)$ of the ring $R$ is the {\it smallest\/} $n$ such
that every unimodular row $(a_1,\ldots,a_{n+1})$ of length $n+1$ is
{\it stable\/}. In other words, there exist
$b_1,\ldots ,b_n\in R$ such that the row
$$ (a_1+a_{n+1}b_1,a_2+a_{n+1}b_2,\ldots,a_n+a_{n+1}b_n) $$
\noindent
of length $n$ is unimodular. If no such $n$ exists, one writes
$\sr(R)=\infty$.
\par
Bass himself denoted stability of unimodular rows of
length $n+1$ by $\SR_{n+1}(R)$. It is easy to see that condition
$\SR_m(R)$ implies condition $\SR_n(R)$ for all $n\ge m$, so that
the stable rank is defined correctly: if $n>\sr(R)$, then
every unimodular row of length $n$ is stable.  Clearly, this means
that when $n>\sr(R)+1$ one can iterate
the process of shortening a unimodu\-lar row and eventually
reduce any unimodular row to a unimodular row of length $\sr(R)$.
\par
For representations other than the vector representations of
$\SL_n$ and $\Sp_{2l}$, the stock of available elementary
transformations is limited, so that one has to work with
pieces of unimodular rows, that are not themselves unimodular.
However, stability of all non-unimodular rows is an {\it
exceedingly\/} restrictive condition --- though Dedekind rings satisfy precisely something of the sort!
\par
%%% Thus, one has to weaken the sense in which stability
%%% of rows a weaker
The most familiar variation of stable rank, that works for other
classical groups, is the absolute stable rank. For commutative rings
this condition was introduced by David Estes and Jack Ohm
\cite{EO}, whereas Michael Stein \cite{Stein2} discovered its relevance in the
study of orthogonal groups and exceptional groups.
%%% See also [37], [45]--[47], [13], [15], [12] for estimates of
%%% the absolute stable rank, and its non-commutative
%%% generalisations.
\par
For a row $(a_1,\ldots,a_{n})\in {}^{n}R$ let us denote by
$J(a_1,\ldots,a_{n})$ the intersection of the maximal ideals of
the ring $R$ containing $a_1,\ldots,a_{n}$. In particular, a row
is unimodular if and only if $J(a_1,\ldots,a_{n})=R$.
\par
One says that a commutative ring $R$ satisfies condition
$\ASR_{n+1}$ if for any row $(a_1,\ldots,a_{n+1})$ of length
$n+1$ there exist $b_1,\ldots b_n\in R$ such that
$$ J(a_1+a_{n+1}b_1,\ldots,a_n+a_{n+1}b_n)=
J(a_1,\ldots,a_{n+1}). $$
\noindent
It is obvious that condition $\ASR_m(R)$ implies condition
$\ASR_n(R)$ for all $n\ge m$. The {\it absolute stable rank\/}
$\asr(R)$ of the ring $R$ is the smallest natural $n$ for which
condition $\ASR_{n+1}(R)$ holds. Clearly, $\sr(R)\le\asr(R)$.
\par
The classical theorem of Estes and Ohm \cite{EO} asserts that
for commutative rings one has
$$ \asr(R)\le\dim\big(\Max(R)\big)+1, $$
\noindent
a similar estimate for $\sr(R)$ follows from a classical theorem of Bass.
Thus, in particular, any Dedekind ring satisfies $\ASR_3(R)$ --- and,
as we recall below, a much stronger condition.

%%%%%%%%%%%%%%%%%%%%%%%%%%%%

\subsection{Surjective stability for ${\rm K}_1$ and bounded reduction.} Recall
that the $\mathrm K_1$-functor modelled on a Chevalley group $G(\Phi,R)$
is defined as
$$ \mathrm K_1(\Phi,R)=G(\Phi,R)/E(\Phi,R). $$
\noindent
For [irreducible] root systems of rank $\ge2$ the elementary subgroup
$E(\Phi,R)$ is a normal subgroup of $G(\Phi,R)$, so that in
this case ${\rm K}_1(\Phi,R)$ is a group.
\par
Now, by the homomorphism theorem every embedding of root
systems $\Delta\subset\Phi$ gives rise to the stability map
$$ \nu=\nu_{\Delta\to\Phi}\colon \mathrm K_1(\Phi,\Delta)\map \mathrm K_1(\Phi,R), $$
\noindent
and one of the archetypical classical problems of the algebraic
$\mathrm K$-theory, whose study was initiated by Hyman Bass in the early
1960s, is to find
conditions under which this map is surjective or injective.
\par
Clearly, {\bf surjective stability} for the embedding
$\Delta\subset\Phi$ amounts to the equality
$$ G(\Phi,R)=G(\Delta,R)E(\Phi,R). $$
\noindent
In other words, any matrix $g\in G(\Phi,R)$ can be expressed
as a product of a matrix from $G(\Delta,R)$ and elementary
unipotents.
\par
However, {\it in the stable range\/}, that is when $\rk(\Delta)$
is large with respect to $\dim(R)$, one can use the above stability
conditions and establish rather more. In this setup, all customary proofs
of surjective stability afford not just elementary reduction to smaller
rank, but {\it bounded\/} elementary reduction. In other words, they
establish an equality of the type
$$ G(\Phi,R)=G(\Delta,R)E^L(\Phi,R), $$
\noindent
for some constant $L$ depending on the dimension of the ring
$R$ and the embedding $\Delta\subset\Phi$. This means that we
have {\bf bounded reduction}: any matrix $g\in G(\Phi,R)$ can be
expressed as a product of a matrix from $G(\Delta,R)$ and not more
than $L$ elementary unipotents, where $L$ does not depend on $g$.
\par
When $\Delta$ is the reductive part of a parabolic subset $S$ of $\Phi$,
%% [the semisimple part of] the Levi subgroup
the actual value of $L$ is estimated in terms of the order of the
unipotent part $\Sigma$ of $S$.  Thus, as we have already mentioned
in Section \ref{sec:scheme}, for the embedding $\rA_{n-1}\subset\rA_n$ the original Bass's
proof furnishes the following classical decomposition
%%% {\bf Bass---Kolster decomposition}
$$ \SL(n+1,R)=\SL(n,R)U_nU^-_nU_nU^-_n, $$
\noindent
which implies that in this case $L$ is at most $4n$.
Actually, since one needs only $\sr(R)$ additions to shorten a unimodular
row, this bound immediately reduces to $3n+\sr(R)$.
\par
However, for all other embeddings, apart from $\rC_{n-1}\subset\rC_n$,
and especially for exceptional groups and for root subsystems that are
not reductive parts of parabolic subsets, it is not that immediate. Even
in the classical cases, not to mention the exceptional ones, the exact
number of elementary unipotents used in the reduction was not
explicitly tracked.
\par
Indeed, the existing proofs of surjective stability do not bother about explicit
bounds. At the moment, one could invoke a previously known stability
result with the same or weaker stability condition, one would do that,
without actually reproducing the reduction procedure, or worrying for the
shortest elementary expressions. For anyone familiar with the proofs of
surjective stability in, say \cite{Stein2, Pl1,Pl2, Gv1},
it is clear that they afford bounded reduction with {\it some\/} $L$.
Note that these bounds are valid in the case of any base ring of Krull dimension 1
and hence for any Dedeking ring.
But
any such bounds are not explicit there, and one should go over all proofs
in these papers once again even to produce {\it some\/} bounds (not the
best possible ones!).
\par
Additional features of the exceptional cases are that --- with the sole exception
of $\rG_2$ --- their minimal representations are too large for manual matrix
computations, and even in these representations
the elementary unipotents are significantly more complicated. Thus, instead
of matrices one should use some tools from representation theory, as do
\cite{Stein2, Pl1, Pl2, Gv1}. It would take quite a few pages to
describe these tools, and adjust them to our needs. To establish Theorem A
with some [reasonable] bound, we do not need that. Actually, we intend to
return to this issue in the sequel to this paper, and come up with {\it sharp\/}
bounds. In the next section we limit ourselves with the proof specifically for the
long root embeddings $\rA_1\subset\rA_2\subset\rG_2$.

%%%%%%%%%%%%%%%%%%%%%%%%%%%%%%%

\subsection{Proof of Theorem A for $\rG_2$}
In his pathbreaking paper \cite{Stein2} Michael Stein proves, in particular,
that under the absolute stable range condition $\ASR_3(R)$ one has
$$ G(\rG_2,R)=G(\rA_1,R)E(\rG_2,R)=G(\rA_2,R)E(\rG_2,R), $$
\noindent
[long root embeddings], this is his Theorem 4.1.m. Below, we go through
the proof of that theorem, to come up with an actual bound.

\begin{theorem}\label{G2}
Under the assumption $\ASR_3(R)$ one has
$$ G(\rG_2,R)=G(\rA_1,R)E^{24}(\rG_2,R)=G(\rA_2,R)E^{24}(\rG_2,R). $$
\end{theorem}

Clearly, this result together with the main theorem of \cite{Nic} immediately
implies the claim of Theorem A for the case of $\rG_2$.
Indeed, $\SL(2,R)$ is boundedly elementary generated, and since
Dedekind rings have dimension $\le 1$ and thus satisfy condition $\ASR_3$, it follows from the above
result that
$$ w_E\big(G(\rG_2,R)\big)\le w_E\big(G(\rA_2,R)\big)+24. $$

\begin{proof}
Our proof closely follows that in \cite{Stein2}, pages 102--104, and we
essentially preserve the notation thereof. Let $\alpha_1,\alpha_2$ be the
fundamental roots of $\rG_2$, with $\alpha_2$ long. Further, consider the
short roots
$$ \alpha=-\alpha_1,\qquad \beta=2\alpha_1+\alpha_2,\qquad
\gamma=-\alpha_1-\alpha_2, $$
\noindent
which clearly sum to zero, $\alpha+\beta+\gamma=0$.
\par
Consider the 7-dimensional short root representation of $G(\rG_2,R)$,
with the highest weight $\mu=\beta$, its weights are the short roots
$\pm\alpha,\pm\beta,\pm\gamma$ and 0. Order the weights by height,
$\mu=\beta,-\gamma,-\alpha,0,\alpha,\gamma,-\beta$.
\par
As usual, the entries of matrices $g\in G(\rG_2,R)$ are indexed by
pairs of weights, $g=(g_{\lambda,\mu})$, where
$\lambda,\mu=\beta,\ldots,-\beta$.
\par
Initially, we concentrate on the first column $g_{*\beta}$ of this matrix, which
is the image of the highest weight vector under the action of $g$. For
typographical reasons, we denote this column by
$$ (x_\beta,x_{-\gamma},x_{-\alpha},x_0,x_\alpha,x_\gamma,x_{-\beta}). $$
\noindent
It is our intention to reduce this column to the form $(1,*,*,*,*,*,*)$ by
elementary unipotents.
\par
This can be done as follows. Not to proliferate indices in this {\bf and further stability calculations},
we will not {\it rename\/} [as mathematicians would do], but {\it reset\/}
[as is typical in programming] our variables $g$ and $x$, still denoting them
by the same letters after each successive transformation.

In order to make the action of elementary unipotents visible, below we present the weight diagram  of
the 7-dimensional short root representation of $G(\rG_2,R)$:
$$
\begin{tikzpicture}
\draw[thick] (0, 0) circle (0.07) node[below] {$\mu = \beta$};
\draw[thick] (1.7, 0) circle (0.07) node[below] {$-\gamma$};
\draw[thick] (3.4, 0) circle (0.07) node[below] {$-\alpha$};
\draw[thick] (5.1, 0) circle (0.07) node[below] {$0$};
\draw[thick] (6.8, 0) circle (0.07) node[below] {$\alpha$};
\draw[thick] (8.5, 0) circle (0.07) node[below] {$\gamma$};
\draw[thick] (10.2, 0) circle (0.07) node[below] {$-\beta$};
\draw[thick] (0.07, 0) to node[midway, above] {$\beta + \gamma$} (1.7 - 0.07, 0);
\draw[thick] (1.7 + 0.07, 0) to node[midway, above] {$\alpha - \gamma$} (3.4 - 0.07, 0);
\draw[thick] (3.4 + 0.07, 0) to node[midway, above] {$-\alpha$} (5.1 - 0.07, 0);
\draw[thick] (5.1 + 0.07, 0) to node[midway, above] {$-\alpha$} (6.8 - 0.07, 0);
\draw[thick] (6.8 + 0.07, 0) to node[midway, above] {$\alpha - \gamma$} (8.5 - 0.07, 0);
\draw[thick] (8.5 + 0.07, 0) to node[midway, above] {$\beta + \gamma$} (10.2 - 0.07, 0);
\end{tikzpicture}
$$
As usual, the action of the elementary unipotent $x_\gamma (t)$ on the first column $g_{*\beta}$ can be viewed by looking for pairs of weights on the weight digramm connected by the root $\gamma$.
\par\smallskip
$\bullet$ Using condition $\SR_3(R)$, we can find $a_1,a_2\in R$ such that
the shorter column
$$ (x_\beta+a_1x_0,x_{-\gamma},x_{-\alpha}+a_2x_0,\_,
x_\alpha,x_\gamma,x_{-\beta}), $$
\noindent
where the blank indicates the position of the component $x_0$ that we drop,
is unimodular. Reset $g$ to $x_{\beta}(a_1)x_{-\alpha}(a_2)g$ ---
this requires 2 elementary operations. After this step we may assume that
$(x_\beta,x_{-\gamma},x_{-\alpha},x_\alpha,x_\gamma,x_{-\beta})$ is unimodular.
\par\smallskip
$\bullet$ Observe that every elementary long root unipotent
$x_{\delta}(\xi)$ adds one of the components
$x_{\beta},x_{\alpha},x_{\gamma}$  to another one of them,
acts in the opposite direction on the components $x_{-\beta},x_{-\alpha},x_{-\gamma}$,
and fixes $x_0$. This corresponds to the decomposition of the 7-dimensional
representation of $G(\rG_2,R)$ into two 3-dimensional and one 1-dimensional
invariant subspaces, when restricted to $G(\rA_2,R)$.
\par
Thus, we consider the ideal $I$ generated by the components
$x_{\beta},x_{\alpha},x_{\gamma}$. As we just observed, this ideal is not
changed by the action of any element of $E(\rA_2,R)$. However, under
the condition $\SR_3(R/I)$ transitivity of the action $\SL(3,R)$ in the
3-dimensional vector representation is well known from the work of Bass.
For this we need 2 additions to shorten a unimodular column over $R/I$ of
length 3 to two positions, then 2 additions to get 1 in the third position,
and, finally, 2 additions to clear the components in the remaining two positions.
This is 6 elementary operations altogether.
\par
This means that further multiplying $g$ by 6 factors of the form
$x_{\pm(\beta-\alpha)}(*)$ and $x_{\pm(\beta-\gamma)}(*)$
we obtain a column of height 6
$$ (x_\beta+a_1x_0,x_{-\gamma},x_{-\alpha}+a_2x_0,\_,
x_\alpha,x_\gamma,x_{-\beta}), $$
\noindent
subject to the extra condition that
$$ x_{-\beta}\equiv 1\pmod{I},\qquad x_{-\alpha},x_{-\gamma}\equiv0\pmod{I}. $$
\noindent
In other words, already the following column of height 4
$$ (x_\beta,\_,\_,\_,x_\alpha,x_\gamma,x_{-\beta}). $$
\noindent
is unimodular.
\par
So far, we only invoked the usual stable rank condition $\SR_3$.
Next, the tricky part comes, which requires the use of $\ASR_3$.
\par\smallskip
$\bullet$ Using condition $\ASR_3(R)$ we can find $b_1,b_2\in R$ such that
the ideal $J$ generated by
$x_{\alpha}+b_1x_{\beta},x_{\gamma}+b_2x_{\beta}$ is contained in
the same maximal ideals that the [a priori larger] ideal $I$ generated by
$x_\beta,x_\alpha,x_\gamma$.
\par
This means that resetting $g$ to
$x_{\alpha-\beta}(b_1)x_{\gamma-\beta}(b_2)g$ ---
that's further 2 elementary operations --- we may assume that
the following column of height 3
$$ (\_,\_,\_,\_,x_\alpha,x_\gamma,x_{-\beta}) $$
\noindent
is unimodular.
\par\smallskip
$\bullet$ Now, using condition $\SR_3(R)$ once more
we can find $c_1,c_2\in R$ such that the following column of height 2
$$ (\_,\_,\_,\_,x_\alpha+c_1x_{-\beta},x_\gamma+c_2x_{-\beta},\_) $$
\noindent
is unimodular.
\par
As usual, we reset $g$ to $x_{-\gamma}(c_1)x_{-\alpha}(c_2)g$ ---
that's 2 more elementary operations.
\par\smallskip
$\bullet$ After the previous step we may assume that
$$ (\_,\_,\_,\_,x_\alpha,x_\gamma,\_) $$
\noindent
is unimodular, and we are done. It remains to express
$$ 1-x_{\beta}=d_1x_{\alpha}+d_2x_{\gamma}, $$
\noindent
and to reset $g$ to $x_{\beta-\alpha}(d_1)x_{\beta-\gamma}(d_2)g$ ---
that's 2 more elementary operations --- to achieve our intermediate goal
$x_{\beta}=1$.
\par\smallskip
$\bullet$ Up to now we have used 14 elementary operations in $E(\rG_2,R)$.
On the other hand, a matrix $g$ with 1 in the diagonal position corresponding
to the highest weight can be readily reduced to smaller rank, in our case,
$$ g\in G(\rA_1,R)U_2U^-_2. $$
\noindent
This is exactly the celebrated Chevalley--Matsumoto decomposition theorem, see, for instance \cite{Mats, Stein2, NV91}
(the same argument was used in \cite{Ta}).
But $\dim(U_2)=5$, which consumes $\le 10$
more elementary unipotents, not more 24 elementary factors altogether, as
claimed.
\end{proof}

We are in possession of similar reduction results, with pretty sharp bounds, also
for all other exceptional cases. But calculations with columns of height 26, 27,
56 and 248 are quite a bit more involved. In the present paper we limit
ourselves with {\it some\/} explicit bounds, resulting from Tavgen's
approach. We intend to come up with much sharper bounds in the sequel to
this paper.

%%%%%%%%%%%%%%%%%%%%%%%%%%%%%%%

\subsection{Improvements for Dedekind rings} \label{subsec:Improvements}
As is well known, for Dedekind rings the constants in the reduction can be slightly improved. This is based on the well-known property that the
ideals $I$ in Dedekind rings are not just 2-generated, but
rather 1.5-generated. In other words, one of the generators
can be an arbitrary non-zero element of $I$.
\par
More precisely, let $I\unlhd R$ be an ideal of a Dedekind ring
$R$. Then for any $a\in I$, $a\neq 0$, there exists
$b\in I$ such that $aR+bR=I$. This translates into the following
stability condition, weaker than $\sr(R)=1$, but strictly
stronger than $\sr(R)=2$.
\par
\begin{lemma}\label{Dedek}
Let $R$ be a Dedekind ring, and $I\unlhd R$ be its ideal.
Then for any three elements $a,b,c\in R$ generating $I$
there exists $d\in R$ such that $a,b+dc$ or
$a+dc,b$ generate $I$.
\end{lemma}
In particular, {\it one\/} addition, instead of two suffices
to shorten a unimodular colum of height 3. This
property was used by Carter and Keller to get a sharp
bound for $\SL(n,R)$, since to reduce a matrix from
$\SL(3,R)$ to a matrix from $\SL(2,R)$ one now needs
7 elementary operations  instead of 8 that are expected
for general rings with $\sr(R)=2$.
\par
Here we illustrate this idea by slightly improving the
bound in the result of the previous section pertaining to
groups of type $\rG_2$.
\begin{prop}\label{G2bis}
For a Dedekind ring $R$ one has
$$ G(\rG_2,R)=G(\rA_2,R)E^{20}(\rG_2,R). $$
\end{prop}
\begin{proof}
In each one of the first, third and fourth steps of the
procedure described in the proof of Theorem \ref{G2} one now
needs only 1 elementary operation instead of 2. Further,
at the second step inside $\SL(3,R)$ one now needs 5
elementary operations instead of 6.
\end{proof}
We have a similar improvement for all other exceptional
cases, which is new in the number case, and allows one to
improve all known bounds. However, its proof requires
a painstaking tracking of elementary operations in their
minimal representations, and we postpone it to the
sequel of this paper.

%%%%%%%%%%%%%%%%%%%%%%%%%%%%%%G2caseend

\section{Proof of Theorem A in the case of $\rC_2$}\label{secC2}

%Our proof for $\Phi=C_2$ case is an elaboration of the proof from \cite{Tavgen} with some modifications and variations.

\subsection{Notation and stability calculations for $\rC_2$}

Let $G=G(\rC_2,R)$, where $R=\mathbb F_q[t]$. Fix an order on $\Phi$, and let as usual $\Phi^+$ and $\Pi$  be
the sets of positive and fundamental roots, respectively. Then $\Pi=\{\alpha=\epsilon_1-\epsilon_2, \beta=2\epsilon_2\}$ and
$$ \Phi^+=\{\alpha=\epsilon_1-\epsilon_2, \beta=2\epsilon_2, \alpha+\beta=\epsilon_1+\epsilon_2,
2\alpha+\beta=2\epsilon_1\}. $$
\par\noindent
We fix a representation with the highest weight $\mu=\epsilon_1$. So the other weights are
$$ \mu-\alpha =\epsilon_2, \mu - (\alpha+\beta)=-\epsilon_2,
\mu - (2\alpha+\beta)=-\epsilon_1. $$
\par\noindent
Then $G(\rC_2,R)$ is the symplectic group $\Sp(4,R)$ of $4\times 4$-matrices preserving the form
$$ B(x,y)=(x_ly_{-1} - x_{-1}y_1)+(x_2y_{-2} - x_{-2}y_2). $$
\par\noindent Finally, $\alpha$  and $\alpha+\beta$ are short roots while $\beta$ and $2\alpha+\beta$ are long ones.

%llllllllllllllllllllllllllllllllllll

Take an arbitrary matrix
$$A=\left(\begin{array}{cccc} a_{11}&a_{12}&a_{13}&a_{14} \\ a_{21}&a_{22}&a_{23}&a_{24} \\a_{31}&a_{32}&a_{33}&a_{34}  \\a_{41}&a_{42}&a_{43}&a_{44}  \\ \end{array}\right)
\in\Sp(4,R). $$
\par
As we know, any embedding of root systems $\Delta\subset\Phi$ induces a group homomorphism $G(\Delta,R)\to G(\Phi,R)$. Its image will be denoted by $G(\Delta\subset\Phi,R)$.  This can be applied to the special case $\Delta=\{\pm \gamma\}$, $\gamma$ is a root of $\Phi$. We get an embedding $\varphi_\gamma$ of the group $G(\Delta, R)$, which is isomorphic
to $\SL(2,R)$, into the Chevalley group $G(\Phi, R)$.  In this case the image of this
embedding will be denoted by $G^\gamma = G^\gamma (R)$.
\par
Thus, for every root $\gamma\in\rC_2$ we have the subgroup
$G^\gamma (R)=\Sp^\gamma(4,R)$.  In particular,
$$
x_\gamma(\xi)=\varphi_\gamma\left(\begin{array}{cc} 1&\xi \\ 0&1  \\ \end{array}\right), \quad\quad\quad
x_{-\gamma}(\xi)=\varphi_\gamma\left(\begin{array}{cc} 1&0 \\ \xi&1  \\ \end{array}\right).
$$
\par
In this notation, set
$$
A'=\varphi_\beta\left(\begin{array}{cc} a&b \\ c&d  \\ \end{array}\right),\qquad \widetilde A'=\varphi_\alpha\left(\begin{array}{cc} a&b \\ c&d  \\ \end{array}\right),
$$
\noindent
so that the  regular embedding $\rA_1\subset \rC_2$ on the long roots $\beta$ and $-\beta$  gives rise to the matrix

$$A'=\left(\begin{array}{cccc} 1&0&0&0 \\ 0&a&b&0 \\0&c&d&0 \\0&0&0&1  \\ \end{array}\right), $$
\noindent
and the  regular embedding $\widetilde \rA_1\subset \rC_2$ on the short roots $\alpha$ and $-\alpha$  gives rise to the matrix
$$\widetilde A'=\left(\begin{array}{cccc} a&b&0&0 \\ c&d&0&0 \\0&0&a&-b \\0&0&-c&d  \\ \end{array}\right).$$

% The homomorphism of $K_1$-functors $\nu: K_1(A_1,R)\to K_1(C_2,R)$ is surjective under $SR_3$ condition \cite{B},\cite{BMS} which, evidently holds for $R=\mathbb F_q$. So, $G(C_2,R)=E(C_2,R)G(A_1,R)$. The latter equality can be refined as
%$$G(C_2,R)=(U^-(C_2-A_1),R)U^+(C_2-A_1),R))^2G(A_1,R)U^+(C_2-A_1),$$
% where $C_2-A_1$ denotes roots from $C_2$ that  do not belong to $A_1\subset C_2$, see the illuminating description in , p.84. Hence,  the matrix $A$  can be transformed to the matrix

%$$A'=\left(\begin{array}{cccc} 1&0&0&0 \\ 0&a&b&0 \\0&c&d&0 \\0&0&0&1  \\ \end{array}\right)$$
%via a bounded number of elementary transformations in $G(C_2,R)$.

Since we need a bunch of calculations with matrices from $\Sp(4,R)$, we start with some visualization of these calculations. Our main tool is the technique of weight diagrams (see \cite{VP}, \cite{Stein2}). We work with representations with some highest weight $\mu$. In our case the weight diagram of $\rC_2$ type is quite simple:

$$
\begin{tikzpicture}
\draw[thick] (0, 0) circle (0.07) node[above] {$\mu = \varepsilon_1$};
\draw[thick] (1.4, 0) circle (0.07) node[above] {$\varepsilon_2$};
\draw[thick] (2.8, 0) circle (0.07) node[above] {$-\varepsilon_2$};
\draw[thick] (4.2, 0) circle (0.07) node[above] {$-\varepsilon_1$};
\draw[thick] (0.07, 0) to node[midway, below] {$\alpha$} (1.4 - 0.07, 0);
\draw[thick] (1.4 + 0.07, 0) to node[midway, below] {$\beta$} (2.8 - 0.07, 0);
\draw[thick] (2.8 + 0.07, 0) to node[midway, below] {$\alpha$} (4.2 - 0.07, 0);
\end{tikzpicture}
$$

The entries of matrices $g\in G(\rC_2,R)$ are indexed by
pairs of weights, $g=(g_{\lambda_1,\lambda_2})$. %, where $\lambda,\mu=\beta,\ldots,-\beta$.
We concentrate on the first column $g_{*\mu}$ of this matrix, which
is the image of the highest weight vector under the action of $g$. The action of elementary unipotents on the first column of $g$ is depicted on the following self-explaining picture: %which shows how elementary unipotents acts on the columns of matrices from $Sp(4,R)$:
%WWWWWWWWWWWWWWWWWWWWWWWWWWWWWWWWWWWWWWWWWWWWWWWWWWWWWWW
$$
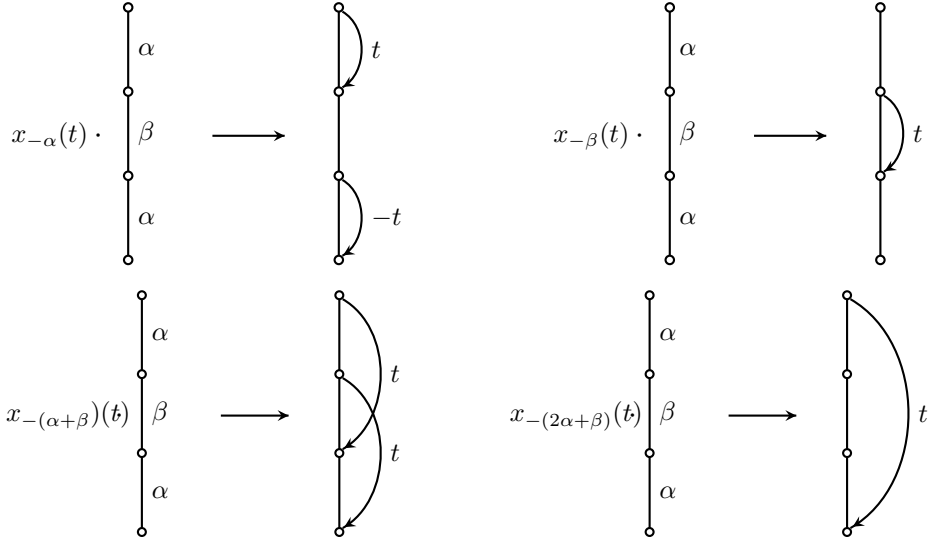
\begin{tikzpicture}[scale=0.8]
\draw[thick] (0, 0) circle (0.07);
\draw[thick] (0, 1.4) circle (0.07);
\draw[thick] (0, 2.8) circle (0.07);
\draw[thick] (0, 4.2) circle (0.07);
\draw[thick] (0, 0.07) to node[midway, right] {$\alpha$} (0, 1.4 - 0.07);
\draw[thick] (0, 1.4 + 0.07) to node[midway, right] {$\beta$} (0, 2.8 - 0.07);
\draw[thick] (0, 2.8 + 0.07) to node[midway, right] {$\alpha$} (0, 4.2 - 0.07);
\node at (-1.3, 4.2*0.5 - 0.07*0.5) {$x_{-\alpha}(t)$};
\node at (-0.5, 4.2*0.5 - 0.07*0.5) {\Large $\cdot$};

\draw[thick] (0 + 3.5, 0) circle (0.07);
\draw[thick] (0 + 3.5, 1.4) circle (0.07);
\draw[thick] (0 + 3.5, 2.8) circle (0.07);
\draw[thick] (0 + 3.5, 4.2) circle (0.07);
\draw[thick] (0 + 3.5, 0.07) to (0 + 3.5, 1.4 - 0.07);
\draw[thick] (0 + 3.5, 1.4 + 0.07) to (0 + 3.5, 2.8 - 0.07);
\draw[thick] (0 + 3.5, 2.8 + 0.07) to (0 + 3.5, 4.2 - 0.07);

\draw[thick, arrow=0.99] (0 + 3.56, 1.4 - 0.07) to[bend left, out=60, in = 180-60] node[midway, right] {$-t$} (0 + 3.56, 0.07);
\draw[thick, arrow=0.99] (0 + 3.56, 4.2 - 0.07) to[bend left, out=60, in = 180-60] node[midway, right] {$ t$} (0 + 3.56, 2.8 + 0.07);

\draw[thick, arrow=1] (1.4, 4.2*0.5 - 0.07*0.5) to (2.6, 4.2*0.5 - 0.07*0.5);

\begin{scope}[xshift=9cm]
\draw[thick] (0, 0) circle (0.07);
\draw[thick] (0, 1.4) circle (0.07);
\draw[thick] (0, 2.8) circle (0.07);
\draw[thick] (0, 4.2) circle (0.07);
\draw[thick] (0, 0.07) to node[midway, right] {$\alpha$} (0, 1.4 - 0.07);
\draw[thick] (0, 1.4 + 0.07) to node[midway, right] {$\beta$} (0, 2.8 - 0.07);
\draw[thick] (0, 2.8 + 0.07) to node[midway, right] {$\alpha$} (0, 4.2 - 0.07);
\node at (-1.3, 4.2*0.5 - 0.07*0.5) {$x_{-\beta}(t)$};
\node at (-0.5, 4.2*0.5 - 0.07*0.5) {\Large $\cdot$};

\draw[thick] (0 + 3.5, 0) circle (0.07);
\draw[thick] (0 + 3.5, 1.4) circle (0.07);
\draw[thick] (0 + 3.5, 2.8) circle (0.07);
\draw[thick] (0 + 3.5, 4.2) circle (0.07);
\draw[thick] (0 + 3.5, 0.07) to (0 + 3.5, 1.4 - 0.07);
\draw[thick] (0 + 3.5, 1.4 + 0.07) to (0 + 3.5, 2.8 - 0.07);
\draw[thick] (0 + 3.5, 2.8 + 0.07) to (0 + 3.5, 4.2 - 0.07);

\draw[thick, arrow=0.99] (0 + 3.56,  2.8 - 0.07) to[bend left, out=60, in = 180-60] node[midway, right] {$ t$} (0 + 3.56, 1.4 + 0.07);

\draw[thick, arrow=1] (1.4, 4.2*0.5 - 0.07*0.5) to (2.6, 4.2*0.5 - 0.07*0.5);
\end{scope}
\end{tikzpicture}
$$
%WWWWWWWWWWWWWWWWWWWWWWWWWWWWWWWWWWWWWWWWWWWWWWWWWWWWWWW
$$
\begin{tikzpicture}[scale=0.75]
\draw[thick] (0, 0) circle (0.07);
\draw[thick] (0, 1.4) circle (0.07);
\draw[thick] (0, 2.8) circle (0.07);
\draw[thick] (0, 4.2) circle (0.07);
\draw[thick] (0, 0.07) to node[midway, right] {$\alpha$} (0, 1.4 - 0.07);
\draw[thick] (0, 1.4 + 0.07) to node[midway, right] {$\beta$} (0, 2.8 - 0.07);
\draw[thick] (0, 2.8 + 0.07) to node[midway, right] {$\alpha$} (0, 4.2 - 0.07);
\node at (-1.3, 4.2*0.5 - 0.07*0.5) {$x_{-(\alpha + \beta})(t)$};
\node at (-0.4, 4.2*0.5 - 0.07*0.5) {\Large $\cdot$};

\draw[thick] (0 + 3.5, 0) circle (0.07);
\draw[thick] (0 + 3.5, 1.4) circle (0.07);
\draw[thick] (0 + 3.5, 2.8) circle (0.07);
\draw[thick] (0 + 3.5, 4.2) circle (0.07);
\draw[thick] (0 + 3.5, 0.07) to (0 + 3.5, 1.4 - 0.07);
\draw[thick] (0 + 3.5, 1.4 + 0.07) to (0 + 3.5, 2.8 - 0.07);
\draw[thick] (0 + 3.5, 2.8 + 0.07) to (0 + 3.5, 4.2 - 0.07);

\draw[thick, arrow=0.99] (0 + 3.56, 2.8 - 0.07) to[bend left, out=60, in = 180-60] node[midway, right] {$ t$} (0 + 3.56, 0.07);
\draw[thick, arrow=0.99] (0 + 3.56, 4.2 - 0.07) to[bend left, out=60, in = 180-60] node[midway, right] {$ t$} (0 + 3.56, 1.4 + 0.07);

\draw[thick, arrow=1] (1.4, 4.2*0.5 - 0.07*0.5) to (2.6, 4.2*0.5 - 0.07*0.5);

\begin{scope}[xshift=9cm]
\draw[thick] (0, 0) circle (0.07);
\draw[thick] (0, 1.4) circle (0.07);
\draw[thick] (0, 2.8) circle (0.07);
\draw[thick] (0, 4.2) circle (0.07);
\draw[thick] (0, 0.07) to node[midway, right] {$\alpha$} (0, 1.4 - 0.07);
\draw[thick] (0, 1.4 + 0.07) to node[midway, right] {$\beta$} (0, 2.8 - 0.07);
\draw[thick] (0, 2.8 + 0.07) to node[midway, right] {$\alpha$} (0, 4.2 - 0.07);
\node at (-1.3, 4.2*0.5 - 0.07*0.5) {$x_{-(2\alpha + \beta)}(t)$};
\node at (-0.3, 4.2*0.5 - 0.07*0.5) {\Large $\cdot$};

\draw[thick] (0 + 3.5, 0) circle (0.07);
\draw[thick] (0 + 3.5, 1.4) circle (0.07);
\draw[thick] (0 + 3.5, 2.8) circle (0.07);
\draw[thick] (0 + 3.5, 4.2) circle (0.07);
\draw[thick] (0 + 3.5, 0.07) to (0 + 3.5, 1.4 - 0.07);
\draw[thick] (0 + 3.5, 1.4 + 0.07) to (0 + 3.5, 2.8 - 0.07);
\draw[thick] (0 + 3.5, 2.8 + 0.07) to (0 + 3.5, 4.2 - 0.07);

\draw[thick, arrow=0.99] (0 + 3.56,  4.2 - 0.07) to[bend left, out=60, in = 180-60] node[midway, right] {$ t$} (0 + 3.56, 0.07);

\draw[thick, arrow=1] (1.4, 4.2*0.5 - 0.07*0.5) to (2.6, 4.2*0.5 - 0.07*0.5);
\end{scope}
\end{tikzpicture}
$$

%yyyyyyyyyyyyyyyyyyyyyyyyyyyyyyyyyyyyyyyyyyyyyyyyyyyyyyyyyyyyyyyyyyyyyyyyyy

%For
%typographical reasons, we denote this column by
%$$ (x_\beta,x_{-\gamma},x_{-\alpha},x_0,x_\alpha,x_\gamma,x_{-\beta}). $$
%\noindent
%It is our intention to reduce this column to the form $(1,*,*,*,*,*,*)$ by
%elementary unipotents.
%\par
%This can be done as follows. Not to proligerate indices in this calculation,
%we will not {\it rename\/} [as mathematicians would do], but {\it reset\/}
%[as is typical in programming] our variables $g$ and $x$, still denoting them
%by the same letters after each successive transformation.
%\par\smallskip

%One can check that by stable calculation (cf. \cite{Stein2})

\begin{lemma}\label{stabC2} A matrix $A$ in $G(\rC_2,R)$ can be moved to $A'$ in $G(\rA_1\subset \rC_2,R)$ by $\leq 10$ elementary transformations.
\end{lemma}
\begin{proof}
Recall that $A'$ in $G(\rA_1\subset \rC_2,R)$ is the image of
$A\in G(\rA_1,R)$ embedded into $G(\rC_2,R)$ on long roots.  We  denote elements of the first column of $A\in \Sp(4,R)$ lexicographically, $x=(x_1,x_2,x_3,x_4)$. %,x_\alpha,x_\gamma,x_{-\beta}). $$
Let $I=\langle x_1, x_2,x_3\rangle$ be the ideal generated by the first 3 entries of $x$. Thus, $I+\langle x_4\rangle=R$.
\par
Since $R$ is a Dedekind ring, there exists $t\in R$ such that in $x'=x_{-\alpha}(t) x$ the ideal $I$ is generated by two entries, $I=\langle x'_2,x'_3\rangle$, see Lemma \ref{Dedek}. Note that $\langle x_4\rangle\equiv\langle x'_4\rangle \pmod I.$  The first column of $x'=x_{-\alpha}(t) x$  is unimodular, and we have
$$ \langle x'_2,x'_3,x'_4\rangle=I+\langle x'_4\rangle=R. $$
\par\noindent
Then there exist $t_1, t_2,t_3\in R$ such that in $x_\alpha(t_1)x_{\alpha+\beta}(t_2)x_{2\alpha+\beta}(t_3)x'$ we obtain the first column of the form $x=(1,*,*,*)$,
(cf. \cite{Stein2}).
\par
Having an invertible element in the NW corner of the matrix, it remains to  make three elementary moves downstairs and three left-to-right elementary moves to get zeros in the first column and the first row. Thus we  transformed $A$ to $A'$ by $10=1+3+3+3$ elementary transformations in total.
\end{proof}

%%%%%%%%%%%%%%%%%%%%%%%%%%%%%
%%%%%%%%%%%%%%%%%%%%%%%%%%%%%
%%%%%%%%%%%%%%%%%%%%%%%%%%%%%%%%

\subsection{Extracting roots of Mennicke symbols}
Our goal is to prove, in the function field case, that one can extract
$m^{th}$ roots of Mennicke symbols. This is an essential ingredient in
performing elementary operations below, see Lemmas \ref{square1} and \ref{lem:long-short},
which can only be applied when one of the matrix entries is a square.
\par
The previous stability argument was quite general.
Below, %% As throughout the paper,
we restrict our attention to the particular case
of the base ring $R=\mathcal O=\mathbb F_q[t]$.

Actually, we only need the case $m=2$. We shall proceed along a more general
way of reasoning. Namely, we shall first establish the statement in the case $m=q-1$.
If $q$ is odd, the case $m=2$ follows:
after extracting an $m^{th}$ root, we can then raise to the $(m/2)^{th}$ power to get
a square root. So we first assume that $q$ is odd and $m=q-1$, leaving the problem of
extracting square roots in the case of characteristic 2 for separate consideration.
\par
Let us fix some notation.
Denote $K=\mathbb F_q(t)$. Let $\pt$ be the infinite place of $K$, it corresponds to
the valuation $v_{\infty}$ of $\mathcal O$ given by
$$ v_{\infty}(f)=-\deg f. $$
\par\noindent
This valuation naturally extends to $K$ by setting $v_{\infty}(f/g)=\deg g-\deg f$.
For the completion of $K$ at this place we have
$$ K_{v_{\infty}}=\mathbb F_q((1/t)), $$
\par\noindent
the field of Laurent series in $1/t$. For brevity, we denote this field by $K_{\infty}$.
Let $\mathcal O_{\infty}=\mathbb F_q[[1/t]]$ denote its ring of integers, it is
a discrete valuation ring with maximal ideal $\pt=(1/t)$ and residue field $\mathbb F_q$.
The residue of $f_0=a_0+a_{-1}/t+\dots \in\mathcal O_{\infty}$ equals $a_0$. If $f,g\in \mathbb F_q[t]$
are polynomials of the same degree, the residue of $f/g$ is equal to the ratio of their leading
coefficients.
\par
As mentioned above, we first consider the case where $q$ is odd and $m=q-1$.
\par
We start with the following observation on extracting roots in $K_{\infty}$.
\par
\begin{observation} (cf. \cite{SE}) \label{ob:root}
Given $f\in K_{\infty}$ with leading term $a_Mx^M$, $f$ is an $m^{th}$ power
if and only if $M$ is divisible by $m$ and $a_M$ has an
$m^{th}$ root in $\mathbb F_q$.
\par
Indeed, suppose that $f=g^m$, then $v_{\infty}(f)=mv_{\infty}(g)$, so that
$-\deg(f)=-m\deg(g)$ and $m$ divides $M$. Write $f=x^Mf_0$ where $f_0=a_M+a_{M-1}/t+\dots$,
then $f_0=g_0^m$ for some $g_0\in \mathcal O_{\infty}$, $g_0=b_0+b_{-1}/t+\dots$. Taking residues
modulo $\pt$, we get $a_M=b_0^m$.
\par
Conversely, write $f=x^Mf_0$, where $f_0=a_M+a_{M-1}/t+\dots$, and suppose that $m$ divides $M$
and $a_M$ is an $m^{th}$ power in $\mathbb F_q$.
Then the polynomial $x^m-a_M\in\mathbb F_q[x]$ has a root in $\mathbb F_q$, and as $m=q-1$ is prime to the
characteristic of $\mathbb F_q$, this root is simple. Hence by Hensel's lemma, it lifts to
a root of the polynomial $x^m-f_0\in\mathcal O_{\infty}[x]$, which belongs to $K_{\infty}$. Therefore
$f_0$ is an $m^{th}$ power in $K_{\infty}$, hence so is $f$. \qed
\end{observation}

The subsequent arguments are mainly based on combining two powerful classic tools: algebraic,
the $m^{th}$ power reciprocity law, and analytic,
(generalized) Dirichlet's theorem on primes in arithmetic progressions,
as in \cite{CaKe1} (and also \cite{BMS}, \cite{Mi}).
\par
More precisely, we use the Kornblum--Artin
version of Dirichlet's theorem:

\begin{theorem} {\rm(}Kornblum--Artin,
\cite[Theorem 4.8]{Ros}\/{\rm)} \label{kornblum}
Let  $a$, $b$ be relatively prime polynomials in $\mathcal O=\mathbb F_q[t]$, $\deg a >0$. Then
there are infinitely many monic irreducible polynomials $b'$ congruent to $b$ modulo $a\mathcal O$. Moreover, such $b'$ can be
of arbitrary degree $N$, provided $N$ is sufficiently large.
\end{theorem}

The reciprocity law we use in our set-up can be formulated as a product formula
for local residue $m^{th}$ power symbols %(usually called Hilbert symbols)
\begin{equation} \label{prodformula}
\prod_{\mathfrak p}\left(\frac{\alpha,\beta}{\mathfrak p}\right)_m=1.
\end{equation}
Here $\alpha, \beta\in K^*$ are fixed, and $\mathfrak p$ runs over all places of $K$. For computations below,
we use an explicit formula by Hermann Ludwig Schmid,
see, e.g. formula (27) in \cite{Roq}:
\begin{equation} \label{Roq-gen}
\left(\frac{\alpha,\beta}{\mathfrak p}\right)_m = N_{\mathfrak p}\left((-1)^{ab} \frac{\alpha^b}{\beta^a}\left(\mathfrak p\right)\right)^{\frac{q-1}{m}},
\end{equation}
where $a=v_{\mathfrak p}(\alpha)$, $b=v_{\mathfrak p}(\beta)$,
$f(\mathfrak p)$ stands for the image of $f\in K$ in the residue field $\kappa(\mathfrak p)$ of $\mathfrak p$,
and $N_{\mathfrak p}$ is the norm map from $\kappa(\mathfrak p)$ to $\mathbb F_q$.
(The expression raised to the power $(q-1)/m$ in formula \eqref{Roq-gen} is usually called tame symbol.)

The power residue symbol takes values in the group of $m^{th}$ roots of 1 (which is clear from the right-hand side
of formula \eqref{Roq-gen}). From the same formula it is clear that for all but finitely many $\mathfrak p$
these values are equal to 1 (namely, for those with $v_{\mathfrak p}(\alpha )=v_{\mathfrak p}(\beta )=0$).
It is well known that this symbol is bimultiplicative.

%Since $A$ is unimodular, we may assume that at least one of these elements is not divisible by $t$.
%Let it be $b_1$.

We are now ready to state and prove an arithmetic lemma which allows us to perform
the needed elementary transformations below. It is completely
parallel (in the statement and in the proof) to Lemma 3 of \cite{CaKe1}.

\begin{lemma}\label{mainn1}
Let $G=\SL_2(\mathcal O)$. Let $m$ be either $q-1$ or $2$. Then for any
$
A=\begin{pmatrix}a_1 & b_1 \\ c_1 & d_1 \end{pmatrix}
\in G
$
there exists
$
A'=\begin{pmatrix}a^m & b \\ c & d \end{pmatrix}
\in G
$
elementarily equivalent to $A$.
\end{lemma}

\begin{proof}
We combine the proof of Lemma 3 in \cite{CaKe1} with some facts from \cite{BMS}.
We closely follow the arguments and the notation of \cite{CaKe1}.\footnote{There is an exception: in \cite{CaKe1}
the term `local units' is used for calling nonzero elements of the local field $K_v$, where $v$ is a place of $K$.
We avoid using such a terminology because `local unit' commonly refers to an invertible element of the valuation ring $O_v$.}

As mentioned above, we first assume that $q$ is odd and $m=q-1$.

As in \cite{CaKe1}, we may assume that the elements of the first row of $A$ are nonzero. Indeed, if, say, $a_1=0$, then
$b_1$ is a nonzero constant, and hence $A$ is elementarily equivalent to $A'$ with $a=1$ (note that 1 is an $m^{th}$
power in $\mathbb F_q$.)

For reader's convenience, we break the proof into several short steps and emphasize the conclusive part of each step
by putting in it boldface.
\par\bigskip
$\bullet$ {\bf Step 1}.
{\bf{One can choose}} ${\mathbf{u,w\in K^*_{\infty}}}$ {\bf{so that the}}
${\mathbf{m^{th}}}$ {\bf{local residue at}} $\mathbf\pt$
$$
{\mathbf{\zeta = \left(\frac{u,w}{\pt}\right)_m}}
$$
{\bf{is a primitive}} ${\mathbf{m^{th}}}$ {\bf{root of 1}}.

This follows from the fact that the residue symbol is non-degenerate,
see, e.g. the proof of Case 1 of Theorem 3.5 in \cite{BMS}. In our set-up,
one can argue in a more straightforward way, using formula  \eqref{Roq-gen}.
Under our assumptions, this formula reduces to

\begin{equation} \label{Roq}
\left(\frac{u,w}{\pt}\right)_m = (-1)^{\deg u \deg w} \frac{u^{-\deg w}}{w^{-\deg u}}\left(\pt\right).
\end{equation}
%where $f(\pt )$ stands for the image of $f$ in the residue field $\mathbb F_q$ of $\pt$.
(Note that the numerator and denominator of the fraction appearing in formula \eqref{Roq}
are polynomials of the same degree, hence its residue is well defined and equals the ratio
of their leading coefficients.)

Hence one can choose degree one polynomials $u=u_0+u_1t$ and $w=w_0+w_1t$ such
that $-w_1/u_1$ is a primitive element
of $\mathbb F_q$. Say, let us choose $w=-1+t$ and $u_1$ a primitive element
of $\mathbb F_q$.
\par\bigskip
$\bullet$ {\bf Step 2}. Consider the arithmetic progression $\{a_1+b_1\mathcal O\}$. By Theorem \ref{kornblum},
it contains a monic irreducible
polynomial $a_2=t^d+\alpha_{d-1}t^{d-1}+\dots $ of sufficiently large degree $d$ such that
\begin{equation} \label{eq2}
d\equiv 1 \pmod m.
\end{equation}
%Indeed, by our assumption $b_1$ is not divisible by $t$, so that the polynomials $b_1$ and $t$ are coprime, and hence one can
%combine the Dirichlet theorem with the Chinese remainder theorem. Note in addition that $a_1$ and $b_1$ are coprime because
%the matrix $A$ is unimodular, therefore also $a_2$ and $b_1$ are coprime because $a_2\equiv a_1 \Mod{b_1}$.

With our choice of $w=-1+t$, we have
$$
\frac{1}{w}=\frac{1}{-1+t}=\frac{1}{t(1-t^{-1})}=\frac{1}{t}\left(1+t^{-1}+t^{-2}+\dots \right),
$$
so that
$$
a_2/w=t^{d-1}(1+\alpha_{d-1}t^{-1}+\dots )\left(1+t^{-1}+t^{-2}+\dots \right).
$$
Combining congruence \eqref{eq2} with Observation \ref{ob:root} and noticing
that 1 is an $m^{th}$ power in $\mathbb F_q$ for $m=q-1$, we conclude that
${\mathbf{a_2/w}}$ {\bf{is an}} ${\mathbf{m^{th}}}$ {\bf{power in}} ${\mathbf{K_{\infty}}}$.
\par\bigskip
$\bullet$ {\bf Step 3}. We have
$$
\left(\frac{u,a_2}{\pt}\right)_m=\left(\frac{u,a_2/w}{\pt}\right)_m\cdot \left(\frac{u,w}{\pt}\right)_m=1\cdot \left(\frac{u,w}{\pt}\right)_m=\zeta.
$$
The first equality follows from the multiplicativity of the power residue symbol, and the second equality is a consequence of the choice of $a_2$
made at Step 2. (Recall that if one of the components of the symbol is an $m^{th}$ power, the symbol equals 1.)

Thus, $\displaystyle{\mathbf{\left(\frac{u,a_2}{\pt}\right)_m}}$ {\bf{is a primitive}} ${\mathbf{m^{th}}}$ {\bf{root of 1}} (see Step 1).
\par\bigskip
$\bullet$ {\bf Step 4}. Since by Step 3 the symbol $\left(\frac{u,a_2}{\pt}\right)_m$ is a primitive $m^{th}$ root of 1,
its powers take all nonzero values in $\mathbb F_q$.
%Since $a_2$ and $b_1$ are coprime (see Step 2), the symbol
%$\left(\frac{b_1,a_2}{a_2\mathcal O}\right)_m$ is not equal to zero.
Hence there exists $k$ such that
$$
\left(\frac{u,a_2}{\pt}\right)_m^k = \left(\frac{b_1,a_2}{a_2\mathcal O}\right)_m^{-1},
$$
i.e. we have
\begin{equation} \label{eq3}
 \left(\frac{b_1,a_2}{a_2\mathcal O}\right)_m\cdot \left(\frac{u,a_2}{\pt}\right)_m^k =1.
\end{equation}
Note that if necessary, we can replace $k$ by any larger integer $k'$ congruent to $k$ modulo $m$,
and equality \eqref{eq3} will remain valid. {\bf{So we set}} ${\mathbf{s=u^k}}$ {\bf{and assume that}} $\mathbf k$ {\bf{is large enough}}.
\par\bigskip
$\bullet$ {\bf Step 5}. Using Theorem \ref{kornblum} once again, choose an irreducible polynomial
$b$ of degree $k=\deg s$ such that

\begin{equation} \label{eq4}
b\equiv b_1 \Mod{a_2\mathcal O }.
\end{equation}

On multiplying $b$ by a nonzero constant, we can equalize the leading coefficients
of the polynomials $b$ and $s$. Thus in the sequel we may and shall assume that
$\mathbf b$ {\bf{and}} $\mathbf s$ {\bf{have the same degree and the same leading coefficient}}.

%\begin{equation} \label{eq5}
%b\equiv s \Mod t,
%\end{equation}
%where $s$ is the polynomial constructed at Step 4.
\par\bigskip
$\bullet$ {\bf Step 6}. Since the polynomials $b$ and $a_2$ are irreducible, the $m^{th}$ power reciprocity law
reduces to the equality

\begin{equation}  \label{eq6}
\left(\frac{b,a_2}{b\mathcal O}\right)_m \cdot \left(\frac{b,a_2}{a_2\mathcal O}\right)_m \cdot \left(\frac{b,a_2}{\pt}\right)_m = 1
\end{equation}
(all other symbols $\displaystyle\left(\frac{b,a_2}{\mathfrak p}\right)_m$ are equal to 1 because $v_{\mathfrak p}(b)=v_{\mathfrak p}(a_2)=0$).

Let us show that the product of the second and third factors equals 1.

Looking at the second factor, we note that by congruence \eqref{eq4},
$$
\left(\frac{b,a_2}{a_2\mathcal O}\right)_m = \left(\frac{b_1,a_2}{a_2\mathcal O}\right)_m
$$
\noindent
(use formula \eqref{Roq-gen}). As to the third factor, it is equal
to $\displaystyle\left(\frac{s,a_2}{\pt}\right)_m$ because
the polynomials $b$ and $s$ are chosen at Step 5 so that they have the same degree and the same leading coefficient,
and hence by formula \eqref{Roq} the corresponding symbols coincide. By the choice of $s$ made at Step 4, we have $s=u^k$, so that
by the multiplicativity of the residue symbol we have
$$
\left(\frac{s,a_2}{\pt}\right)_m=\left(\frac{u,a_2}{\pt}\right)_m^k,
$$
and we finish by applying \eqref{eq3}.

Thus \eqref{eq6} gives $\displaystyle\left(\frac{b,a_2}{b\mathcal O}\right)_m=1$. Swapping components of the symbol inverts its value, hence also
\begin{equation} \label{eq7}
{\mathbf{\left(\frac{a_2,b}{b\mathcal O}\right)_m=1}}.
\end{equation}
\par\bigskip
$\bullet$ {\bf Step 7}. As both $b$ and $a_2$ are irreducible, from \eqref{eq7} we conclude that $a_2$ is
an $m^{th}$ power modulo $b$, i.e. there exists $a$ such that
\begin{equation} \label{eq8}
{\mathbf{a^m\equiv a_2 \Mod b}}.
\end{equation}
\par\bigskip
$\bullet$ {\bf Step 8}. The choices made for $a_2$ at Step 2 and  for $b$ at Step 5, together with congruence \eqref{eq8}
obtained at Step 7, allow one to prove the lemma by three elementary operations:
$$
\begin{pmatrix}a_1 & b_1 \\ c_1 & d_1 \end{pmatrix} \to \begin{pmatrix}a_2 & b_1 \\ * & * \end{pmatrix} \to
\begin{pmatrix}a_2 & b \\ * & * \end{pmatrix} \to \begin{pmatrix}a^m & b \\ * & * \end{pmatrix}.
$$

This finishes the proof in the case where $q$ is odd.

Suppose now that $q$ is a power of 2. In this case, extracting $(q-1)^{th}$ roots of Mennicke symbols
can be done in exactly the same way. %(note that in the case $q=2$, which could be problematic for Step 1
%because of the absence of nontrivial primitive elements, we have $q-1=1$, and the problem of extracting
%roots becomes void).

So we only have to consider the problem of extracting square roots. In characteristic 2, this is easy.
Indeed, if a polynomial $f\in \mathbb F_q[t]$ is irreducible, any $g\in  \mathbb F_q[t]$ is a square
modulo $f$ because its image $\bar g$ in the field $\mathbb F_q[t]/(f)$ of characteristic 2 is a square, as
any other element of a finite field of characteristic 2. Thus it is enough to implement Steps 2 and 5 of the first
part of the proof, only taking care of the irreducibility of $a_2$ and $b$.
\end{proof}

%%%%%%%%%%%%%%%%%%%%%%%%%%%%%%

\begin{remark} \label{corners}
If needed, one can arrange the $m^{th}$ power in the NE
corner of $A'$ instead of the upper-left one,
without additional elementary operations.
\end{remark}

\begin{remark} \label{bottomrow}
If needed, one can arrange an irreducible polynomial not only in the NE corner of $A'$
but also in the lower-left one (at the expense of the fourth elementary operation).
Indeed, as the matrix $A'$ is unimodular, the entries $a^m$ and $c$
of its left column are coprime, and one can apply the Kornblum--Artin theorem to the arithmetic progression
$\{c+a^m\mathcal O\}$ to find an irreducible $c'$ congruent to $c$ modulo $a^m$. On adding an appropriate
multiple of the first row to the second one provides the needed irreducible polynomial $c'$ in the SW corner.
\end{remark}
%% \end{proof}

\subsection{Swindling Lemma for ${\widetilde{\rA}}_1\subset\rC_2$}
The following lemmas are headed towards Proposition 5.10,
which is a symplectic analogue of the swindling lemma by
Nica \cite{Nic} for the {\it short root\/} embedding of a matrix
$A\in \SL(2,R)$ into $\Sp(4,R)$.
\par
We start with the following symplectic analogue of the
swindling lemma for the {\it long root\/} embedding.
It is weaker than what we actually need, since here we
can only move squares. These calculations are purely
formal, here $R$ is an arbitrary commutative ring.

\begin{lemma}\label{sw_long}
Let $a,b,c,d,s\in R$, $ad-bcs^2=1$  and, moreover,
$a\equiv d\equiv 1\pmod s$. Then
$$ \phi_{\beta}
\begin{pmatrix} a&b\\ cs^2&d\\ \end{pmatrix}
\quad\text{can be moved to}\quad
 \phi_{2\alpha+\beta}
\begin{pmatrix} d&-c\\ -bs^2&a\\ \end{pmatrix} $$
\noindent
by $8$ elementary transformations.
\end{lemma}
\begin{proof}
Specifically, let
$a=1+st$, for some $t\in R$.
Start with a matrix

$$ A=\phi_{\beta}
\begin{pmatrix} a&b\\ cs^2&d\\ \end{pmatrix}=\begin{pmatrix}
1&0&0&0\\ 0&a&b&0\\ 0&cs^2&d&0\\ 0&0&0&1\\
\end{pmatrix}\in\SL(2,R)\le\Sp(4,R). $$
\noindent

\par\bigskip
$\bullet$ {\bf Step 1}
$$ A=Ax_{-(\alpha+\beta)}(s)=
\begin{pmatrix}
1&0&0&0\\ bs&a&b&0\\ ds&cs^2&d&0\\ 0&s&0&1\\
\end{pmatrix} $$

\par\bigskip
$\bullet$ {\bf Step 2}
$$ A=x_{\alpha}(cs)A=
\begin{pmatrix}
1+bcs^2&acs&bcs&0\\ bs&a&b&0\\ ds&0&d&-cs\\ 0&s&0&1\\
\end{pmatrix} $$

\par\bigskip
$\bullet$ {\bf Step 3}
$$ A=x_{\alpha+\beta}(-t)A=
\begin{pmatrix}
d&acs&bcs-dt&cst\\ bs&1&b&-t\\ ds&0&d&-cs\\ 0&s&0&1\\
\end{pmatrix} $$

\par\bigskip
$\bullet$ {\bf Step 4}
$$ A=Ax_{-\alpha}(-bs)=
\begin{pmatrix}
d-abcs^2&acs&bcs-dt+bcs^2t&cst\\ 0&1&b-bst&-t\\
ds&0&d-bcs^2&-cs\\ -bs^2&s&bs&1\\
\end{pmatrix} $$

\par\bigskip
$\bullet$ {\bf Step 5}
$$ A=Ax_{\beta}(-b+bst)=
\begin{pmatrix}
d-abcs^2&acs&-dt+abcs^2t&cst\\ 0&1&0&-t\\
ds&0&d-bcs^2&-cs\\ -bs^2&s&bs^2t&1\\
\end{pmatrix} $$

\par\bigskip
$\bullet$ {\bf Step 6}
$$ A=Ax_{\alpha+\beta}(t)=
\begin{pmatrix}
d-abcs^2&acs&0&(1+a)cst\\ 0&1&0&0\\
ds&0&1&-cs\\ -bs^2&s&0&a\\
\end{pmatrix} $$

\par\bigskip
$\bullet$ {\bf Step 7}
$$ A=x_{2\alpha+\beta}(-ac)A=
\begin{pmatrix}
d&0&0&-c\\ 0&1&0&0\\
ds&0&1&-cs\\ -bs^2&s&0&a\\
\end{pmatrix} $$

\par\bigskip
$\bullet$ {\bf Step 8}
$$ A=x_{-(\alpha+\beta)}(-s)g=
\begin{pmatrix}
d&0&0&-c\\ 0&1&0&0\\
0&0&1&0\\ -bs^2&0&0&a\\
\end{pmatrix} $$
\end{proof}
The following lemma is an explicit version of Bass--Milnor--Serre,
\cite{BMS}, Lemma 13.3. It expresses one of the [various!] multiplicativity properties of Mennicke symbols in the symplectic case. We use it here, since it is cheaper than other such
multiplicativity properties, in terms of the number of elementary moves.
%Here $\gamma$ is the short fundamental
%root of $\rm C_2$.

\begin{lemma}\label{bms1}
Let $a,b,c,d,x,y,z\in R$, $ad-bc=1$ and $az-xy=1$. Then
$$
\phi_{\alpha}\begin{pmatrix}
a&b\\ c&d\\
\end{pmatrix}
\phi_{\beta}\begin{pmatrix}
a&x\\y&z\\
\end{pmatrix}=
\begin{pmatrix}
a&b&0&0\\ c&d&0&0\\ 0&0&a&-b\\ 0&0&-c&d\\
\end{pmatrix}\cdot
\begin{pmatrix}
1&0&0&0\\ 0&a&x&0\\ 0&y&z&0\\ 0&0&0&1\\
\end{pmatrix}
$$
\noindent
can be moved to
$$
\phi_{2\alpha+\beta}\begin{pmatrix}
a&b^2x\\ c^2y&d(1-bc)+b^2c^2z\\
\end{pmatrix} =
\begin{pmatrix}
a&0&0&b^2x\\ 0&1&0&0\\
0&0&1&0\\
c^2y&0&0&d(1-bc)+b^2c^2z\\
\end{pmatrix} $$
\noindent
by $6$ elementary transformations.
\end{lemma}

\begin{proof}
The product we start with  equals
$$ A=\begin{pmatrix}
a&ab&bx&0\\ c&ad&dx&0\\ 0&ay&az&-b\\ 0&-cy&-cz&d\\
\end{pmatrix} $$

\par\bigskip
$\bullet$ {\bf Step 1}
$$ A=Ax_{\alpha}(-b)=
\begin{pmatrix}
a&0&bx&b^2x\\ c&1&dx&bdx\\ 0&ay&1+xy&bxy\\ 0&-cy&-cz&d-bcz\\
\end{pmatrix} $$

\par\bigskip
$\bullet$ {\bf Step 2}
$$ A=Ax_{-\alpha}(-c)=
\begin{pmatrix}
a&0&abdx&b^2x\\ 0&1&ad^2x&bdx\\
-acy&ay&1+adxy&bxy\\ c^2y&-cy&-cdxy&d-bcz\\
\end{pmatrix} $$

\par\bigskip
$\bullet$ {\bf Step 3}
$$ A=Ax_{\beta}(-ad^2x)=
\begin{pmatrix}
a&0&abdx&b^2x\\ 0&1&0&bdx\\
-acy&ay&1-abcdxy&bxy\\ c^2y&-cy&bc^2dxy&d-bcz\\
\end{pmatrix} $$

\par\bigskip
$\bullet$ {\bf Step 4}
$$ A=Ax_{\alpha+\beta}(-bdx)=
\begin{pmatrix}
a&0&0&b^2x\\ 0&1&0&0\\
-acy&ay&1&-b^2cxy\\ c^2y&-cy&0&d(1-bc)+b^2c^2z\\
\end{pmatrix} $$

\par\bigskip
$\bullet$ {\bf Step 5}
$$ A=x_{-(\alpha+\beta)}(cy)A=
\begin{pmatrix}
a&0&0&b^2x\\ 0&1&0&0\\
0&ay&1&0\\
c^2y&0&0&d(1-bc)+b^2c^2z\\
\end{pmatrix} $$

\par\bigskip
$\bullet$ {\bf Step 6}

$$ A=x_{\beta}(-ay)A=
\begin{pmatrix}
a&0&0&b^2x\\ 0&1&0&0\\
0&0&1&0\\
c^2y&0&0&d(1-bc)+b^2c^2z\\
\end{pmatrix} $$
\end{proof}

\begin{lemma}\label{sh}
Let $a,b,c,d\in R$, $ad-bc=1$. Then
$$
\phi_{\alpha}\begin{pmatrix}
a&b\\ c&d\\
\end{pmatrix}=
\begin{pmatrix}
a&b&0&0\\ c&d&0&0\\ 0&0&a&-b\\ 0&0&-c&d\\
\end{pmatrix}
$$
\noindent
can be moved to
$$
\phi_{(2\alpha+\beta)}\begin{pmatrix}
a&b^2\\ -c^2&d(1-bc)\\
\end{pmatrix} =
\begin{pmatrix}
a&0&0&b^2\\ 0&1&0&0\\
0&0&1&0\\
-c^2&0&0&d(1-bc)\\
\end{pmatrix} $$
\noindent
by not more than $9$ elementary transformations.
\end{lemma}
\begin{proof}
In the previous lemma, take
$$ \begin{pmatrix} a&x\\ y&z\\ \end{pmatrix}=
\begin{pmatrix}
a&1\\-1&0\\ \end{pmatrix}. $$
\noindent
This last matrix is a product of 3 elementary transformations
in $\SL_2$,
$$ \begin{pmatrix} a&1\\-1&0\\ \end{pmatrix}=
t_{21}(-1)t_{12}(1)t_{21}(a-1)=
t_{12}(1-a)t_{21}(-1)t_{12}(1), $$
\noindent
summing up to $6+3=9$.
\end{proof}

Now, we are all set to derive from Lemmas \ref{sw_long} and \ref{sh} a life-size symplectic analogue of the {\it swindling lemma} by Nica \cite{Nic} for {\it short roots\/}.

\begin{prop}\label{sw_sh}
Let $a,b,c,d,s\in R$, $ad-bcs=1$  and, moreover,
$a\equiv d\pmod s$. Then
$$ \phi_{\alpha}
\begin{pmatrix} a&b\\ cs&d\\ \end{pmatrix}
\quad\text{can be moved to}\quad
 \phi_{\alpha}
\begin{pmatrix} d&c\\ bs&a\\ \end{pmatrix} $$
\noindent
by not more than $26$ elementary transformations.
\end{prop}

\begin{proof}
By Lemma \ref{sh}
$$ \phi_{\alpha}
\begin{pmatrix} a&b\\ cs&d\\ \end{pmatrix}
\quad\text{can be moved to}\quad
 \phi_{2\alpha+\beta}
\begin{pmatrix} a&b^2\\ -c^2s^2&d(1-bcs)\\ \end{pmatrix} $$
\noindent
by not more than 9 elementary operations.
\par
Now we can apply Lemma \ref{sw_long} % --- or its analogue with
%$a\equiv d\equiv -1\pmod s$ ---
to transform the latter matrix to
the matrix of the form $\phi_{-\beta}$
$$  \phi_{\beta}
\begin{pmatrix} d(1-bcs)&c^2\\ -b^2s^2&a\\ \end{pmatrix}= \phi_{-\beta}\begin{pmatrix} a&b^2s^2\\ -c^2&d(1-bcs)\\ \end{pmatrix} $$
\noindent
by 8 elementary operations.
\par
Note that switching the first column with the second one
is nothing else than replacing $\beta$ by $-\beta$. Inside
$\SL_2$, such a replacement amounts to the conjugation
by $w_{\beta}$. However, with respect to the embedding
of $\SL_2$ into $\Sp_4$, it is just a different parametrisation, which gives {\it the same\/} matrix in $\Sp_4$, so that no additional elementary moves are needed.
\par
The angle between $\alpha$ and $2\alpha+\beta$ is the
samr as the angle between $-\beta$ and $\alpha$. Thus,
we can apply %[the analogue of]
Lemma \ref{sh} once more, and get the desired matrix by
not more than 9 further elementary operations:
$$
 \phi_{-\beta}\begin{pmatrix} a&b^2s^2\\ -c^2&d(1-bcs)\\ \end{pmatrix}\longrightarrow \phi_{\alpha}
\begin{pmatrix} d&c\\ bs&a\\ \end{pmatrix}
 $$
\noindent
%---
%for the second column instead of the first row ---
Altogether, we have expended not more than
$8 + 9 + 9 = 26$ elementary moves.
\end{proof}

\noindent
\begin{remark} The above lemmas allow numerous
releases.
\par\smallskip
$\bullet$ To replace columns by rows, you transpose
all factors: the transpose of an elementary move is
again an elementary move.
%\par\smallskip
%$\bullet$ To replace the first column by the second one
%means simply to replace $\beta$ by $-\beta$, which
%means conjugating the whole calculation by $w_{\beta}$.
%Again, a conjugate of an elementary operation by a Weyl
%group element is an elementary operation.
\par\smallskip
$\bullet$ More interestingly, one can switch $a$ and $b$
in, say, Lemma \ref{sh}, thus reducing
$$ \phi_{\gamma}\begin{pmatrix}
a&b\\ c&d\\
\end{pmatrix}\quad\text{to the form}\quad
\phi_{\beta}\begin{pmatrix}
a^2&b\\c(1+ad)&d^2\\
\end{pmatrix}. $$
\noindent
However, this is {\bf not} a conjugation. It amounts to a multiplication by a Weyl group element on the left, and by
{\it another\/} Weyl group element on the right! Such
transformations are still elementary, of course, but they
may affect the length.
\end{remark}
%\par\smallskip
%Should we state all such lemmas?
%\par\smallskip
%\noindent
\begin{remark} We believe that the estimate in this lemma
might be {\bf grossly} exaggerated. In the above proof we
switched between the short root and the long root positions.
We would expect that by by implementing swindling in place,
the number of
elementary operations here could be reduced to something
like 7, 8 or 9.
\end{remark}
%\par
%Should we do {\it all of it\/} now, or just select the shortest
%path with {\it some\/} plausible bound, to return to this
%later? Here, I see {\it two\/} immediate improvements.
%\par\smallskip
%$\bullet$ Better stability for Dedekind rings --- never done
%before, apart from the cases of $\SL_n$ and $\Sp_{2l}$.
%\par\smallskip
%$\bullet$ Better bounds for $\Sp_4$ and $\Sp_{2l}$, $l\ge 3$.
%\par\smallskip
%This might constitute one or two sequels.

%555555555555555555555555555555555555555555555555555555555555555555555555555
\subsection{Bounded elementary generation for $\Sp(4, \mathbb F_q[t])$}
We start with a matrix
$$ A=\begin{pmatrix}a&b\\c&d \\ \end{pmatrix}\in\SL(2,R) $$
\noindent
embedded into $\Sp(4,R)$ on the {\it long root\/} position $A\in G(\rA_1\subset \rC_2,R)$, as in Lemma \ref{stabC2}:
$$
A=\varphi_\beta\left(\begin{array}{cc} a&b \\ c&d  \\ \end{array}\right) = \left(\begin{array}{cccc} 1&0&0&0 \\ 0&a&b&0  \\0&c&d&0\\0&0&0&1\\ \end{array}\right) ,
$$
\par
We agrue as follows. % for $\widetilde A_1\subset\rC_2$}
First, we need to get a matrix in $\SL(2,R)$ with a square entry,
to be able to move it to a {\it short root\/} position.

\begin{lemma}\label{squareSL2} %Suppose $\rm{char} \ \mathbb F_q\neq 2$.
 Any matrix
$$
A=\left(\begin{array}{cc} a&b \\ c&d  \\ \end{array}\right)
$$
from $\SL(2,R)$ can be moved by $3$ elementary transformations in $\SL(2,R)$ to a matrix of the form
 $$
A=\left(\begin{array}{cc} *&b_1^2 \\ *&*  \\ \end{array}\right).
$$
\end{lemma}

\begin{proof} See Lemma \ref{mainn1}.
\end{proof}

\begin{lemma}\label{square1}
Let $A\in \SL(2,R)$ be of the form
 $$
A=\left(\begin{array}{cc} a&b^2 \\ c'&d'  \\ \end{array}\right).
$$
Then it can be transformed to the matrix of the form
 $$
A=\left(\begin{array}{cc} a&b^2 \\ -c^2&d  \\ \end{array}\right).
$$
by $1$ elementary transformation.
\end{lemma}

\begin{proof}
The argument we produce below is in fact a minor conversion
of \cite{BMS}, Lemma 5.3. Indeed, let
$$
A=\left(\begin{array}{cc} a&b^2 \\ c'&d'  \\ \end{array}\right).
$$
\noindent
Then $(a,b^2)$ is unimodular, and there exist $x,y\in R$ such
that $ax+yb^2=1$. Setting $c=-b^2y^2$,
$d=x(1+b^2y)$, we get
$$ ad-b^2c=ax+ab^2xy+b^4y^2=ax+b^2y(ax+b^2y)=1. $$
\noindent
Consequently,
$$
A_1=\left(\begin{array}{cc} a&b^2 \\ c&d  \\ \end{array}\right)
\in\SL(2,R) $$
\noindent
and thus
$$
AA_1^{-1}=\left(\begin{array}{cc} a&b^2 \\ c'&d'  \\ \end{array}\right)
\left(\begin{array}{cc} d&-b^2 \\ -c&a  \\ \end{array}\right)=
\left(\begin{array}{cc}1&0 \\ c'd-d'c&1  \\ \end{array}\right).
$$
Finally,
$$
\left(\begin{array}{cc}1&0 \\ -c'd+d'c&1  \\ \end{array}\right)\left(\begin{array}{cc} a&b^2 \\ c'&d'  \\ \end{array}\right)=
\left(\begin{array}{cc} a&b^2 \\ c&d  \\ \end{array}\right)=
\left(\begin{array}{cc} a&b^2 \\ -b^2y^2&d  \\ \end{array}\right).
$$.
\end{proof}

\begin{lemma}\label{lem} \label{lem:long-short}
Any matrix of the form
$$
A=\varphi_\beta\left(\begin{array}{cc} a&b^2 \\ c&d  \\ \end{array}\right)\in\Sp(4,R) $$
\noindent
can be moved to a matrix of the form
$$
A_1=\varphi_\alpha\left(\begin{array}{cc} a&b \\ *&*  \\ \end{array}\right).
$$
\noindent
by not more than $10$ elementary  transformations in $\Sp(4,R)$.
\end{lemma}
\begin{proof} Use Lemma \ref{square1} to get square
in the SW corner of $A$ by 1 elementary move. We get a matrix
of the form
$$
A'=\varphi_\beta\left(\begin{array}{cc} a&b^2 \\ -c^2&*  \\ \end{array}\right).
$$
\noindent
Use Lemma \ref{sh} to transform $A'$ to
 $$
A_1=\varphi_\alpha\left(\begin{array}{cc} a&b \\ *&*  \\ \end{array}\right).
$$
\noindent
by not more than 9 elementary moves.
\end{proof}

%%%%%%%%%%%%%%%%%%%%%%%%%%%
%%%%%%%%%%%%%%%%%%%%%%%%%%%%%Pach
%%%%%%%%%%%%%%%%%%%%%%%%%%%%%
%%%%%%%%%%%%%%%%%%%%%%%%%%5
\begin{remark}\label{ls}
The above amounts to saying that we need at most 9 elementary transformations to move a fundamental short root $\SL_2$ to
a fundamental long root $\SL_2$, but we might spend up to 10 elementary transformations to move in the opposite direction.
\end{remark}
\par\bigskip
$\bullet$
Summarising the above, we managed to move the original
matrix $A$ to a matrix of the form
$$ \varphi_\alpha\left(\begin{array}{cc} *&* \\ *&*  \\ \end{array}\right)\in\Sp^\alpha_4(R), $$
\noindent
the fundamental $\SL(2,R)$ in the short root embedding.
%% in $\Sp(4,R)$ on short roots.
The total number of elementary transformations to that stage
is 3+10=13.

%%%%%%%%%%%%%%%%%%%%%%%%%%%%%%%%
%%%%%%%%%%%%%%%%%%%%%%%%%%%%%%%%%
%$$$$$$$$$$$$$$$$$$$$$$$$$$$$$$$$
%\subsection{Swindling Lemma for $\rC_2$}

The symplectic swindling lemma for the short root embedding
${\widetilde \rA}_1 \to \rC_2$ was established in Proposition \ref{sw_sh}. At this point, we can follow the proof by Nica
for the $\SL(3,R)$ case almost verbatim. [Alternatively, we
{\it could\/} follow Carter--Keller's approach, but Nica's
approach furnishes a somewhat better bound.]
%%% The proof of Nica is $\SL_2$-closed and does not depend on the embedding $\rA_1\subset \rA_2$ apart from the  place covered by Swindling Lemma.
For the sake of self-completeness, we reproduce all details
(see \cite{Nic} for the original exposition).
\par
We start with a matrix
$$ A=\left(\begin{array}{cc} a&b \\ c&d  \\ \end{array}\right)
\in \SL(2,R). $$
\noindent
and proceed as follows.
\par\bigskip
$\bullet$
Using  the Kornblum--Artin
version of Dirichlet's theorem (see Theorem \ref{kornblum}), make $b$ and $c$ in the above matrix irreducible
of coprime degrees $\deg(b)$ and $\deg(c)$. Then
$$ \delta(b)=\frac{q^{\deg(b)}-1}{q-1}\qquad
\text{and}\qquad \delta(c)=\frac{q^{\deg(c)}-1}{q-1}$$
\par\smallskip\noindent
are also coprime. In other words, there exist $u,v\in\mathbb{N}$
such that
$$ u\delta(b)-v\delta(c)=1. $$
\noindent
This requires not more than 2 elementary moves.
\par\bigskip
$\bullet$ It follows that
$$ \begin{pmatrix} a&b\\ c&d\\ \end{pmatrix}=
{\begin{pmatrix} a&b\\ c&d\\ \end{pmatrix}}^{u\delta(b)}\cdot
\begin{pmatrix} a&b\\ c&d\\ \end{pmatrix}^{-v\delta(c)}. $$
\par\noindent
We reduce the factors independently.
\par\bigskip
$\bullet$ To this end, recall that by the Cayley--Hamilton theorem,   $A^2=\tr(A)A-I$ and
 $A^m=x(\tr (A))I+y(\tr(A))A$, where $I$ stands for the identity matrix and $x$, $y$ are polynomials in $\Int[t]$ (see Remark \ref{Cheb}  below).
For an arbitrary $m$ one has
$$ \begin{pmatrix} a&b\\ c&d\\ \end{pmatrix}^m
=x\begin{pmatrix} 1&0\\0&1\\ \end{pmatrix}+
y\begin{pmatrix} a&b\\ c&d\\ \end{pmatrix}=
\begin{pmatrix} x+ya&yb\\ yc&x+yd\\ \end{pmatrix}. $$
%\par\bigskip
By explicit calculations we get
$$ x+ya\equiv a^m\pmod{b}\qquad
\text{and}\qquad  x+ya\equiv a^m\pmod{c}. $$

%\par\bigskip
\noindent
\begin{remark} \label{Cheb}
In fact, $x$ and $y$ are explicitly known,
morally they are the values of two consecutive Chebyshev polynomials $U_{m-1}$ and $U_m$ at $\tr(A)/2=(a+d)/2$,
which allows one to argue differently, {\it without swindling\/}.
But we do not use it here because this approach would
require more elementary moves.
\end{remark}
%\par\bigskip
\par\bigskip
$\bullet$ Now, using swindling on short roots embedding provided by Proposition \ref{sw_sh}  we reduce
$$ A=\begin{pmatrix} x+ya&yb\\ yc&x+yd\\ \end{pmatrix} $$
\par\smallskip\noindent
[in the short root position!] to either
$$ B=\begin{pmatrix} x+ya&y^2b\\ c&x+yd\\ \end{pmatrix} $$
\par\smallskip\noindent
or
$$ C=\begin{pmatrix} x+ya&b\\ y^2c&x+yd\\ \end{pmatrix} $$
\par\smallskip\noindent
depending on whether we argue modulo $c$ or modulo $b$.

%Recall that $x+ya\equiv a^m\pmod{c}$.
\par\bigskip
$\bullet$
Taking $m=v\delta(c)$, we see that the first matrix is
triangular modulo $c$ and that $x+ya \mod c  \in\GF{q}^*$.
Since $x+ya\equiv a^m\pmod{c}$, for the latter inclusion
we shall check that $z:=a^{\delta(c)}  {\mathrm{mod}}\, c$ lies in $\GF{q}^*.$
Denote $M=\deg c$. Let $\mathbb F_{q'} = \mathbb F_{q^M}$ be the extension
of degree $M$ of the field $ \GF{q}$, and set $e:=a \,  {\mathrm{mod}}\, c \in \mathbb F_{q'}$. We shall prove
$z^q=z$, i.e., $z^{q-1}=1.$ We have
$$\begin{aligned} z^{q-1} &= (a^{\delta(c)} {\mathrm{mod}}\, c)^{q-1} = ((a \, {\mathrm{mod}}\, c)^{\delta(c)})^{q-1} =
(e^{\delta(c)})^{q-1} \\
& =(e^{(q^M-1)/(q-1})^{q-1} = e^{q^M-1}=e^{q'}=1 .
\end{aligned}
$$

%%\par\bigskip
Denote $u:=x+ya=u\mod c \in\GF{q}^*$. %Thus in 5 moves we get
%$$
%\begin{pmatrix} x+ya&y^2b\\ c&x+yd\\ \end{pmatrix}=\begin{pmatrix} u+ct&y^2b\\ c&x+yd\\ \end{pmatrix} \longrightarrow %\left(\begin{array}{cc} 1&0 \\0&1  \\ \end{array}\right)
%$$
%\par\bigskip
Applying the same arguments to the matrix $B^{-1}$, we conclude that $x+yd \mod c = u^{-1}\in\GF{q}^*$.
  We have
$$
\begin{pmatrix} x+ya&y^2b\\ c&x+yd\\ \end{pmatrix}=\begin{pmatrix} u+cr&y^2b\\ c&u^{-1}+cq\\ \end{pmatrix} \longrightarrow
\left(\begin{array}{cc} u&0 \\c&u^{-1}  \\ \end{array}\right) \longrightarrow \left(\begin{array}{cc} u&0 \\0&u^{-1}  \\ \end{array}\right)=h_2
$$
%\par\bigskip
\noindent
in 3=2+1 elementary moves (the element in the NE corner of the penultimate matrix is automatically zero because the determinant of the matrix is equal to one).
\par\bigskip
$\bullet$ Similarly, taking $m=u\delta(b)$, we see that the second matrix is triangular modulo  $b$, and we have $v:=x+ya \mod b \in\GF{q}^*$, $x+yd \mod b =v^{-1}\in\GF{q}^*$.
Accordingly, it can be reduced the matrix of the form
$$
\begin{pmatrix}  \begin{array}{cc} v&0 \\0&v^{-1}  \\ \end{array}\end{pmatrix}=h_1
$$
in 3 elementary moves.
\par\bigskip
$\bullet$
By Corollary \ref{cor:diagonal},
$$
\varphi_\alpha(h_1)\varphi_\alpha(h_2)=\begin{pmatrix} \begin{array}{cccc} v&0&0&0\\0 &v^{-1}&0&0 \\0 &0&v^{-1} &0 \\0&0&0&v  \\ \end{array}\end{pmatrix} \begin{pmatrix} \begin{array}{cccc} u&0&0&0\\0 &u^{-1}&0&0 \\0 &0&u^{-1} &0 \\0&0&0&u  \\ \end{array}\end{pmatrix}
$$
\noindent
can be reduced to the identity matrix in 4 moves.
%%%%%%%%%%%%%%%%%%%%%
%%%%%%%%%%%%%%%%%%%%%
%%%%%%%%%%%%%%%%%%%%%%

%The same operations reduce the first factor to the identity matrix by the same number of elementary transformations. Thus, both factors in the expression of $A$ are elementary.
\par\bigskip
Calculating the total number of all elementary transformations  used so far one gets the following result.

\begin{theorem}\label{sp4} The elementary width of $\Sp(4,\mathbb F_q[t])$ is finite and, moreover,
$$ w_E\big(\Sp(4,\GF{q}[t]\big)\le 79. $$
\end{theorem}

\begin{proof}
We have to apply Lemmas \ref{stabC2} (10 moves), \ref{squareSL2} (3 moves), \ref{lem:long-short} (10 moves),
Proposition \ref{sw_sh} (twice) ($2\cdot 26=52$ moves), and Corollary \ref{cor:diagonal} (4 moves), which gives 79 moves,
as claimed.
\end{proof}

%%%%%%%%%%%%%%%%%%%%%%%%%%%%%%

%\begin{prop}\label{c2}
% The group $\Sp(4,\GF{q}[X])$ is boundedly
%generated by  elementary matrices.
%\end{prop}
%\begin{proof}

%By stable calculations we can to treat only the case of the embedding of root systems $\rA_1\to\rC_2$ on long roots. So arbitrary matrix
%$A\in\Sp(4,\GF{q}[X])$ can be transformed to a matrix of the form
%$$
%A_1=\varphi_\beta\left(\begin{array}{cc} a&b \\ c&d  \\ \end{array}\right).
%$$
%%By Lemma \ref{square} matrix $A_1$ can  be transformed to
%$$
%A_2=\varphi_\beta\left(\begin{array}{cc} *&b^2 \\ *&*  \\ \end{array}\right).
%$$
%By Lemma \ref{squaremove}, matrix $A_2$ can be moved to
%$$
%%A_3=\varphi_\alpha\left(\begin{array}{cc} *&* \\ *&*  \\ \end{array}\right).
%$$
%By Lemma \ref{mainn} matrix $A_3$ can  be transformed to
%$$
%A_4=\varphi_\alpha\left(\begin{array}{cc} e^{q-1}&b \\ *&*  \\ \end{array}\right).
%$$
%By Lemma \ref{power1} matrix $A_4$ can  be transformed to the identity matrix.

%\end{proof}

%%%%%%%%%%%%%%%%%%%%%%%%%%%%

%%%%%%%%%%%%%%%%%%%%%%%%%%%%%%
\section{Proof of Theorem А via the reduction to rank 3 case}\label{rank3}

Let $G(\Phi,R)$ be a Chevalley group of rank $\ge 3$.
Then by stable calculations we can reduce the question of bounded elementary generation of $G(\Phi,R)$ to the root
systems of rank 3 rather than those of rank 2. This approach allows us to obtain somewhat better estimates for the elementary width of $G(\Phi,R)$. With this end we have to consider $\Phi=\rC_3$ and $\Phi=\rB_3$ separately.

\subsection{Proof of Theorem А for $\rC_3$ case}
%Let $G=G(C_2,R)$, where $R=\mathbb F_q[t]$. Fix an order on $\Phi$, and let $\Phi^+$ and $\Pi$   be
%the sets of positive and fundamental roots, respectively. Then
Recall that $G(\rC_3,R)$ is the symplectic group $\Sp(6,R)$ of $6\times 6$-matrices preserving the form
$$ B(x,y=(x_ly_{-1} - x_{-1}y_1)+(x_2y_{-2} - x_{-2}y_2)+(x_3y_{-3} - x_{-3}y_3). $$
\noindent
In this case,
$$ \Pi=\{\alpha=\epsilon_1-\epsilon_2, \beta=\epsilon_2-\epsilon_3, \gamma=2\epsilon_3\}. $$ % and $\Phi=\{\alpha=\epsilon_1-\epsilon_2, \beta=2\epsilon_2, \alpha+\beta=\epsilon_1+\epsilon_2, 2\alpha+\beta=2\epsilon_1\}$.
\noindent
We fix a representation with the highest weight $\mu=\epsilon_1$
--- the vector representation. Other weights of the vector representation are
\begin{multline*}
\mu-\alpha =\epsilon_2,\
\mu - (\alpha+\beta)=\epsilon_3,\
\mu - (\alpha+\beta+\gamma)=-\epsilon_3,\\
\mu - (\alpha+2\beta+\gamma)=-\epsilon_2,\
\mu - (2\alpha+2\beta+\gamma)=-\epsilon_1.
\end{multline*}
\noindent
The corresponding weight diagram looks as follows:
$$
\begin{tikzpicture}
\draw[thick] (0, 0) circle (0.07) node[below] {$\mu = \epsilon_1$};
\draw[thick] (1.7, 0) circle (0.07) node[below] {$\epsilon_2$};
\draw[thick] (3.4, 0) circle (0.07) node[below] {$\epsilon_3$};
\draw[thick] (5.1, 0) circle (0.07) node[below] {$-\epsilon_3$};
\draw[thick] (6.8, 0) circle (0.07) node[below] {$-\epsilon_2$};
\draw[thick] (8.5, 0) circle (0.07) node[below] {$-\epsilon_1$};
%\draw[thick] (10.2, 0) circle (0.07) node[below] {$-\beta$};
\draw[thick] (0.07, 0) to node[midway, above] {$\alpha$} (1.7 - 0.07, 0);
\draw[thick] (1.7 + 0.07, 0) to node[midway, above] {$\beta$} (3.4 - 0.07, 0);
\draw[thick] (3.4 + 0.07, 0) to node[midway, above] {$\gamma$} (5.1 - 0.07, 0);
\draw[thick] (5.1 + 0.07, 0) to node[midway, above] {$\beta$} (6.8 - 0.07, 0);
\draw[thick] (6.8 + 0.07, 0) to node[midway, above] {$\gamma$} (8.5 - 0.07, 0);
%\draw[thick] (8.5 + 0.07, 0) to node[midway, above] {$\beta + \gamma$} (10.2 - 0.07, 0);
\end{tikzpicture}
$$
\par\smallskip
Take an arbitrary matrix
$$A=\left(\begin{array}{cccccc} a_{11}&a_{12}&a_{13}&a_{14}&a_{15}&a_{16} \\ a_{21}&a_{22}&a_{23}&a_{24}&a_{25}&a_{26} \\a_{31}&a_{32}&a_{33}&a_{34}&a_{35}&a_{36}  \\a_{41}&a_{42}&a_{43}&a_{44}&a_{45}&a_{46}\\
 a_{51}&a_{52}&a_{53}&a_{54}&a_{55}&a_{56}  \\ a_{61}&a_{62}&a_{63}&a_{64}&a_{65}&a_{66} \end{array}\right)
\in\Sp(6,R). $$
\noindent
The embedding $\rC_2\subset \rC_3$ gives rise to
$$A'=\left(\begin{array}{cccccc} 1&0&0&0&0&0 \\ 0&a_{22}&a_{23}&a_{24}&a_{25}&0 \\0&a_{32}&a_{33}&a_{34}&a_{35}&0  \\0&a_{42}&a_{43}&a_{44}&a_{45}&0\\
 0&a_{52}&a_{53}&a_{54}&a_{55}&0  \\ 0&0&0&0&0&1 \end{array}\right)\in G(\rC_2\subset \rC_3). $$

%NNNNNNNNNNNNNNNNNNNNNNNNNNNNNNNNNN
\begin{lemma}\label{squareC3}\label{stabC3} A matrix $A$ in $G(\rC_3,R)$ can be moved to $A'$ in $G(\rC_2\subset C_3,R)$ by $\leq 16$ elementary transformations.
\end{lemma}
\begin{proof} %The proof word by word repeats the one from Lemma \ref{stabC2}.
Let the fundamental roots of $\rC_3$ be $\alpha=\epsilon_1-\epsilon_2,$ $\beta=\epsilon_2-\epsilon_3,$ $\gamma=2\epsilon_3.$
We fix a representation with the highest weight $\mu=\epsilon_1$. The corresponding weight diagram is as follows:
$$
\begin{tikzpicture}
\draw[thick] (0, 0) circle (0.07) node[below] {$\mu = \epsilon_1$};
\draw[thick] (1.7, 0) circle (0.07) node[below] {$\epsilon_2$};
\draw[thick] (3.4, 0) circle (0.07) node[below] {$\epsilon_3$};
\draw[thick] (5.1, 0) circle (0.07) node[below] {$-\epsilon_3$};
\draw[thick] (6.8, 0) circle (0.07) node[below] {$-\epsilon_2$};
\draw[thick] (8.5, 0) circle (0.07) node[below] {$-\epsilon_1$};
%\draw[thick] (10.2, 0) circle (0.07) node[below] {$-\beta$};
\draw[thick] (0.07, 0) to node[midway, above] {$\alpha$} (1.7 - 0.07, 0);
\draw[thick] (1.7 + 0.07, 0) to node[midway, above] {$\beta$} (3.4 - 0.07, 0);
\draw[thick] (3.4 + 0.07, 0) to node[midway, above] {$\gamma$} (5.1 - 0.07, 0);
\draw[thick] (5.1 + 0.07, 0) to node[midway, above] {$\beta$} (6.8 - 0.07, 0);
\draw[thick] (6.8 + 0.07, 0) to node[midway, above] {$\gamma$} (8.5 - 0.07, 0);
%\draw[thick] (8.5 + 0.07, 0) to node[midway, above] {$\beta + \gamma$} (10.2 - 0.07, 0);
\end{tikzpicture}
$$

%Let $G=G(C_2,R)$, where $R=\mathbb F_q[t]$. Fix an order on $\Phi$, and let $\Phi^+$ and $\Pi$   be
%the sets of positive and fundamental roots, respectively. Then $\Pi=\{\alpha=\epsilon_1-\epsilon_2, %\beta=2\epsilon_2\}$ and $\Phi=\{\alpha=\epsilon_1-\epsilon_2, \beta=2\epsilon_2, \alpha+\beta=\epsilon_1+\epsilon_2, 2\alpha+\beta=2\epsilon_1\}$. We fix a representation with the highest weight $\mu=\epsilon_1$. So the other weights %are $\mu-\alpha =\epsilon_2, \mu - (\alpha+\beta)=-\epsilon_2, \mu - (2\alpha+\beta)=-\epsilon_1$.
 %$G(C_3,R)$ is the symplectic group $\Sp(6,R)$ of 6x6 matrices, preserving the form $B(x,y) : (x_ly_{-1} - x_{-1}y_1)+(x_2y_{-2} - x_{-2}y_2)+(x_3y_{-3} - x_{-3}y_3)$. %Finally, $\alpha$  and $\alpha+\beta$ are the short roots, while $\beta$ and $2\alpha+\beta$ are long ones.

Let $x$ be the first column of $A\in \Sp(6,R)$,
$$ x=(x_1,x_2,x_3,x_{-3},x_{-2},x_{1}).$$ %,x_\alpha,x_\gamma,x_{-\beta}). $$
We need to reduce it by elementary transformations to
$$ x=(1,0,0,0,0,0).$$

 %We  denote elements of a first column of a matrix $A\in Sp(6,R)$ lexicographically:
%$$ x=(x_1,x_2,x_3,x_4).$$ %,x_\alpha,x_\gamma,x_{-\beta}). $$
\noindent
%Let $I=\langle x_1, x_2,x_3\rangle$ be the ideal generated by the first 3 entries of $x$. So, $I+\langle x_4\rangle=R$.
\par\smallskip
$\bullet$ Since $R$ is a Dedekind ring,  there exists $t\in R$ such that  $x_{-\alpha}(t) x$  is unimodular, see Lemma \ref{Dedek}. %Note that $\langle x_4\rangle=\langle x'_4\rangle \ mod I.$  The first column of $x'=x_{-\alpha}(t) x$  is unimodular and we have  $\langle x'_2,x'_3,x'_4\rangle=I+\langle x'_4\rangle=R$.
\par\smallskip
$\bullet$ Then there exist $t_1, t_2,t_3,t_4,t_5\in R$ such that in $$x_\alpha(t_1)x_{\alpha+\beta}(t_2)x_{\alpha+\beta+\gamma}(t_3)x_{\alpha+2\beta+\gamma}(t_4)x_{\alpha+2\beta+2\gamma}(t_5) x$$ we obtain the first column of the form
$$ x=(1,*,*,*,*,*)$$
(cf. \cite{Stein2}).
\par\smallskip
$\bullet$
Having $1$ in the NW corner of the matrix, it remains to apply 5
downward elementary moves to get
$$ x=(1,0,0,0,0,0).$$
Other 5 elementary moves allow to make the first row  $x=(1,0,0,0,0,0)$ as well.
\par\smallskip
Summarising the above, we see that at most $16=1+5+5+5$ moves are needed to reduce $A\in \Sp(6,R)$ to $A'$ in $G(\rC_2\subset \rC_3,R)$.
%Having an invertible element in the upper left corner of the matrix it remains to  make three elementary moves downstairs and three left-to-right elementary moves to get zeros in the first colomn and the first row. Thus we  transformed $A$ to $A'$ by 10=(1+3+3+3) elementary transformations in total.
%Recall that $A'$ in $G(A_1\subset C_2,R)$ is the image of $A\in G(A_1,R)$ embedded into $G(C_2,R)$ on long roots.  We  denote a first column of a matrix $A\in Sp(4,R)$ lexicographically:
%$$ x=(x_1,x_2,x_3,x_4).$$ %,x_\alpha,x_\gamma,x_{-\beta}). $$
%\noindent
%Let $I=\langle x_1, x_2,x_3\rangle$ be the ideal generated by the first 3 entries of $x$. So, $I+\langle x_4\rangle=R$.  Since $R$ is a Dedekind ring  there exists $t\in R$ such that in $x'=x_{-\alpha}(t) x$ the ideal $I$ is generated two entries $I=\langle x'_2,x'_3\rangle$, see Lemma \ref{Dedek}. Note that $\langle x_4\rangle=\langle x'_4\rangle \ mod I.$  The first column of $x'=x_{-\alpha}(t) x$  is unimodular and we have  $\langle x'_2,x'_3,x'_4\rangle=I+\langle x'_4\rangle=R$. Then there exist $t_1, t_2,t_3\in R$ such that in $x_\alpha(t_1)x_{\alpha+\beta}(t_2)x_{2\alpha+\beta}(t_3)x'$ we obtain the first column of the form
%$$ x=(1,*,*,*)$$
%(cf.,\cite{Stein2}).Having an invertible elemenet in the upper left corner of the matrix it remains to  make three elementary moves downstairs and three left-to-right elementary moves to get zeros in the first colomn and the first row. Thus we  transformed $A$ to $A'$ by 10=(1+3+3+3) elementary transformations in total.
\end{proof}

%%%%%%%%%%%%%%%%%%%%%%%%%%%%%
%%%%%%%%%%%%%%%%%%%%%%%%%%%%%kkkkkkkkkkkkkkkkkkkkkkkkkkkkkkkkkkkkk
%%%%%%%%%%%%%%%%%%%%%%%%%%%%%%%%
Using Lemma \ref{stabC2}, the matrix $A'$ can be moved to $G (\rA_1\subset \rC_2\subset \rC_3, R )$ by not more than 10
elementary moves.
\par\smallskip
Similarly, using Lemma \ref{mainn1} + the usual stability for $SL(3,R)$
% and \ref{bassms},
the matrix $A'$ can be moved to  $A''\in G({\widetilde {\rA}}_1\subset \rC_2\subset \rC_3, R)$ by at most 3+9=12 elementary moves.
\par\smallskip
The matrix $A''$ is of the form
$$A''=\left(\begin{array}{cccccc} 1&0&0&0&0&0 \\ 0&a_{22}&a_{23}&0&0&0 \\0&a_{32}&a_{33}&0&0&0  \\0&0&0&a_{44}&a_{45}&0\\
 0&0&0&a_{54}&a_{55}&0  \\ 0&0&0&0&0&1 \end{array}\right)= \left(\begin{array}{cccc}1&0&0&0\\ 0&B&0&0\\0&0& B^{-1}&0\\ 0&0&0&1   \end{array}\right),  $$
\noindent
where $B\in\SL(2,R)$.
% What happened is that in fact we managed to get the $K_1$-surjective stability theorem for the embedding ${\widetilde{\rA}}_2\subset \rC_3$ for a Dedekind ring of arithmetic type (cf. \cite{Pl1}, Table 1).

Now look at the matrix
$$\left(\begin{array}{ccc}1&0&0\\ 0&a_{22}&a_{23}\\0&a_{32}&a_{33}  \end{array}\right)\in\SL(2,R)\le \SL(3,R). $$
\noindent
According to Nica's Theorem it can be moved to the identity
 matrix in not more than 34 elementary transformations \cite{Nic}.

 Summing up all elementary moves above we get
 \begin{theorem}\label{sp6} The elementary width of $\Sp(6,\mathbb F_q[x])$ is finite and, moreover,
$$ w_E\big(\Sp(6,\GF{q}[t]\big)\le 72. $$
\end{theorem}
\begin{proof}
16+10+12+32=70.
\end{proof}

 \subsection{Proof of Theorem А for $\rB_3$ case}

In this case,
$$ \Pi=\{\alpha=\epsilon_1-\epsilon_2, \beta=\epsilon_2-\epsilon_3, \gamma=\epsilon_3\}. $$ % and $\Phi=\{\alpha=\epsilon_1-\epsilon_2, \beta=2\epsilon_2, \alpha+\beta=\epsilon_1+\epsilon_2, 2\alpha+\beta=2\epsilon_1\}$.
\noindent
We fix the 7-dimensional orthogonal representation with the highest weight $\mu=\epsilon_1$ --- the vector representation. Other weights of the vector representation are
\begin{multline*}
\mu-\alpha =\epsilon_2,\
\mu - (\alpha+\beta)=\epsilon_3,\
\mu - (\alpha+\beta+\gamma)=0,\
\mu - (2\alpha+\beta+2\gamma)=-\epsilon_3,\\
\mu - (\alpha+2\beta+2\gamma)=-\epsilon_2,\
\mu - (2\alpha+2\beta+2\gamma)=-\epsilon_1.
\end{multline*}

 Take an arbitrary matrix
$$A=\left(\begin{array}{ccccccc} a_{11}&a_{12}&a_{13}&a_{14}&a_{15}&a_{16}&a_{17} \\ a_{21}&a_{22}&a_{23}&a_{24}&a_{25}&a_{26}&a_{27} \\a_{31}&a_{32}&a_{33}&a_{34}&a_{35}&a_{36}&a_{37}  \\a_{41}&a_{42}&a_{43}&a_{44}&a_{45}&a_{46}&a_{47}\\
 a_{51}&a_{52}&a_{53}&a_{54}&a_{55}&a_{56}&a_{57}  \\ a_{61}&a_{62}&a_{63}&a_{64}&a_{65}&a_{66}&a_{67}\\
  a_{71}&a_{72}&a_{73}&a_{74}&a_{75}&a_{76}&a_{77} \end{array}\right)\in\SO(7,R). $$
 \noindent
The embedding $\rB_2\subset \rB_3$ gives rise to
$$A'=\left(\begin{array}{ccccccc} 1&0&0&0&0&0&0 \\ 0&a_{22}&a_{23}&a_{24}&a_{25}&a_{26}&0 \\0&a_{32}&a_{33}&a_{34}&a_{35}&a_{36}&0  \\0&a_{42}&a_{43}&a_{44}&a_{45}&a_{46}&0\\
 0&a_{52}&a_{53}&a_{54}&a_{55}&a_{56}&0\\ 0&a_{62}&a_{63}&a_{64}&a_{65}&a_{66}&0  \\ 0&0&0&0&0&0&1
  \end{array}\right)\in G(\rB_2\subset \rB_3). $$

\begin{lemma}\label{stabB3} A matrix $A$ in $G(\rB_3,R)$ can be moved to $A'$ in $G(\rB_2\subset \rB_3,R)$ by $\leq 21$ elementary transformations.
\end{lemma}
\begin{proof}

As usual, we focus on the first column $A_{*\mu}$ of $A$. The action of elementary unipotents on the first column of $A$ can be viewed via the weight diagram

$$
\begin{tikzpicture}
\draw[thick] (0, 0) circle (0.07) node[below] {$\mu = \epsilon_1$};
\draw[thick] (1.7, 0) circle (0.07) node[below] {$\epsilon_2$};
\draw[thick] (3.4, 0) circle (0.07) node[below] {$\epsilon_3$};
\draw[thick] (5.1, 0) circle (0.07) node[below] {$0$};
\draw[thick] (6.8, 0) circle (0.07) node[below] {$-\epsilon_3$};
\draw[thick] (8.5, 0) circle (0.07) node[below] {$-\epsilon_2$};
\draw[thick] (10.2, 0) circle (0.07) node[below] {$-\epsilon_1$};
\draw[thick] (0.07, 0) to node[midway, above] {$\alpha$} (1.7 - 0.07, 0);
\draw[thick] (1.7 + 0.07, 0) to node[midway, above] {$\beta $} (3.4 - 0.07, 0);
\draw[thick] (3.4 + 0.07, 0) to node[midway, above] {$\gamma$} (5.1 - 0.07, 0);
\draw[thick] (5.1 + 0.07, 0) to node[midway, above] {$-\gamma$} (6.8 - 0.07, 0);
\draw[thick] (6.8 + 0.07, 0) to node[midway, above] {$-\beta$} (8.5 - 0.07, 0);
\draw[thick] (8.5 + 0.07, 0) to node[midway, above] {$-\alpha$} (10.2 - 0.07, 0);
\end{tikzpicture}
$$

Denote the first column by
$$ x=(x_1,x_2,x_3,x_0,x_{-3},x_{-2},x_{-1}). $$
%\noindent
%It is our intention to reduce this column to the form $(1,*,*,*,*,*,*)$ by
%elementary unipotents.
We need to get the column
$$
(1,0,0,0,0,0,0)
$$
by elementary transformations. The adapt the proof from \cite{Stein2}, Theorem~2.1, with some minor improvements
for Dedekind rings.
\par\smallskip
$\bullet$
Consider the ideal $I=\langle x_{-3},x_{-2},x_{-1}\rangle $. Then the column $ (x_1,x_2,x_3,x_0)$ is unimodular in $R/I$. By Lemma \ref{Dedek}, there exists $t_0$ such that in  $x_\gamma(t_0)x$ the column $ (x_1,x_2,x_3)$ is unimodular in
$R/I$.
\par\smallskip
$\bullet$ There are $t_1,t_2,t_3$ such that the first component
of $x_{-\alpha}(t_1)x_\beta(t_2)x_\alpha(t_3)x$ is a unit in $R/I$.
\par\smallskip
$\bullet$ Then there are $t_4,t_5$ such that in $x_{-\alpha}(t_4)x_{-\beta}(t_5)x$ we have
$$ x_1\equiv 1\pmod I,\qquad x_2\equiv x_3\equiv 0 \pmod I. $$
\noindent
Hence the column
$$
(x_1,-,-,-,x_{-3},x_{-2},x_{-1})
$$
is unimodular in $R$.
\par\smallskip
$\bullet$ Then there exists $t_6$ (Lemma \ref{Dedek}) such that in $x_{\alpha}(t_6)x$ the column
$$
(x_1,-,-,-,x_{-3},x_{-2},-)
$$
 is unimodular in $R$.
\par\smallskip
$\bullet$ Then there is $t_7$ such that in either
$x_\beta(t_7)x$ or in $x_{\alpha+2\beta+2\gamma}(t_7)x$
the column
$$ (x_1,-,-,-,x_{-3},-,-) $$
\noindent
is unimodular.
\par\smallskip
$\bullet$ Then there exist $t_8$ and $t_9$ such that in
$x_{-(\alpha+2\beta+2\gamma)}(t_{9})x_{-\beta}(t_8))x$ we obtain the column
$$ (x_1,-,-,-,x_{-3},1,-). $$
\par\smallskip
$\bullet$ One more elementary transformation provides  the column
 $$  (1,-,-,-,-,-,-). $$
\par\smallskip
$\bullet$ Finally, we need 5 more unipotents acting downstairs to get the first column
$$ (1,0,0,0,0,0,0). $$
The total number of elementary unipotents used in the process is 16.
\par\smallskip
$\bullet$ We need 5 more transformations to bring the first row to the same shape.
\par\smallskip
Summarising the above, we see that the total number of  elementary transformations needed to reduce $A$ in $G(\rB_3,R)$ to $A'$ in $G(\rB_2\subset \rB_3,R)$ is 21.
%The proof word by word repeats the one from Lemma \ref{stabC2}.
%Recall that $A'$ in $G(A_1\subset C_2,R)$ is the image of $A\in G(A_1,R)$ embedded into $G(C_2,R)$ on long roots.  We  denote a first column of a matrix $A\in Sp(4,R)$ lexicographically:
%$$ x=(x_1,x_2,x_3,x_4).$$ %,x_\alpha,x_\gamma,x_{-\beta}). $$
%\noindent
%Let $I=\langle x_1, x_2,x_3\rangle$ be the ideal generated by the first 3 entries of $x$. So, $I+\langle x_4\rangle=R$.  Since $R$ is a Dedekind ring  there exists $t\in R$ such that in $x'=x_{-\alpha}(t) x$ the ideal $I$ is generated two entries $I=\langle x'_2,x'_3\rangle$, see Lemma \ref{Dedek}. Note that $\langle x_4\rangle=\langle x'_4\rangle \ mod I.$  The first column of $x'=x_{-\alpha}(t) x$  is unimodular and we have  $\langle x'_2,x'_3,x'_4\rangle=I+\langle x'_4\rangle=R$. Then there exist $t_1, t_2,t_3\in R$ such that in $x_\alpha(t_1)x_{\alpha+\beta}(t_2)x_{2\alpha+\beta}(t_3)x'$ we obtain the first column of the form
%$$ x=(1,*,*,*)$$
%(cf.,\cite{Stein2}).Having an invertible elemenet in the upper left corner of the matrix it remains to  make three elementary moves downstairs and three left-to-right elementary moves to get zeros in the first colomn and the first row. Thus we  transformed $A$ to $A'$ by 10=(1+3+3+3) elementary transformations in total.
\end{proof}

\begin{lemma}\label{stab23} A matrix  $A'$ in $G(\rB_2\subset \rB_3,R)$ can be moved to $A''$ in $G(\rA_1\subset \rB_2,R)$ by
$\leq 10$ elementary transformations.
\end{lemma}
\begin{proof}
Since the groups of types $\rB_2$ and $\rC_2$ are isomorphic, one can refer to Lemma \ref {stabC2}.
\end{proof}

Ultimately, reduction of a matrix from $\G(\rB_3,R)$
to $G(\rA_1,R)$ along the chain of root system embeddings
$\rA_1\subset \rB_2\subset B_3$ requires
$\leq 31$ elementary transformations.
\par
Since we have a commutative diagram of root embeddings
$$
\begin{tikzpicture}
\def\a{1.5} \def\b{2}\def\c{4.3}
\path
(-\a,0) node (A_1) {$A_1$}
(\a,0) node (A_2) {$A_2$}
(\c,0) node (B_3) {$B_3$}
(1.5,-\b) node (B_2) {$B_2$};%\\[-1mm]\rotatebox{90}{=}\\[-1mm]$D$};
%(0,-\b) node[align=center] (B_2) {$B_2$\\[-1mm]\rotatebox{90}{=}\\[-1mm]$D$};
\begin{scope}[nodes={midway,scale=.75}]
\draw[->] (A_1)--(A_2) node[above]{$f$};
\draw[->] (A_2)--(B_3) node[above]{$g$};
\draw[->] (A_1)--(B_2.120) node[left]{$\varphi$};
\draw[->] (B_2.60)--(B_3) node[right]{$\psi$};
\end{scope}
\end{tikzpicture},
$$
we have the corresponding diagram of homomorphisms of $\mathrm K_1$-functors, see \cite{Stein2} or \cite{Pl2},
Lemma~3.
\par
Lemmas \ref{stabB3} and \ref{stab23} imply that the composition  $\psi\circ \varphi$ is an epimorphism. Hence the homomorphism of $K_1$-functors $g$ corresponding to $\rA_2\to \rB_3$ is an epimorphism as well. Thus we obtain
$$ G(\rB_3,R)=G(\rA_2,R)E^{31}(\rB_3,R). $$
\par
Combining this with Nica's theorem, that gives additional
$\le 34$ elementary transformations, we obtain the following
result.

\begin{theorem}\label{so4} The elementary width of $\SO(7,\mathbb F_q[x])$ is finite and, moreover,
$$ w_E\big(\SO(7,\GF{q}[t]\big)\leq 65. $$
\end{theorem}
%\begin{proof}

\begin{remark}
In this section we used the adjoint group of type $\rB_3$ and not the simply connected one.
As noted in the introduction, this does not affect the finiteness of the elementary width
of an arbitrary group of this type.
\end{remark}

\section{Proof of Theorem C} \label{Queen}
%%% Groups over Laurent polynomials, via Queen's approach}

%% \subsubsection{Bounded generation of Chevalley
%% groups over Laurent polynomials}

Actually, for applications to Kac--Moody groups, we mostly
need results for Chevalley groups not over the polynomial
ring $\GF{q}[t]$ but rather over the Laurent
polynomial rings $\GF{q}[t,t^{-1}]$. The key difference
between these cases is that while the above polynomial
ring contains finitely many units, the Laurent polynomial
ring has infinitely many of them, namely all $at^m$, where
$m\in\Int$, $a\in\GF{q}^*$.
\par
As we have already mentioned in Section \ref{sec:art}, Chevalley groups
over rings with finitely many and infinitely many units may
behave very differently. This phenomenon is most striking
for $\SL(2,R)$. Recall the typical situation occurring in the
number case: the group $\SL(2,\Int)$ does not have
the property of elementary bounded generation whereas
the group $\SL(2,R)$, where $R$ is the ring of $S$-integers
in a number field which has infinitely many units, does,
see, e.g., \cite{MRS} for details.
\par
It seems that elementary bounded generation of $\SL(2,R)$
for rings $R$ of $S$-integers in a global function field which contain infinitely many units, is in general still open.
However, the case $R=\GF{q}[t,t^{-1}]$ can be easily
deduced, and at that with rather sharp bounds, from the
results of Clifford Queen \cite{Qu}.
\par
Theorem \ref{th:reduction} reduces the proof of Theorem C to the
case of the group $\SL(2,R)$. However, a very short
elementary expression in $\SL(2,R)$, for $R=\mathcal O_S$
under some additional assumptions on $S$, was
established by \cite{Qu}.  More precisely, Theorem 2
of the above paper [after correction of a minor inaccuracy]
amounts essentially to the following result.

\begin{prop} \label{th:Queen}
Let $R=\mathcal O_S$ be the ring of $S$-integers of $K$, a
function field of one variable over $\GF{q}$
with $S$ containing at least two places. Assume that at least
one of the following holds:
\par\smallskip
$\bullet$ either at least one of these places has degree one,
\par\smallskip
$\bullet$ or the class number of $R$, as a Dedekind domain,
is prime to $q-1$.
\par\smallskip\noindent
Then any matrix
$C\in\SL(2,R)$ can be expressed as the product of five
elementary transvections.
%%% $$ C=t_{21}(\xi_1)t_{12}(\xi_2)t_{21}(\xi_3)
%%% t_{12}(\xi_4), $$ \noindent
%%%for appropriate $\xi_1,\xi_2,\xi_3,\xi_4\in R$.
%% \qed
\end{prop}
\begin{proof}
In follows from Theorem 2 of \cite{Qu} that in
this situation any matrix $g\in\SL(2,R)$ can be
expressed as the product
$$ g=t_{12}(\zeta_1)t_{21}(\zeta_2)
t_{12}(\zeta_3)t_{21}(\zeta_4)h_{12}(\epsilon), $$
\noindent
for some $\zeta_1,\zeta_2,\zeta_3\in R$ and
$\zeta_4,\epsilon\in R^*$, which immediately gives expression
of $g$ as a product of {\it seven\/} elementary transvections.
\par
However, we can refer to Lemma \ref{diagonal}, asserting that
the first or the last factor in the expression
of $h_{12}(\epsilon)$ as a product of elementary
transvections can be an arbitrary invertible element of
$R$. Thus, we can start our elementary expression of
$h_{12}(\epsilon)$ with the factor $t_{21}(-\zeta_4)$,
that cancels with the previous one. After that
$t_{12}(\zeta_3)$ can be subsumed into the second
factor of the elementary expression of $h_{12}(\epsilon)$,
giving us an expression of $g$ as a product of {\it five\/}
factors of the form $UU^-UU^-U$.
\par
Implementing the same reduction procedure as in the proof
of \cite[Theorem~2]{Qu} for the second column of $g$
instead of the first one, we get a similar expression of $g$
of the form $U^-UU^-UU^-$.
\end{proof}
\noindent

\begin{remark}
Queen's proof is mainly based on the principles proposed in the seminal paper
of Cooke and Weinberger \cite{CW} in the number field set-up. Namely, it uses subtle
analytic ingredients, such as a function field analogue of Artin's primitive root conjecture,
in order to obtain short division chains. In contrast to the number field case where the validity
of Artin's conjecture is only known conditionally on the Generalized Riemann Hypothesis (GRH),
its function field analogue, developed by Bilharz in the 1930's, became an unconditional theorem
after Weil's work. See the paper of Lenstra \cite{Le} for more details, as well for some
strengthening of Queen's theorem.
\par
In \cite{Qu} this result is {\it stated\/} correctly, in
the form to which we
referred in our proof, but if you look inside the proof on
p.~56, it is claimed there that by three multiplications by
elementary matrices one can reduce the first column of
$g$ to the form $(1,0)^t$. This is not the case, from
Lemma~5 it only follows that it can be reduced to the
form $(\epsilon,0)^t$. Thus, there is no way to express
a matrix $g$ as a product of {\it four\/} elementary
transvections, as would result from the text of the
proof of Theorem 2.
\par
One can correct this either as we do above, or, alternatively,
by reducing the first column of $g$ to the form
$(1,\epsilon)^t$, with $\epsilon\in R^*$, by {\it three\/} elementary operations. After that, one needs two more,
to remove $\epsilon$, and another one to remove the
non-diagonal element in the first row. This gives the same
{\it five\/} elementary factors.
\par
It follows from \cite{VSS} that this result is the best possible.
The decomposition $E(2,R)=UU^-UU^-$ --- or, in fact, any
such decomposition of length 4 for any Chevalley group
--- is {\it equivalent\/} to $\sr(R)=1$. Thus, {\it five\/}
elementary factors is the best bound one can expect in the
number case.
\end{remark}
Now, precisely the same argument as the proof of Theorem 1
in the work of Smolensky \cite{Sm} gives us the following
estimate of the commutator width.

\begin{corollary} \label{bound-Laurent-2}
Let $R$ be as in Theorem $\ref{th:Queen}$. Then the
commutator width of the simply connected Chevalley group
 $G=G(\Phi ,R)$ is $\le L$, where
\par\smallskip
$\bullet$ $L=3$ for $\Phi=\rA_l, \rF_4${\rm;}
\par\smallskip
$\bullet$ $L=4$ for $\Phi=\rB_l, \rC_l, \rD_l$, for $l\ge 3$ or
$\Phi=\rE_7, \rE_8$, or, finally, $\Phi=\rC_2, \rG_2$ under the
additional assumption that $1$ is the sum of two units in $R$
{\rm(}which is automatically the case, provided $q\neq2${\rm);}
\par\smallskip
$\bullet$ $L=5$ for $\Phi=\rE_6$.
\end{corollary}

\begin{proof}
In fact, Smolensky proves these bounds for Chevalley
groups over rings with $\sr(R)=1$. The only property of
such a ring $R$ that is used in the proof, is the presence
of a unitriangular factorisation of length {\it four\/},
$E(\Phi,R)=UU^-UU^-$.
\par
However, since the set of commutators is closed under
conjugation, the proof in \cite{Sm} works if not necessarily the matrix
$g\in G(\Phi,R)$ itself, but some of its conjugates admits a unitriangular factorisation of length four. However, in
our situation this immediately follows from Theorem C,
which establishes the unitriangular factorisation of
length {\it five\/}, $E(\Phi,R)=UU^-UU^-U$. Up to
conjugacy the last factor can be carried in front, and
subsumed by the first factor.
\end{proof}

%We believe that for $\Phi=\rE_6$ one could also take $L=4$,
%but could not prove this.

\begin{remark} \label{non-simply-connected}
\medspace

(i) We believe that for $\Phi = \rE_6$ one could also take $L = 4$, but could not prove this.

(ii) We do not know whether one can improve the estimates for non simply connected groups.

\end{remark}

On the other hand, the precise bound on the number of
elementary generators is somewhat more delicate.
Of course, Theorem C immediately implies the
following obvious estimate of the elementary width.

\begin{corollary} \label{bound-Laurent-3}
Let $R$ be as in Theorem C. Then the width
of the Chevalley group $G(\Phi ,R)$ with respect to the
elementary unipotents is $\le 5N$, where $N=|\Phi^+|$
is the number of positive roots.
\end{corollary}

This bound is quite reasonable, but still not the best
possible one.
Using the bounded reduction under stability conditions
we can get very sharp estimates for the number of
elementary factors in other Chevalley groups.
For $\SL(n,R)$ such a reduction with the sharpest
possible bound is very classical and is implemented
already in Carter---Keller \cite{CaKe1}. By the same token, from
the above proposition we get

\begin{corollary} \label{bound-Laurent-4}
Let $R$ be as in Theorem C. Then any
$g\in\SL(n,R)$ can be expressed as a product of
$\le {1\over 2}(3n^2-n)$ elementary transvections.
\end{corollary}

\begin{proof}
Immediately follows from the proposition, via
improvement of bounded reduction for Dedekind
rings. By the contents of Section \ref{subsec:Improvements}, reduction of
$\SL(n+1,R)$ to $\SL(n,R)$ requires $\le 3n+1$ elementary operations.
\end{proof}

%ЧТО ИЗ ЭТОГО МЫ ХОТИМ ОСТАВИТЬ??

%more interestingly --- do we want similar EXPLICIT bounds
%for other Chevalley groups? We COULD do that, since all
%necessary cases of surjective stability ARE in the literature ---
%MOSTLY Stein and Plotkin, --- but the EXPLICIT number
%of elementary factors for other Chevalley groups will be NEW.
%\par
%In fact, it will be new even in the NUMBER case, it is NOT
%produced in the paper by Morgan---Rapinchuk---Sury , so it
%might be a separate small subsection!!!

%в том, что касается явных оценок.
%It is also worth noting that in the subsequent work on the number field case the dependence
%of bounded generation on GRH had been eliminated, see the paper of Morgan, Rapinchuk and Sury
%\cite{MRS} for the definitive results and the history of this activity. Some further improvements
%valid in certain particular cases were recently obtained by Jordan and Zaytman \cite{JZ1}, \cite{JZ2}.

%As mentioned above, in the function field case elementary bounded generation of $\SL_2(\mathcal O_S)$ with infinite
%$\mathcal O_S^*$ remains unknown in general. It is a tempting problem whether one can remove
%the additional assumptions on $S$ in Theorem \ref{th:Queen}.

%%%%%%%%%%%%%%%%%%%%%%%%%%%%%

\section{Applications} \label{appl}

In this section we briefly discuss two immediate applications
of our results. First of all, they imply that Kac--Moody
groups of affine type over a finite field have finite commutator width. This problem served as one of the major initial
motivations of the present work. As another application, we
state several results on bi-interpretability in model theory.

%%%%%%%%%%%%%%%%%%%%%%%%%%%%%%

\subsection{Applications to Kac--Moody groups} \label{sec:KM}
Here we discuss finite commutator width, where there is an
especially straightforward connection between the results for
the usual Chevalley group $G(\Phi,\GF{q}[t,t^{-1}])$ over
the Laurent polynomial ring and the corresponding affine
Kac--Moody group $\tilde G(A,\GF{q})$ over the finite
field itself.
%%%%%%%%%%%%%%%%%%%%%%%%%%%%%
%%%%%%%%%%%%%%%%%%%%%%%%%%%%%
%%%%%%%%%%%%%%%%%%%%%%%%%%%%%%
%%%%%%%%%%%%%%%%%%%%%%%%%%%%%%

Let $A$ be an $n \times n$ indecomposable generalized Cartan matrix of (untwisted) affine type, and let $K$ be a field. By an %(untwisted)
affine  Kac--Moody  group $\widetilde G_{sc}(A,K)$ we mean the value of the simply connected Tits functor \cite{Ti}, cf. \cite{PK},
corresponding to the Cartan matrix $A$. Denote by $\widetilde E_{sc}(A,K)$ its elementary subgroup. The centers $Z(\widetilde G_{sc}(A,K))$ and $Z(\widetilde E_{sc}(A,K))$ coincide. 
%Assuming $|K|\geq 4$, 
%we have 
%$$
%[\widetilde G_{sc}(A,K), \widetilde G_{sc}(A,K)]=\widetilde E_{sc}(A,K).
%$$  % So, one can write $\widetilde G(K)=\widetilde G_{sc}(A,K)$.
We have a short exact sequence 
\begin{equation} \label{seq:KM}
1 \to Z(\widetilde{E}_{sc}(A,K)) \to \widetilde{E}_{sc}(A,K)\to G_{ad}(\Phi,R) \to 1, 
\end{equation} 
%\noindent
% Then,
%$$
%\widetilde E_{sc}(A,K)/Z(\widetilde E_{sc}(A,K))\simeq  G_{ad}(\Phi,R)\simeq  E_{ad}(\Phi,R).
%$$
the group $G_{ad}(\Phi, R)\simeq  E_{ad}(\Phi,R)=E_{ad}(\Phi,K[t,t^{-1}])$ is usually called the {\it loop group} \cite{Ga}.
So, the elementary affine Kac--Moody group is just a central extension of the loop group. Now we are in a position to prove Theorem D. Recall its statement. %that it states the following %By some abuse of language

%%%%%%%%%%%%%%%%%%%%%%%%%%%%%
%%%%%%%%%%%%%%%%%%%%%%%%%%%%%
%%%%%%%%%%%%%%%%%%%%%
\begin{smaintheorem}\label{smtheorem}
The commutator width of an affine elementary untwisted Kac--Moody
group $\widetilde E_{sc}(A,\mathbb F_q)$ over a finite field $\mathbb F_q$ 
%($q\ne 2,3 (?))$.
is $\le L'$,
where
\par\smallskip
$\bullet$ $L'=5$ for $\Phi=\rF_4$ and $\Phi=\rA_l$, $l=2k+1$, $k=0,1,\dots ${\rm;}
\par\smallskip
%$\bullet$ $L=6$ for $\Phi=\rA_l$, $l\geq 2${\rm;}
%\par\smallskip
$\bullet$ $L'=6$ for $\Phi=\rA_l$, $l=2k$, $k=1,2,\dots $,  $\Phi=\rB_l, \rC_l, \rD_l$, for $l\ge 3$ or
$\Phi=\rE_7, \rE_8$, or, finally, $\Phi=\rC_2, \rG_2$ under the
additional assumption that $1$ is the sum of two units in $R$
{\rm(}which is automatically the case provided $q\neq2${\rm);}
\par\smallskip
$\bullet$ $L'=7$ for $\Phi=\rE_6$.
\end{smaintheorem}

\begin{proof} The idea is to get separate estimates for the commutator lengths of the elements of left and right  
terms of exact sequence \eqref{seq:KM} and deduce an estimate for the commutator width of the middle term. 

For any $g\in \widetilde E_{sc}(A,K)$ denote by $\bar g\in G_{ad}(\Phi,R)$ its projection. 
Then $\bar g$ is a product of $L$ commutators, 
$\bar g=[\bar a_1,\bar b_1]\dots [\bar a_L,\bar b_L]$, where $L$ is given by  Corollary \ref{bound-Laurent-2}. 
Define $g':=[a_1,b_1]\dots [a_L,b_L]$.  
As $\bar g=\bar g'$, we have $g=g'h$ for some $h\in Z(\widetilde E_{sc}(A,K))$. We will prove that $h$ is a product of two or three commutators, 
depending on $\Phi$.

%Consider the loop group $\widetilde E_{sc}(A,K)/Z(\widetilde E_{sc}(A,K))\simeq  G_{ad}(\Phi,\mathbb F_q[t,t^{-1}]).$
%One can apply Corollary \ref{bound-Laurent-2} to get an estimate for the commutator width of $G_{ad}(\Phi,\mathbb F_q[t,t^{-1}])$.
%Take $g\in G_{ad}(\Phi,\mathbb F_q[t,t^{-1}])$ and consider its lift $\bar g=gh$, where $h\in Z(\widetilde E_{sc}(A,K))$.

 Denote by $\Pi=\{\alpha_1,\ldots,\alpha_l\}$ the set of fundamental roots of $\Phi$. Then $A$ is determined by the affine  root system $\widetilde \Phi$ with fundamental roots $\widetilde\Pi=\{\alpha_0,\alpha_1,\ldots,\alpha_l\}$, see, e.g. \cite{Ka}, \cite{CaCh}. Accordingly, $h$ can be written as $h=h_{\alpha_0}(\lambda_0)h_{\alpha_1}(\lambda_1)\cdots h_{\alpha_l}(\lambda_l),$ cf. \cite{CaCh}.
 \footnote{The relevant facts in \cite{CaCh} are formulated for Kac--Moody groups over $\mathbb C$. However, the construction remains valid for an appropriate $\mathbb Z$-model \cite{Ga} 
 and hence the needed results from \cite{CaCh} can be extended to groups over $\mathbb F_q$.} 
 Each $h_{\alpha_i}$ lives in $\SL(2,\mathbb F_q)$ and has a bounded commutator length. 
 %Hence the commutator length of $\bar g$ is also bounded. 
 More precisely, suppose that $\widetilde \Phi\neq\widetilde \rA_l$. Then we can represent $h$ as $h_1h_2$, where $h_1=h_{\alpha_{i_1}}\cdots h_{\alpha_{i_k}}$, $h_2=h_{\beta_{j_1}}\cdots h_{\beta_{j_s}}$ such that all the  roots ${\alpha_{i_n}}$ and ${\alpha_{i_m}}$, $n\neq m$, as well as ${\beta_{_p}}$ and ${\beta_{_t}}$, $p\neq t$, are mutually orthogonal. Every  $h_{\alpha_{i_n}}$, $1\leq n\leq k$,  and $h_{\beta_{j_m}}$, $1\leq m\leq s$,  lies in $\SL(2,\mathbb F_q)$, belongs to the center of this group, and is thus a single commutator,  see \cite[Theorem~1]{Th}. Hence each of $h_1$ and $h_2$ belongs to a direct product of  $\SL(2,\mathbb F_q)$ and is thus a single commutators. As a result, $h$ is a product of two commutators.

 The affine Dynkin diagram of type $\widetilde \rA_l$, $l\geq 2$, is a loop. Let $\widetilde \Phi=\widetilde \rA_l$, $l=2k+1$, $k\ge 1$.  Then still 
 %$h=h_{\alpha_0}h_{\alpha_1}$, 
 $h=h_1h_2$, as above, 
 and we need two commutators for $h$.  
 If $\widetilde \Phi=\widetilde \rA_l$, $l=2k$, $k\ge 1$, then there exists a representation $h=h_1h_2h_3$ with the properties as above. In this case $h$ is a product of three commutators.

It remains to combine the estimates for the commutator length of $g'$ from  Corollary \ref{bound-Laurent-2} with the estimates for the commutator length of $h$ to get the required values of $L'$ for any $g$.
\end{proof}

%%%%%%%%%%%%%%%%%%%%%%%%%%%%
%%%%%%%%%%%%%%%%%%%%%%%%%%%%hhhhhhhhhhhh

\begin{remark}\label{spheric}
We do not attempt to state similar results for the bounded
elementary generation, in view of the ambiguity of this
notion. In fact, elementary generators of
$G(\Phi,\GF{q}[t,t^{-1}])$ correspond to the
{\it spherical roots\/} of $\Phi$ and themselves do not have
bounded width with respect to the elementary generators
of the affine Kac---Moody group $\tilde G(\Phi,\GF{q})$,
parametrised in terms of {\it affine roots\/}.
%% The root subgroup $X_{\a}$, $\a\in\Phi$.
\end{remark}

\begin{remark}\label{complete}
Let $\overline G(A, K)$ be a {\it complete} affine Kac--Moody group over a field $K$. Then $\overline G(A, K)$
is isomorphic to the Chevalley group of the
form $G(\Phi, K((t)))$ where $K((t))$ is the field of formal Laurent
series over $K$. %This field is not algebraically closed and the commutator width of  $G(\Phi, K((t)))$ is unknown. It is definitely less than 4 due to \cite[Theorem~1.4]{HLS}.

According to \cite{EG}, any noncentral element $g$ of $G(\Phi, K((t)))$ is a single commutator. Any central element $z$ is representable 
as a product of two noncentral elements and hence as a product of two commutators. Thus the commutator width of $\overline G(A, K)$ is 
at most two. 
\end{remark}

\begin{remark}\label{inna}
It was noticed by Inna Capdeboscq (private correspondence), that the finiteness of the commutator width for Kac-Moody groups can be deduced directly from the polynomial case via Theorem \ref{mtheorem2a}, using the affine Bruhat decomposition. However this approach yields much worse estimates than the ones from Theorem D.
%\ref{mtheoremkac}.
\end{remark}

\subsection{Logical applications}\label{la}
%%% In this section
Here we state several corollaries of Theorem \ref{mtheorem2a} related to model theory.

First note that some of the facts we use in this section require that the group under consideration is finitely generated. 
In our context, this is guaranteed for Chevalley groups of rank > 1 thanks to the results of Helmut Behr \cite{Be}. 

The notion of bi-interpretability which plays a crucial role in model-theoretic applications can be found in many sources.
We refer the reader to \cite{KMS}.

The first important tool is the following Theorem 3.1 of \cite{AKNS}:

\begin{theorem}[\cite{AKNS}]\label{khelif} Every infinite finitely generated integral domain is bi-interpretable with $\mathbb Z$.
\end{theorem}

The next lemma can be, in fact, extracted from \cite{KM}. Independently, it immediately follows from Theorem \ref{khelif}.

\begin{lemma}\label{bii} $\mathbb F_q[t]$ and $\mathbb F_q[t,t^{-1}]$ are bi-interpretable.
\end{lemma}
\begin{proof} By Theorem \ref{khelif} both rings are bi-interpretable with  $\mathbb Z$. So they are bi-interpretable with each other.
\end{proof}

\begin{corollary}
The groups $G(\Phi,\mathbb F_q[t])$ and $G(\Phi,\mathbb F_q[t,t^{-1}])$, $\rk (\Phi)>1$, are bi-interpretable  with each other and with the rings $\mathbb F_q[t]$ and $\mathbb F_q[t,t^{-1}]$.
\end{corollary}

\begin{proof}

Follows immediately from Theorem 1.1  of \cite{ST}, which states that if $G(\Phi,R)$, $\rk (\Phi)>1$, $R$ is an integral domain,  has finite elementary width,
then $R$ and $G(\Phi,R)$ are bi-interpretable (assuming that for $\Phi=\rE_6$, $\rE_7$, $\rE_8$, $\rF_4$  the order of $R^\ast$ is at least 2).
We use also that $\mathbb F_q[t]$ and $\mathbb F_q[t,t^{-1}]$ are bi-interpretable in view of Lemma  \ref{bii}.
\end{proof}

Recall that given a class of groups $\mathcal C$, a group $G\in \mathcal C$ is {\it first order rigid} if every group $H\in \mathcal C$ which is elementarily equivalent to $G$ is isomorphic to $G$. We take $\mathcal C$ to be the class of finitely generated groups. A group $G\in \mathcal C$  is called {\it finitely axiomatizable} in $\mathcal C$ if the elementary theory  $Th(G)$ is determined by a single formula $\varphi$, that is every group $H\in \mathcal C$ which satisfies $\varphi$ is isomorphic to $G$. If $\mathcal C$ is the class of finitely generated groups, then the property above is used to be called quasi-finite axiomatizability, or QFA-property \cite{Nie}, \cite{OS}.

\begin{corollary}
The groups $G(\Phi,\mathbb F_q[t])$ and $G(\Phi,\mathbb F_q[t,t^{-1}])$, $\rk (\Phi)>1$, are first order rigid and quasi-finitely axiomatizable.
\end{corollary}

\begin{proof}
Follows from Corollary 1.2 from \cite{ST}.
\end{proof}

For the following definitions and facts see \cite{KMS} and \cite{ChKa}. A model $M$ of the theory $T$ is called a {\it prime model} of $T$ if it elementarily embeds in any model of
$T$. A model $M$ of $T$ is atomic if every type realized in $M$ is principal. A model $M$ is
{\it homogeneous} if for every two tuples $\bar a=(a_1,\ldots,a_n)$, $\bar b=(b_1,\ldots,b_n)$ in $ M^n$ that realize the same types in
$M$ there is an automorphism of $M$ that takes $\bar a$ to $\bar b$. It is known that a model $M$ of $T$
is prime if and only if it is countable and atomic. Furthermore, if $M$ is atomic then it is
homogeneous.

The next applications are the consequence of the philosophy of rich groups, i.e., groups
where the first-order logic has the same power as the weak second-order logic. This powerful theory is developed by
Kharlampovich--Myasnikov--Sohrabi \cite{KMS}. The crucial observation regarding rich systems is
the following

\begin{theorem}[\cite{KMS}]\label{ko}
\begin{itemize}
\item[]
\item Any structure bi-interpretable with a rich structure is rich.
\item The structures $\mathbb N$ and $\mathbb Z$ are rich.
\end{itemize}
\end{theorem}

The proof is contained in Theorem 4.7  and Lemma 4.14 of \cite{KMS}.

\begin{theorem}\label{gen} Let $G(\Phi,R)$ be a simply connected Chevalley group, $\rk \Phi >1$, and let $R$ be an infinite finitely generated integral domain. Assume that  $G(\Phi,R)$ is boundedly elementary generated. Assume also  that for $\Phi=\rE_6$, $\rE_7$, $\rE_8$, $\rF_4$  the order of   $R^\ast$ is at least $2$. Then $G(\Phi,R)$ is a rich group.
\end{theorem}
\begin{proof}  By  Theorem 1.1 of \cite{ST}, the ring $R$ and the group $G(\Phi,R)$ are bi-interpretable. By Theorem \ref{khelif}, $R$ and $\mathbb Z$ are bi-interpretable. By Theorem \ref{ko}, $\mathbb Z$ is rich. Hence $G(\Phi,R)$ is also rich by Theorem \ref{ko}.
\end{proof}
%%%%%%%%%%%%%%%%%%%%%%%%%%%%%%%
%%%%%%%%%%%%%%%%%%Zdes`
%%%%%%%%%%%%%%%%%%%%%%%%%%%%%%%%%

\begin{corollary}\label{gen1}
Let $G(\Phi,R)$ be a simply connected Chevalley group. Assume the conditions of Theorem $\ref{gen}$ are fulfilled. Then
\begin{enumerate}

\item The group $G(\Phi,R)$ is quasi-finite axiomatizable.
\item The group $G(\Phi,R)$ is first order rigid.
\item  The group $G(\Phi,R)$ is prime.
\item The group $G(\Phi,R)$ is atomic.
\item The group $G(\Phi,R)$ is homogeneous.
\item Every finitely generated subgroup of $G(\Phi,R)$ is definable.

\end{enumerate}
\end{corollary}
\begin{proof}
\begin{enumerate}
\item[]
\item Corollary 1.3 of \cite{ST}, see also  \cite{KMS}, Section 4.5.2.
\item Corollary 1.3 of \cite{ST}, see also \cite{Nie}.
\item  This is a property of rich groups, see Lemma 4.16 in \cite{KMS}.
\item Follows from the previous item, see \cite{Ho}, \cite{KMS}, Section 4.5.1.
\item See \cite{KMS}, Section 4.5.1.
\item See \cite{KMS}, Theorem 4.11.
\end{enumerate}
\end{proof}

\begin{remark} Theorem 4.11 of \cite{KMS} states that all finitely generated subgroups of $G(\Phi,R)$ are even uniformly definable, see Definition 4.7 of \cite{KMS}.
\end{remark}

All above evidently implies
\begin{corollary}\label{patik1}
The groups $G=G(\Phi,\mathbb F_q[t])$, $\rk   (\Phi)>2$, and $G=G(\Phi,\mathbb F_q[t,t^{-1}])$, $\rk(\Phi)>1$, are QFA,
first order rigid,  prime, atomic, homogeneous. All their finitely generated subgroups are definable.
\end{corollary}

\begin{remark}
Many facts from Corollary \ref{patik1} are known for Chevalley groups $G(\Phi,\mathcal O)$ over different number rings and for various kinds of arithmetic lattices, see \cite{ALM}, \cite{AvMe}, \cite{KMS}, \cite{SM}, \cite{ST}.
\end{remark}

\begin{remark}
For the sake of completeness, we give a straightforward  proof of definability of finitely generated subgroups of $G(\Phi,\mathbb F_q[t])$ and $G(\Phi,\mathbb F_q[t,t^{-1}])$, $\rk (\Phi)>1$, which is parallel to the one of \cite{ALM}.
\end{remark}

\begin{theorem}
All finitely generated subgroups of  $G(\Phi,\mathbb F_q[t])$ and $G(\Phi,\mathbb F_q[t,t^{-1}])$, $\rk (\Phi)>1$, are definable.
\end{theorem}

\begin{proof}
Every finitely generated group is recursively enumerable. So we are interested in recursively enumerable sets over $\mathbb F_q[t]$. But
every recursively enumerable relation over $\mathbb F_q[t]$ is Diophantine over $\mathbb F_q[t]$, see \cite{De}. Hence every finitely generated subgroup $H$ of $G(\Phi,\mathbb F_q[t])$ is definable. Since $\mathbb F_q[t]$ and $\mathbb F_q[t,t^{-1}]$ are bi-interpretable \cite{KMS}, every finitely generated subgroup of $G(\Phi,\mathbb F_q[t,t^{-1}])$ is definable.
\end{proof}

\begin{remark}
In this section, we took a straightforward approach mainly based on combining our results 
on elementary bounded generation with the work of Segal and Tent \cite{ST}. Actually, 
one can go beyond that and obtain far more general results, valid for Chevalley groups over arbitrary commutative rings. 
This would require a thorough revision of the approach taken in \cite{ST} and is postponed to our forthcoming work.  
\end{remark}

%%%%%%%%%%%%%%%%%%%%%%%%%%%%%%

\section{Final remarks} \label{final}

As mentioned in Section \ref{prospects}, there are many fascinating topics related
to bounded generation, some of them well beyond the theory of algebraic groups.
We are not going to discuss them here, referring the interested reader to the introductory
parts of \cite{MRS} and \cite{CRRZ} .
\par
Instead, we mention %%% and restricting our attention to
some [almost] immediate %% (or almost immediate)
eventual generalisations of the results of the present paper,
to which we plan to return in its [expected] sequel.
\par\smallskip
$\bullet$ Firstly, it is a very challenging problem to perform scrupulous analysis of the proofs in Sections \ref{secG2}--\ref{rank3} with an aim to reduce the number of %% needed
elementary moves. We are pretty sure that the obtained
bounds are far from being optimal. Even without attempting
to get sharp bounds, we believe that we could improve the
bounds in the present paper, and other related results.
\par\smallskip
$\bullet$ Secondly, we plan to produce all details for the
stability reduction for the exceptional cases $\rF_4$, $\rE_6$,
$\rE_7$, $\rE_8$ in the same spirit as we have done here for
$\rG_2$ and $\rB_l$. The goal is obtain new explicit bounds
for the elementary width in these cases, which are better
than the known ones even in the number case.
\par\smallskip
Let us mention also several broader projects on which we are
presently working.
\par\smallskip
$\bullet$ One should be able to extend our results to the cases
of twisted Chevalley groups and quasi-split groups, as in
\cite{Ta}. The case of isotropic groups, in the spirit of
\cite{ErRa}, and of generalised unitary groups, also seem
tractable.
\par\smallskip
It is worth noting here that further generalisations in this direction might be problematic. Namely, the
recent results of Pietro Corvaja, Andrei Rapinchuk, Jinbo Ren, and Umberto Zannier \cite{CRRZ} show that
infinite $S$-arithmetic subgroups of absolutely almost simple anisotropic algebraic groups over number fields are never boundedly generated.
The reason is that anisotropic groups do not contain unipotent elements, and a linear group which is not virtually solvable does not
contain enough semisimple elements to guarantee bounded generation
(some quantitative properties which describe the extent of the absence of bounded generation by semi-simple elements
were announced in the subsequent note of the same authors, joint with Julian Demeio \cite{CDRRZ}).
\par\smallskip
$\bullet$ There remains a tempting problem of extending the results of the present paper, in particular Theorems A and C,
to Chevalley groups over rings of integers in more general (or even arbitrary) global function fields (of course, for rank one groups
one has to assume that the group of units of the ring is infinite). It looks like the most challenging part of such an extension
is to generalise the relevant arithmetic ingredients of the proof. Generalised versions of Dirichlet's theorem are readily available (see, e.g.,
\cite[A.12]{BMS}) but this might not suffice for transferring the whole argument to a broader set-up.
Say, Trost's theorem \cite{Tr} on bounded elementary generation of Chevalley groups of rank at least 2 in the function field case required an analogue
of one of arithmetic statements of \cite[Section~3]{Mor}. In a similar vein, an eventual generalisation for groups of rank 1 would perhaps
require a function field counterpart of a subtle fact from additive combinatorics of integers used in Section~5 of \cite{Mor}. An attempt to get
an explicit estimate by generalising Queen's approach in \cite{Qu} looks even more problematic. However, we are moderately optimistic regarding
the treatability of these problems taking into account substantial progress in analytic arithmetic of global function fields that can be observed over the past decades.
\par\smallskip
$\bullet$ Most of the results so far pertain to the {\it absolute\/} case alone. However, it makes sense to ask similar questions for
the {\it relative\/} case, in other words for the congruence subgroups $G(\Phi,R,I)$, and the elementary subgroups $E(\Phi,R,I)$ of level $I\unlhd R$. The expectation is to get
similar {\it uniform\/} bounds in terms of the elementary
conjugates $x_{-\a}(\eta)x_{\a}(\xi)x_{-\a}(-\eta)$,
$\a\in\Phi$, $\xi\in I$, $\eta\in R$. Some results in this
direction are contained in the paper by Sinchuk and
Smolensky \cite{SiSm}. As a more remote goal one could think
of generalisations to birelative subgroups, see \cite{HSVZ}.
\par\smallskip
$\bullet$ Finally, there is a broader area of {\it partial\/}
bounded generation, bounded generation in terms of
other sets of generators, etc. When bounded generation
in terms of $X$ does not hold for the group $G$ itself, one
could ask, whether the width
$$ w_X(Y)=\sup l_X(g),\qquad g\in Y, $$
\noindent
is bounded, for certain subsets $Y\subseteq G$. For instance,
the results by Stepanov and others that we mentioned in \ref{prospects}, imply that $w_E(C)$ is [uniformly] bounded
for the set $C$ of commutators in any Chevalley group
of rank $\ge$ over an {\it arbitrary\/} commutative ring.
Recently, the third author and Raimund Preusser established
partial results in the same spirit for the set of $m$-th powers.
It is natural to expect that some form of this claim holds for arbitrary words, which would (in particular!) infer a negative
answer to the problem of finite verbal width.

\medskip

\noindent{\it Acknowledgements.}  Very special thanks go to
Inna Capdeboscq. This paper started jointly with her as a
discussion  of the bounded commutator width of
various classes of Kac---Moody groups over finite fields, and
for a long time it was supposed to be a work of {\it four\/} authors. Our sincere thanks go to Nikolai Gordeev, Olga Kharlampovich, Jun Morita, Alexei Myasnikov, Denis Osipov, and Igor Zhukov for useful discussions regarding various aspects of this work.

\end{document}